\documentclass[preprint, aos]{imsart}

\usepackage[utf8]{inputenc}
\usepackage[english]{babel}
\usepackage{amsmath, amsthm, amssymb, amscd}
\usepackage{natbib}
\usepackage{graphicx}
\usepackage{grffile}
\usepackage{subcaption}
\usepackage{microtype}
\usepackage{booktabs}
\usepackage{inconsolata}
\usepackage[usenames,dvipsnames]{xcolor}
\usepackage[colorlinks,breaklinks,citecolor=blue,urlcolor=blue,naturalnames]{hyperref}

\graphicspath{{./figures/}}

\startlocaldefs

\numberwithin{equation}{section}
\numberwithin{figure}{section}

\newtheorem{thm}{Theorem}
\newtheorem{lem}{Lemma}

\newcommand{\wh}{\widehat}
\newcommand{\wt}{\widetilde}

\newcommand{\argmin}{\mathop{\rm arg\min}}
\newcommand{\argmax}{\mathop{\rm arg\max}}

\newcommand{\MM}{B}
\newcommand{\pbleach}{p_{\mathrm{b}}}
\newcommand{\data}{\mathbf{X}^k}

\newcommand{\Bin}{\operatorname{Bin}}
\newcommand{\Var}{\operatorname{Var}}

\newcommand{\E}{\mathbb{E}}
\renewcommand{\P}{\mathbb{P}}

\newcommand{\TV}{\operatorname{TV}}
\newcommand{\Beta}{\operatorname{Beta}}

\allowdisplaybreaks 

\endlocaldefs

\usepackage{footmisc}

\begin{document}

\begin{frontmatter}

\title{Posterior analysis of $n$ in the binomial $(n,p)$ problem with both
  parameters unknown - with applications to quantitative nanoscopy}
\runtitle{Posterior analysis in the binomial model}



\author{\fnms{Johannes} \snm{Schmidt-Hieber}\thanksref{a2,x1},
        \ead[label=e2]{a.j.schmidt-hieber@utwente.nl}}
\author{\fnms{Laura Fee} \snm{Schneider}\thanksref{a1, x1},
        \ead[label=e1]{laura-fee.schneider@mathematik.uni-goettingen.de}}
\author{\fnms{Thomas} \snm{Staudt}\thanksref{a1, x1},
        \ead[label=e3]{thomas.staudt@uni-goettingen.de}}
\author{\fnms{Andrea} \snm{Krajina}\thanksref{a1},
        \ead[label=e4]{andrea.krajina@mathematik.uni-goettingen.de}}
\author{\fnms{Timo} \snm{Aspelmeier}\thanksref{a1},
        \ead[label=e5]{timo.aspelmeier@mathematik.uni-goettingen.de}}
\author{\fnms{Axel} \snm{Munk}\thanksref{a1}
        \ead[label=e6]{munk@math.uni-goettingen.de}}

\affiliation{University of G\"ottingen\thanksmark{a1} and University of Twente\thanksmark{a2}}

\thankstext{x1}{These authors have contributed equally to this work}

\runauthor{Schmidt-Hieber et al.}
\end{frontmatter}

\thispagestyle{empty}

\begin{center}
\begin{minipage}{0.80\textwidth}\footnotesize
Estimation of the population size $n$ from $k$ i.i.d.\ binomial
observations with unknown success probability $p$ is relevant to a multitude of
applications and has a long history. Without additional prior information this
is a notoriously difficult task when $p$ becomes small, and the Bayesian
approach becomes particularly useful. For a large class of priors, we establish
posterior contraction and a Bernstein-von Mises type theorem in a setting where
$p\rightarrow0$ and $n\rightarrow\infty$ as $k\to\infty$. 
Furthermore, we suggest a new class of Bayesian estimators for $n$ and provide
a comprehensive simulation study in which we investigate their performance.
To showcase the advantages of a Bayesian approach on real data, we also benchmark
our estimators in a novel application from super-resolution microscopy.
\end{minipage}
\end{center}

\vspace{1ex}

\noindent
\begin{minipage}{\textwidth}\footnotesize
\end{minipage}

\vspace{1ex}

%
%

\section{Introduction}\label{S:Intro} The binomial distribution
with parameters $n$ and $p$ is the most fundamental model for the
repetition of independent success/failure events. Motivated by several important
applications, we focus on the situation where both $p$ and $n$ are unknown. For
example, $n$ might corresponds to the population size of a certain species
\citep{Otis, Royle2004, Raftery}, the number of defective appliances \citep{DG},
or the number of faults in software reliability \citep{Basu}. In
Section~\ref{S:DataExample} we elaborate on a novel application where $n$ is the
number of unknown fluorescent markers in quantitative super-resolution
microscopy \citep{Betzig, Hell, AEM2015}. 

The joint estimation of the population size $n$ and the success probability $p$
of a binomial distribution from $k$ independent observations has a long history
dating back at least to \citet{Fisher}. In comparison to the estimation of one
of the parameters when the other is known \citep{Lehmann}, this problem turns
out to be much harder. 
Fisher, who regarded the assumption of an unknown integer $n$ as ``entirely
academic'', suggested the use of the sample maximum, arguing that this estimator
is necessarily good if the sample size is sufficiently large. Indeed, if
$X_1,\dots, X_k$ are i.i.d.\ $\Bin(n,p)$ distributed random variables for fixed
$n\in\mathbb{N}$ and $p\in (0, 1)$, the sample maximum $M_k := \max_{i=1,\ldots,
k}X_i$ converges exponentially fast to $n$ as $k\rightarrow \infty,$ since
\begin{equation}\label{max_exp_fast}
	\mathbb{P}\big( M_k = n \big)
	= 1-\mathbb{P}\big( \max_{i=1, \ldots, k} X_i <n \big)
	=1-(1-p^{n})^k.
\end{equation}
In practice, however, the regime with small $p$ (``rare events'') is often the
relevant one (see the references below and Section~\ref{S:DataExample}). In this
setting, the sample maximum strongly underestimates the true $n$ even for
large sample sizes $k.$ This is
explicitly quantified in \cite{Gupta}: if $p=0.1$ and $n=10$, then the sample
size $k$ needs to be larger than $3635$ to ensure $\mathbb{P}(M_k \ge n/2) \ge
1/2$. If $p=0.1$ and $n=20$, one would even need a sample size of more than
$k=900\,000$ for the same probability.

The erratic behavior of the sample maximum can be explored by allowing the
parameters $n$ and $p$ to depend on $k$. Applying Bernoulli's inequality and the
bound $1-x \leq e^{-x}$, it follows from \eqref{max_exp_fast} that
$1-e^{-kp^{n}} \leq \mathbb{P}(M_k = n) \leq kp^{n}$. Therefore, the sample
maximum $M_k$ becomes an inconsistent estimator of $n$ if $kp^{n}\rightarrow 0$
as $k\to\infty$ (see Lemma~\ref{Consistency} for a characterization of domains
of consistency and inconsistency of $M_k$).
One particular example where consistency breaks down is the domain of attraction
of the Poisson distribution: when $n \geq \log(k) \rightarrow \infty$ and $p
\rightarrow 0$ such that $np \rightarrow \mu \in (0,\infty)$, then $kp^{n}\leq
k^{1 + \log p}\to 0$. In this case, $\Bin(n,p)$ approaches the Poisson distribution
with intensity parameter $\mu$, leading to non-identifiability of the parameters
$(n,p)$ in the limit. Consequently, more refined estimation techniques become necessary.

Since \citet{Fisher}, a variety of methods have been proposed to improve upon
the sample maximum. A definite answer, however, remains elusive until today. The general
lesson from the attempts to obtain better estimators in the small $p$ regime is
that further information on $n$ and $p$ is required,
which calls for a Bayesian approach. An early Bayesian estimator of the binomial
parameters dates back to \cite{DG}, who suggested to use the posterior mode
under a uniform prior for $n$ (upper bounded by some maximal value), and
a $\Beta(a,b)$ prior for $p$ with $a,b>0$. Later, \cite{Raftery}, \cite{Chilko},
\cite{Hamedani88}, and \cite{Berger}, besides others, considered different
Bayesian estimators. \cite{Raftery}, for example, introduced a hierarchical
Bayes approach that utilizes a Poisson prior on $n$ with intensity parameter
$\lambda > 0$ and a uniform prior distribution on $[0, 1]$ for $p$. Under the choice
$\pi(\lambda) \sim 1/\lambda$ as hyperprior, this hierarchical approach is equivalent to
choosing the (improper) scale prior $1/n$ for $n$. This prior is also recommended as
an objective prior for $n$ in \cite{Berger}. A broader perspective on objective priors
for discrete parameter spaces is offered in \cite{villa2014, villa2015}, who propose a
prior on $n\,|\,p$ that depends on the Kullback-Leibler divergence between two
successive values of $n$. The Villa-Walker construction therefore also models the
dependency between $n$ and $p.$

Besides considering the posterior
mode and posterior median as estimators, \cite{Raftery} suggested to minimize
the Bayes risk with respect to the relative quadratic loss. From extensive
simulation studies (see the aforementioned references and
Section~\ref{S:SimulationStudy} of this article), it is understood that these
Bayesian estimators generally deliver good results, especially when compared to
frequentist approaches. To the best of our knowledge, however, there is no
rigorous theoretical underpinning of these findings. In particular, little is
known about the posterior concentration of such estimators, and no systematic
understanding of the role of the prior has been established. 

Our contribution to this topic is threefold. First (i), we propose a new class
of Bayesian estimators for $n$, generalizing the approach in \cite{Raftery}.
Secondly (ii), we analyze the asymptotic behavior of the posterior distribution
of $n$ for a large class of priors and asymptotic regimes. This includes
statements of posterior consistency as well as a novel Bernstein-von Mises type
theorem.
Finally (iii), we extend the i.i.d.\ $\Bin(n,p)$ model to a regression setting
and apply the suggested estimators to count the number of fluorophores from
super-resolution images. This is a difficult issues of quantitative biology
and a target of ongoing research.

\paragraph*{Ad (i)}
We consider product priors of the form $\Pi_n\otimes \Pi_p$ on $(n,p)$ with
$\Pi_p\sim\Beta(a,b)$ for some
$a,b>0$ and $\Pi_n(n)\propto n^{-\gamma}$ for all $n\in\mathbb{N}$ and
some $\gamma>1$. Independence of $n$ and $p$ in the prior is a natural
assumption and can be justified in our application based on physical
considerations (Section~\ref{S:DataExample}). The beta prior for $p$
is the standard choice and makes the problem analytically tractable
due to its conjugacy property \citep{DG}. The priors $n^{-\gamma}$
for $n$, which we call scale priors with scaling parameter $\gamma$,
are widely studied in the literature on the binomial $(n,p)$ problem
and its variations,
see \citep{Raftery, Berger, Link, MR2301466, MR4122517}.

Based on these prior choices, we focus on two Bayesian scale estimators for $n$.
The first is the posterior mode estimator $\hat{n}_\mathrm{pm}$ and the second
is the Bayes estimator $\hat{n}_{\mathrm{rql}}$ with respect to the relative
quadratic loss, $\ell(x,y)=(x/y-1)^2$. Following \cite{Raftery}, the respective
estimators are given by
\begin{subequations}
\begin{align}
  \hat{n}_\mathrm{pm}
    &= \argmax_{n \ge M_k} \frac{L_{a, b}(n)}{n^\gamma}, \label{eq:scale_pm} \\
  \hat{n}_\mathrm{rql}
    &= \frac{\E \left[\frac{1}{n}|\data\right]}{\E \left[\frac{1}{n^2}|\data\right]}
    = \frac{\sum_{n=M_k}^\infty \frac{1}{n^{1+\gamma}}L_{a,b}(n)}
      {\sum_{n=M_k}^\infty \frac{1}{n^{2+\gamma}}L_{a,b}(n)}, \label{eq:scale_rql}
\end{align}
\label{eq:scale}%
\end{subequations}
where $\data = (X_1, \ldots, X_k)$ denotes the data vector and $L_{a,b}(n)$
is the (data dependent) beta-binomial likelihood defined in
equation~\eqref{E:posteriorN} below (see also \cite{Carroll}). In our
applications (Section~\ref{S:SimulationStudy} and \ref{S:DataExample}),
we assume $\hat{n}_\mathrm{pm}$ and $\hat{n}_\mathrm{rql}$ to be integer
valued by taking the $\argmax$ over $\mathbb{N}$ and by rounding
$\hat{n}_\mathrm{rql}$ to the nearest integer.

\paragraph*{Ad (ii)}
We provide asymptotic conditions under which the marginal posterior for $n$
concentrates all mass around the true population size. As before, we assume
product priors on $(n,p)$ with a beta prior on $p$. For $n$, we allow
general proper priors that decay at most polynomially,
\begin{equation}\label{a}
  \Pi_n(n) \geq \beta n^{-\alpha},
\end{equation}
for all $n\in\mathbb{N}$ and some $\alpha > 1$ and $\beta > 0$.
To investigate the asymptotic behavior of the posterior
distribution, we let $n$ and $p$ depend on the sample size $k$.
We formalize this by considering parameter domains of the form
\begin{align}\label{eq:mathcalM}
  \mathcal{M}_k(\lambda) := \left\{ (n, p)\,\colon\,
    \frac 1{\lambda} \leq np\leq\lambda, \  n \leq \lambda\,
    \frac{\sqrt{k}}{\log^6(k)} \right\},
\end{align}
where $\lambda > 1$ can be chosen arbitrarily. This class describes binomial
variables $\mathrm{Bin}(n, p)$ with expectation values $np$ bounded away from
$0$ and infinity, such that $n$ grows (at most) slightly slower than $\sqrt{k}$.
Under the condition that $\Pi_n$ satisfies \eqref{a}, posterior contraction
around the true population size $n_0$ is studied in Theorem~\ref{result}.
If $n_0$ does not grow faster than $k^{1/4}/\log(k)$, we will see that the
posterior mass eventually concentrates on the true $n_0$. In
Theorem~\ref{thm:bounded_n}, we then extend our analysis to a different
asymptotic domain in which the true population size $n_0$ stays bounded but
$p_0$ is allowed to decay. Lower bounds that we establish in
Theorem~\ref{thm:lower_bound} guarantee that the rates for consistency in
Theorem~\ref{result} and \ref{thm:bounded_n} are indeed sharp up to logarithmic
factors. We also derive a Bernstein-von Mises type result for the posterior on
$n$ in Theorem~\ref{thm.BvM}, which shows that the limit distribution can be
viewed as a discretized version of a normal distribution.

The main building block underlying the recent advances in the frequentist
analysis of posterior concentration are the connection to posterior mass
conditions and the existence of separating statistical test, see \cite{Schwartz,
Ghosal-2000, Ghosal-book}.
To establish model selection properties of the posterior requires typically
different tools \citep{Castillo2012, Castillo2015, Gao}. Since proving that the
posterior concentrates on the true population size can be viewed as posterior
model selection, it is not surprising that we do not follow the standard
posterior contraction proof technique.
In fact, a much more refined analysis of the likelihood is necessary and we
crucially rely on a decomposition of the log-likelihood via a telescoping sum
that is due to \cite{Hall}.
The main challenge in our approach consists of obtaining uniform results over
parameter classes where $n \to \infty$ and $p \to 0$ is allowed (in order to
capture the small $p$ regime).  For fixed $n$ and $p$ as $k\to\infty$, in
contrast, posterior consistency already follows from Doob's consistency theorem,
see \cite{Doob}.

\paragraph*{Ad (iii)}
Modern cell microscopy allows researchers to observe the activity and
interactions of biomolecules in unprecedented detail. Especially since the
development of super-resolution nanoscopy, for which the 2014 Nobel Prize in
Chemistry was awarded, it has become an indispensable tool for understanding the
biochemical function of proteins (see \cite{Hell2015} for a survey).
Super-resolution techniques rely on photon counts obtained from fluorescent
markers (or fluorophores), which are tagged to the specific protein of interest
and excited by a laser beam.
In this article, we are concerned with single marker switching (SMS) microscopy
\citep{Betzig, Rust, Hess, Folling} where the activation of fluorophores and the
emission of photons is inherently random: after excitation by a laser,
a fluorophore undergoes a complicated cycling through (typically unknown)
quantum mechanical states on different time scales. This severely hinders
a precise determination of the number of molecules at a certain spot in the
specimen, see, e.g., \cite{Lee}, \cite{Rollins}, \cite{AEM2015},
\cite{Staudt2020}. In Section~\ref{S:DataExample} we show how the number of
fluorophores can be obtained from a modified binomial $(n,p)$ model. A common
difficulty in such experiments is that the number of active markers decreases
over the measurement process due to bleaching effects.  We show that the initial
number $n_0$ can still be inferred from observations at later time points by
linking them through an exponential decay. This leads to a variant of the
binomial $(n, p)$ model where the bleaching probability of a fluorophore can be
estimated jointly with $n_0$. We apply this model to experimental data and
determine the number of fluorophores on DNA origami test beds.

\paragraph{Outline}
This paper is organized as follows. Our results on posterior contraction and the
Bernstein-von Mises type theorem can be found in
Section~\ref{S:PosteriorContraction}. For a broader perspective, we also discuss
previous results on the asymptotics of several frequentist estimators for $n$.
Section~\ref{S:SimulationStudy} contains an extensive simulation study in which
we examine the posterior of $n$ for moderate to large $k$ and compare the finite
sample properties of several Bayesian and frequentist estimators. Furthermore,
we study the choice of suitable scale priors in different settings and
investigate robustness against model deviations from the Bin$(n,p)$ model.
In Section~\ref{S:DataExample}, we apply our estimators to data from
super-resolution microscopy.
The proof of our main posterior contraction result (Theorem~\ref{result})
and some auxiliary results about binomial random variables are collected in
Section~\ref{S:Proof}. Further proofs as well as additional figures are
deferred to the supplementary material.

\section{Asymptotic results}\label{S:PosteriorContraction}

Recall that we observe $k$ independent  random variables $X_1,\dots,X_k$ with
$\Bin(n,p)$ distribution. We refer to this setting as the binomial $(n,p)$
model.  The joint distribution of the data $\data =(X_1,\ldots, X_k)$ is denoted
by $\P_{n,p}$ and the expectation with respect to this distribution is
$\E_{n,p}$.  We study product priors $\Pi_n \otimes \Pi_p$  on ($n$,$p$) and set
$\Pi_p =\Beta(a,b)$ with parameters $a,b>0$. The prior $\Pi_n$ for $n$ can be
chosen as any proper probability distribution on the positive integers such that
condition \eqref{a} holds for some $\alpha > 1$ and $\beta>0$. We write $M_k
= \max_{i=1, \ldots, k} X_i$ for the sample maximum and $\smash{S_k
= \sum_{i=1}^k X_i}$ for the sample sum. The true parameter values are denoted
by $n_0$ and $p_0$.

For a measurable set $A\subseteq [0,1]$ and $n\in\mathbb{N}$, the joint posterior
distribution for $(p,n)$ is given by
\begin{align*}
  &\Pi\big( p \in A, n \, | \, \data \big)
    = \\
  &\quad\qquad\qquad\qquad\frac{\int_{A}
    t^{S_k +a-1}(1-t)^{kn- S_k+b-1} \mathrm{d}t \cdot 
      \prod_{i=1}^k \binom{n}{X_i}\cdot \Pi_n(n)}{\sum_{m=1}^{\infty}
      \int_0^1 t^{S_k +a-1}(1-t)^{km- S_k+b-1} \mathrm{d}t \cdot
      \prod_{i=1}^k \binom{m}{X_i}\cdot \Pi_n(m)}
\end{align*}
if $n \ge M_k$ and $\Pi( p \in A, n \, | \, \data )=0$ otherwise. The marginal
posterior distribution of $n$ is thus
\begin{equation}\label{E:posteriorN}
  \Pi\big( n \, | \, \data  \big)
    \propto \underbrace{\prod_{i=1}^k \binom{n}{X_i}
    \frac{\Gamma(kn- S_k+b)\,\Gamma( S_k +a)}{\Gamma(kn+a+b)}\,
    \mathbf{1}( n \geq M_k )}_{=:\,L_{a,b}(n)} \Pi_n(n),
\end{equation}
where $\Gamma$ is the Gamma function, $\mathbf{1}$ the indicator function, and
$L_{a,b}$ the beta-binomial likelihood.

\paragraph{Posterior contraction}
Our first result establishes uniform posterior concentration around the true
value $n_0$ over parameters in the set $\mathcal{M}_k(\lambda)$ defined in
equation \eqref{eq:mathcalM}.

\begin{thm}\label{result}
Consider the binomial $(n,p)$ model under the prior mass condition \eqref{a}.
For fixed $\lambda>1$ and $k \to \infty$,
\begin{equation}\label{E:Result_sup}
  \sup_{(n_0,p_0)\in\mathcal{M}_k(\lambda)} \E_{n_0,p_0}\!\left[
  \Pi\left( n: \left|n - n_0\right| \ge  \frac{n_0^2 \, \log^{7/4}(k)}{\sqrt{k}}\
  \bigg|\, \data \right) \right] \rightarrow 0.
\end{equation}
\end{thm}

Equivalently, this result could also be stated in terms of the relative loss
$\ell(n,n_0)=|n/n_0-1|^2$, which is widely studied in the Bayesian
literature for this and related problems, see \cite{Smith1988}.
A noteworthy consequence of Theorem \ref{result} is that the posterior of $n$
eventually places all mass on the true population size $n_0$ if the parameters
$(n_0, p_0)\in \mathcal{M}_k(\lambda)$ additionally satisfy
\begin{equation}\label{eq:posterior-consistency-condition}
    n_0^2 < \frac{\sqrt{k}}{\log^{7/4}(k)}.
\end{equation}

An inspection of the proof of Theorem~\ref{result} reveals that the lower
bound on the prior mass condition \eqref{a} only has to hold for the true
value $n_0$. If we consider sequences of (proper) priors $\Pi_{n,k}$ for $n$ that
can change with the sample size $k$, it can readily be seen from
bound~\eqref{bounding_prior} in the proof that the assertion of the theorem
also holds if $\Pi_{n,k}(n) \geq \beta/(nk)^{\alpha }$ for
all positive integers $n\leq \lambda k^{1/2}$ and some $\alpha, \beta > 0$.
In particular, it holds for priors with restricted support of the form 
\begin{equation}
    \Pi_{n,k}(n) \propto f(n)\,\mathbf{1}(n \leq \lambda k^\alpha),
    \label{eq.seqs_of_priors}
  \end{equation}
where $f$ satisfies $n^{-\alpha/2} \lesssim f(n) \lesssim n^{\alpha/2}$
for some $\alpha\geq 1/2$.

The techniques used to prove Theorem~\ref{result} can also be adapted to
asymptotic regimes where $n_0$ is bounded and $p_0$ converges to $0$ as $k$
tends to infinity. In this case, we depart from the Poisson limit and it should
thus become easier to discern the parameters $n_0$ and $p_0$. Still, if $p_0$
approaches zero quickly with increasing $k$, only a few observations with
positive counts will remain, such that the problem becomes difficult again. The
next result states that posterior consistency holds in this setting as long as
$p_0 \gtrsim \log\,k/\sqrt{k}$.

\begin{thm}\label{thm:bounded_n}
Consider the binomial $(n,p)$ model. For any $B\geq 2,$ define the parameter regime
  \begin{equation*}
    \mathcal{M}_k^\mathrm{b}(B)
    :=
    \Big\{ (n,p) \,\colon\,2\le n \le B, \frac{\log k}{B \sqrt{k}} \le p \Big\}.
  \end{equation*}
  If $\Pi_n(n) > 0$ for all $n\in\mathbb{N}$ with $2\le n \le B$, the posterior
  asymptotically concentrates all mass on the true population size as
  $k\to\infty$, meaning
  \begin{equation*}
  \sup_{(n_0,p_0)\in \mathcal{M}_k^\mathrm{b}(B)} \E_{n_0,p_0}\big[\Pi( n \neq
    n_0 \ |\, \data) \big] \rightarrow 0.
\end{equation*}
\end{thm}

The uniform posterior concentration on the true value $n_0$ that follows for
parameters in the domain $\mathcal{M}_k^\mathrm{b}(B)$ (by
Theorem~\ref{thm:bounded_n}) and for parameters in
$\smash{\mathcal{M}_k(\lambda)}$ that additionally satisfy
\eqref{eq:posterior-consistency-condition} (by Theorem~\ref{result}) also
implies uniform consistency of the respective posterior mode estimators
$\hat{n}_k\in\argmax_{n} \Pi(n\,|\,\data)$.  Indeed, for any subset
$\mathcal{M}_k$ of the mentioned domains,
\begin{equation}\label{eq:pm_consistency}
  \sup_{(n_0, p_0)\in\mathcal{M}_k}\,\mathbb{P}_{n_0,p_0}(\hat{n}_k \neq n_0) \to 0
\end{equation}
as $k\to\infty$. As a special case, this includes the estimator
$\hat{n}_\mathrm{pm}$ introduced in equation~\eqref{eq:scale_pm}.  Furthermore,
if $\mathcal{M}_k$ is such that $n_0$ stays bounded, consistency of the Bayes
estimator $\hat{n}_\mathrm{rql}$ with respect to the relative quadratic loss
given in \eqref{eq:scale_rql} also follows. The same holds for the Bayesian
estimators introduced in \cite{Hamedani88} and \cite{Chilko}.
Since the estimators in \cite{Raftery}, \cite{Berger}, and \cite{Link} are based
on improper priors for $n$, our results can be applied to modifications of these
estimators where $\Pi_n$ is restricted to a bounded support.

We now state a lower bound proving that no uniformly consistent estimator for
$n_0$ exists if $n_0/p_0 \gtrsim \sqrt{k}$.  Combined with
statement~\eqref{eq:pm_consistency}, this implies that posterior contraction on
the true value $n_0$ is impossible in this regime.
\begin{thm}[lower bound]\label{thm:lower_bound}
Let $\eta, \delta>0$ and fix sequences $(n_k)_k \subset \mathbb{N}$ and $(p_k)_k
\subset (0, 1-\delta)$ such that $n_k/p_k \geq \eta\sqrt{k}$ for all $k$. Define
the set $\mathcal{M}^*_k := \{(n_k, p_k), (n_k+1, p'_k)\}$ where $\smash{p'_k
= \frac{n_k}{n_k+1}\,p_k}$. Then there exists a positive constant $c = c(\eta,
\delta)$ such that for any estimator $\hat{n} = \hat{n}(\data)$ and all $k$
  \begin{equation*}
    \max_{(n_0,p_0)\in\mathcal{M}^*_k} \mathbb{P}_{n_0,p_0}(\hat{n}\neq n_0) \ge c.
  \end{equation*}
\end{thm}
If the expectation value $n_0p_0$ is constant or stays bounded away from zero
and infinity, Theorem~\ref{thm:lower_bound} implies that it is impossible to
recover $n_0$ asymptotically when $n_0 \gtrsim k^{1/4}$.  Therefore, the
sufficient condition~\eqref{eq:posterior-consistency-condition} for posterior
consistency in Theorem~\ref{result} is sharp up to logarithmic factors.
Theorem~\ref{thm:lower_bound} also implies that the asymptotic recovery of
a bounded $n_0 \le B$ is only possible if $p_0 \gtrsim 1/\sqrt{k}$, which proves
that the lower bound on $p$ in Theorem~\ref{thm:bounded_n} can at most be
relaxed by a factor of $\log(k)$.  In particular, this implies that product
priors $\Pi = \Pi_n \otimes \Pi_p$ are already asymptotically optimal in the
settings of Theorem 1 and 2 (at least up to log-factors). Modeling dependencies
between $n$ and $p$ via $\Pi$ may hence affect the finite sample performance,
but it will not improve the asymptotic behavior substantially.

To complete the discussion on posterior concentration, it should be mentioned
that another interesting regime occurs if $p_0$ is bounded away from zero and
$n_0 \to \infty$ as $k\to\infty$.  Since the sample maximum grows quickly in
this case, controlling the posterior requires completely different bounds than
before. This regime is of little relevance for our application and we omitted
the mathematical analysis in this work. Note that a numerical study in
\cite{SchneiderStaudt} indicates that posterior consistency holds in this
setting as long as $n_0$ grows slower than $\sqrt{k}$, which coincides with the
lower bound in Theorem~\ref{thm:lower_bound}.

\subsection*{Limiting shape of the posterior}
In the regime where the binomial expectation $n_0p_0$ is bounded away from zero
and infinity, we can characterize the limiting distribution of the posterior in
the Bernstein-von Mises (BvM) sense. For parametric problems, the standard BvM
theorem states, under weak conditions on the prior and the model, that the
posterior converges in total variation distance to a normal distribution
centered at the MLE (see \cite{Doob} for a precise statement). The BvM
phenomenon has been studied in a variety of non-standard settings as well,
including estimation of the probability mass function \cite{MR2471588},
non-regular models \cite{bochkina2014}, and model selection \cite{Castillo2015}.
To the best of our knowledge, BvM theorems for discrete parameters have not been
considered yet. One might wonder in which sense such a limiting shape theorem
can hold, since a discrete distribution can not converge to a continuous
distribution with respect to the total variation distance. 

For the binomial $(n,p)$ problem, we show below that the posterior on $n$
converges in total variation to a discretized version of the normal
distribution. The total variation distance between two discrete distributions
$P$ and $Q$ defined on the integers is $\TV(P,Q)=\tfrac 12 \sum_{i\in
\mathbb{Z}} |P(i)-Q(i)|$, and we say that an integer-valued random variable $X$
has the discrete normal $\mathcal{N}_\mathrm{d}(\mu, \sigma^2)$ distribution if
it satisfies $\P(X=j)\propto \exp\!\big(-\tfrac{1}{2\sigma^2}(j-\mu)^2\big)$ for
all $j\in\mathbb{Z}$.
This distribution is characterized in \cite{kemp1997} as the probability
distribution on the integers with maximal entropy for given expectation and
variance.  Its connection to the Jacobi theta functions and other properties are
analyzed in \cite{szablowski2001}.

Asymptotically, the posterior of $n$ will be centered at the estimator 
\begin{align}
    \hat n 
    := \frac{S_k^2}{S_k^2-k \sum_{i=1}^k X_i(X_i-1)}
    \qquad\text{with} \ S_k=\sum_{i=1}^k X_i.
    \label{eq.whn_def}
\end{align}
In \cite{MR23035}, this estimator is attributed to \cite{Student}, who derived
it by matching the first two moments of the binomial distribution.

\begin{thm}[discrete Bernstein-von Mises]\label{thm.BvM}
Suppose that the parameter $a$ in the $\Beta(a,b)$ prior on $p$ is
a non-negative integer and $\Pi(n)\propto n^{-\alpha}$ for some $\alpha>1$ and
all $n\in\mathbb{N}$. Then, as $k\to\infty$,
\begin{align*}
    \sup_{(n_0,p_0)\in \mathcal{M}_k(\lambda)}
    \E_{n_0,p_0}\bigg[\TV\!\bigg(\Pi\big(n =\cdot\,|\,\data\big),
      \mathcal{N}_\mathrm{d}\Big(\hat n, \frac{2n_0^2}{kp_0^2}\Big)\bigg)\bigg]
    \to 0.
\end{align*}
\end{thm}
 
The proof is rather involved and precise bounds for the likelihood ratio
in a neighborhood of the true $n_0$ are required. The main step is to establish
that the log-likelihood can locally around $n_0$ be written as 
\begin{align}
    \frac 12 \sum_{i=1}^k (X_i)_2
    \log\Big(1-\frac 1n\Big)
    +\frac{S_k^2}{2kn}
    \label{eq.BvM_LR_first_order}
\end{align}
up to terms of negligible order. It can be checked that $n = \hat n$ is
a maximizer of this expression. A second order Taylor expansion of
\eqref{eq.BvM_LR_first_order} around $\hat n$ then shows that the posterior is
close to the limit on a localized set. The full proof is deferred to
Section~\ref{app:BvM} in the supplement. 

Since $p_0$ is of order $1/n_0$ for parameters $(n_0, p_0)$ in the class
$\mathcal{M}_k(\lambda)$, the limit distribution in Theorem~\ref{thm.BvM}
converges to the point mass on $\hat n$ if $n_0 \ll k^{1/4}$. For $n_0 \gg
k^{1/4}$, on the other hand, the limiting variance diverges with $k$. In this
context, we also mention another possibility to define a discretized normal
distribution $Z\sim \mathcal{N}_\mathrm{D}(\mu, \sigma^2)$ on the integers via
$Z:=\argmin_{j \in \mathbb{Z}}|j-X|$ for $X\sim\mathcal{N}(\mu, \sigma^2)$.
The distributions $\mathcal{N}_\mathrm{D}(\mu, \sigma^2)$ and
$\mathcal{N}_\mathrm{d}(\mu, \sigma^2)$ are not the same, but they are close in
total variation distance for large $\sigma$, see Lemma
\ref{lem.TV_two_discretization_bd} in the supplement. If $n_0 \gg k^{1/4}$, this
implies that we can replace the limit distribution $\mathcal{N}_\mathrm{d}$ in
the BvM type result by $\mathcal{N}_\mathrm{D}$.

We conjecture that discretized normal distributions like the ones above will
occur as generic posterior limit distributions for a wide range of discrete
parameter models, such as the ones considered in \cite{MR2963996}.

\subsection*{Asymptotic results for frequentist methods}
\label{S:Frequentist}

For comparison, we briefly summarize existing asymptotic results for frequentist
estimators. Early estimators for $n$ based on the method of moments and the
maximum likelihood approach can be found in \cite{Haldane} and \cite{Blumen}. In
\cite{OPZ}, it is shown that these estimators are highly irregular if $p$ is
small and methods to stabilize them are proposed.  More recently, two further
estimators were introduced by \cite{Gupta}: another modification of the method
of moments estimator, and a bias correction of the sample maximum.
For the new moments estimator,
$\hat{n}_{\mathrm{NME}}$, which depends on the choice of a tuning parameter
$\alpha > 0$, it holds that
\begin{equation*}
  \sqrt{k}(\hat{n}_{\mathrm{NME}}-n)~
    \overset{\mathcal{D}}{\longrightarrow}~\mathcal{N}(0,2\alpha^2n(n-1))
\end{equation*}
as $k\to\infty$, where $n$ and $p$ are both held fixed. To derive this result,
the authors exploit the exponential convergence of the sample maximum to $n$,
which suggests that the limit distribution is only an accurate approximation for
very large values of $k$, especially if $p$ is small. For the bias corrected
sample maximum $\hat{n}_{\mathrm{bias}}$,
\cite{Gupta} derive
\begin{equation*}
  (nk)^{1/(n-1)}(\hat{n}_{\mathrm{bias}}-n)~
    \overset{\mathcal{D}}{\longrightarrow}~\delta_1
\end{equation*}
as $k\to\infty$, where $\delta_1$ denotes the Dirac measure at 1.

The Carroll-Lombard estimator $\hat{n}_{\mathrm{CL}}$ in \cite{Carroll} is the
maximizer of the beta-binomial likelihood in \eqref{E:posteriorN}. It is
therefore the posterior mode estimator under a beta prior on $p$ and an improper
uniform prior on $n$. For $p$ constant, $n\rightarrow\infty$ and
$\sqrt{k}/n\rightarrow0$ as $k\rightarrow\infty$, it is known that
\begin{align*}
  \sqrt{k} \left( \frac{\hat{n}_{\mathrm{CL}}- n}{n} \right)~
    \overset{\mathcal{D}}{\longrightarrow}~
    \mathcal{N}\left(0,\frac{2(1-p)^2}{p^2}\right).
\end{align*} 

All of the results above hold for $p$ fixed and hence provide only limited
insight into the situation when $p$ is small. A notable extension is discussed
in \cite{Hall}. This article studies a variation $\tilde{n}_\mathrm{CL}$ of the
Carroll-Lombard estimator by restricting the search for the maximum of the
beta-binomial likelihood to a suitable neighborhood around the true $n$. Since
this construction depends on the truth, the maximizer $\tilde{n}_{\mathrm{CL}}$
is in a strict sense not an estimator. It is shown that for $n=n_k\to\infty$ and
$p=p_k\to0,$ $np\to\mu\in(0,\infty],$ and $kp^2\to\infty,$ 
\begin{equation}\label{hall-result}
  \frac{p \sqrt{k}}{\sqrt{2}}\left( \frac{\tilde{n}_{\mathrm{CL}}- n}{n} \right)~
  \overset{\mathcal{D}}{\longrightarrow}~\mathcal{N}(0,1)
\end{equation}
as $k\to\infty$. This setup is similar to the one in Theorem~\ref{result}
and~\ref{thm.BvM}, but it does not cover the asymptotic regime considered in
Theorem~\ref{thm:bounded_n}. For the asymptotic normality in
\eqref{hall-result}, it matters that $\tilde{n}_{\mathrm{CL}}$ is regarded as
maximizer over the real numbers and not the integers. To see this, consider
a sequence such that $p\sqrt{k}/n \to \infty.$ As the rate in
\eqref{hall-result} blows up, we must have that $\tilde{n}_{\mathrm{CL}}$
converges to $n$ in probability, which means that if one replaces
$\tilde{n}_{\mathrm{CL}}$ by the closest integer, one recovers the exact value
of $n$ with probability increasing to one as $k \to \infty.$
Also note that result~\eqref{hall-result} is a specific scenario in a broader
context and relies on further technical conditions, like $n$ to be lower bounded
by some positive power of $k$.

\section{Numerical results}\label{S:SimulationStudy}

In this section, we numerically investigate the posterior distribution and the
finite sample performance of Bayesian estimators for different choices of priors
$\Pi_p$ and $\Pi_n$. We consider beta priors with parameters $a, b > 0$ for $p$,
as well as proper and improper scale priors $\Pi_n(n) \sim n^{-\gamma}$ with
$\gamma \geq 0$.
In situations where we assume a prior guess $\tilde p$ for the value of $p$, the
parameters $a$ and $b$ are chosen such that $a\in\{1, 2\}$ and
$b=a/\tilde{p}-a$.  Then, the $\Beta(a,b)$ distribution has expectation $\tilde
p$ and its probability density function is monotone if $a = 1$, while it is
unimodal if $a = 2$. For comparison, we also study the objective prior with $a
= b = 1$, corresponding to a uniform distribution on the probability of success
$p$.

\begin{figure}[tb!]
  {\footnotesize
    \hspace{2.72cm} $k = 10^2$
    \hspace{1.22cm} $k = 10^3$
    \hspace{1.22cm} $k = 10^4$
    \hspace{1.22cm} $k = 10^5$
  }
  \includegraphics[width=\linewidth]{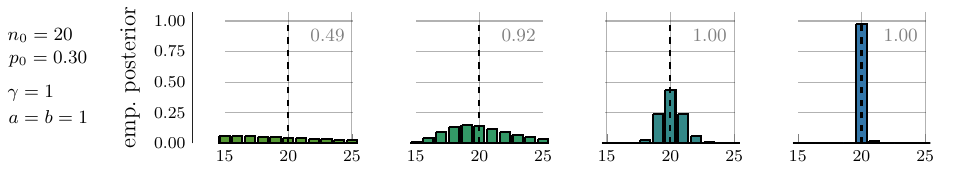}
  \includegraphics[width=\linewidth]{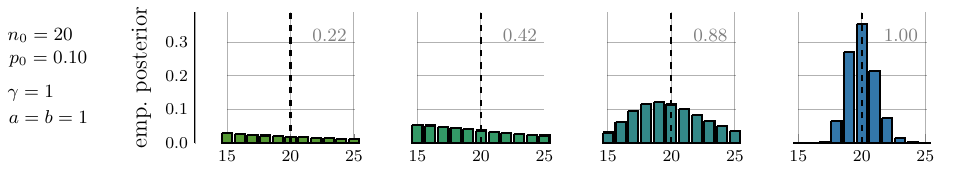}
  \includegraphics[width=\linewidth]{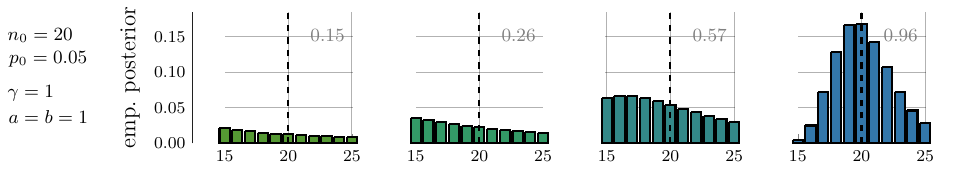}
  \caption{Averaged posterior distributions for true parameters $n_0 = 20,$
    $p_0\in\{0.05, 0.1, 0.3\},$ sample sizes $100 \leq k \leq 10^5$, and
    $\gamma=a=b=1.$ The bar plots display $\E_{n_0, p_0}[\Pi(n\,|\,\data)]$ for
    different values of $n.$ The number in the upper right corner of each graph
    is the expected posterior mass in the interval $[15, 25]$.
  }
  \label{fig:posterior-hist}
\end{figure}

\paragraph{Posterior contraction}
For $\gamma = a = b = 1,$ Figure~\ref{fig:posterior-hist} displays the expected
posterior $n \mapsto \E_{n_0, p_0}\!\big[\Pi(n\,|\,\data)\big]$ for $n_0 = 20$
and $p_0\in\{0.05, 0.1, 0.3\}$ based on $1000$ draws of the data.  Figures for
different parameters can be found in Section~\ref{app:simulations} of the
supplement.  For $n_0 = 20$ and $p_0 = 0.3$, Figure~\ref{fig:posterior-hist}
demonstrates that the posterior distribution visibly contracts to the true value
$n_0 = 20$ for sample sizes $k \ge 10^4$. If $p_0 \le 0.1$, $k = 10^5$ or more
observations become necessary for a comparable effect.
Figure~\ref{fig:app:posterior_50} in the supplement shows that increasing $n_0$
likewise results in broader and less concentrated distributions for given sample
sizes $k$.
Changing $\gamma$, $a$, or $b$ has little effect on the shape of the posterior for
large values of $k$, which is in accordance with the Bernstein-von Mises type result
in Theorem~\ref{thm.BvM}. Still, setting $a = 2$ and $b = 2/p_0-2$ notably
affects the distributions for $k = 100$ and $k = 1000$ by reducing the bias of
the mode, especially when $p_0$ is small (see Figures~\ref{fig:app:posterior_10}
and \ref{fig:app:posterior_20} in the supplement).

It is worth pointing out that the posterior of $n$ behaves considerably better
than the sample maximum $M_k$. For example, if $n_0 = 20$ and $p_0 = 0.3$,
a sample size of at least $k = 10^{10}$ is needed for $\mathbb{P}_{n_0, p_0}(M_k
= 20) \ge 0.35$, while about $10^4$ samples are sufficient for $\E_{n_0,
p_0}\big[\Pi(n = 20\,|\,\data)\big] \ge 0.35$. If $p_0$ is set to $0.1$ in this
comparison, the respective sample sizes are of the dimensions $10^{19}$ versus
$10^5$.

\paragraph{Posterior shape}
In order to examine the validity of the Bernstein-von Mises type result for
finite samples, we compare the posterior of $n$ to the discrete normal
$\mathcal{N}_\mathrm{d}$ distribution predicted as limit in
Theorem~\ref{thm.BvM}. Figure~\ref{fig:bvm-examples} depicts several examples of
the posterior distribution in a setting with $n_0 \sim k^{1/4}$ and $p \sim
1/n_0$, such that the variance parameter $\sigma^2 = 2n_0^2/kp_0^2$ of the
limiting distribution stays (roughly) constant. While the posterior shape
deviates (in part strongly) from the BvM limit for sample sizes $k \le 10^3$, it
clearly approaches the $\mathcal{N}_\mathrm{d}$ distribution as $k$ becomes
larger.  At the same time, the center of the posterior does not seem to
concentrate on the true value $n_0$ as $k$ increases. The posterior often
exhibits a less broad distribution than suggested by Theorem~\ref{thm.BvM}
, especially when the sample maximum $M_k$ reaches into the bulk of the BvM
limit for moderate values of $k$.

Figure~\ref{fig:bvm-convergence} shows the total variation distance between the
posterior and the BvM limit. This time, we consider settings with $n_0\sim
k^{1/4}$ and $n_0\sim k^{1/3}$, which are covered by Theorem~\ref{thm.BvM}, but
also the case $n_0\sim k^{1/2}$, which falls outside of its scope.  One can
clearly see the TV distance decreasing in the former two cases, while it does
not decay if $n_0\sim k^{1/2}$. This indicates that the restriction of $(n_0,
p_0)$ to $\mathcal{M}_k(\lambda)$ in Theorem 4 cannot be relaxed.

\begin{figure}
    \centering
    {\footnotesize\hspace{0.38cm}
    \begin{minipage}{0.22\textwidth}\centering
    $k = 10^2$,\\$n_0 = 9$,~~$p_0 \approx 0.44$
    \end{minipage}\hspace{0.21cm}
    \begin{minipage}{0.22\textwidth}\centering
    $k = 10^3$,\\$n_0 = 16$,~~$p_0 \approx 0.25$
    \end{minipage}\hspace{0.21cm}
    \begin{minipage}{0.22\textwidth}\centering
    $k = 10^4$,\\$n_0 = 30$,~~$p_0 \approx 0.13$
    \end{minipage}\hspace{0.21cm}
    \begin{minipage}{0.22\textwidth}\centering
    $k = 10^5$,\\$n_0 = 53$,~~$p_0 \approx 0.08$
    \end{minipage}
    }\\[1ex]
    \includegraphics[width=\textwidth]{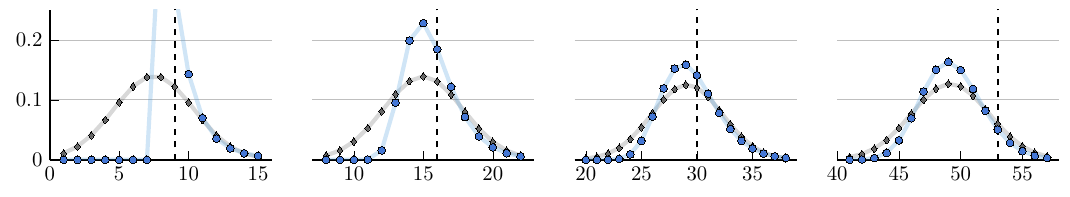}
    \includegraphics[width=\textwidth]{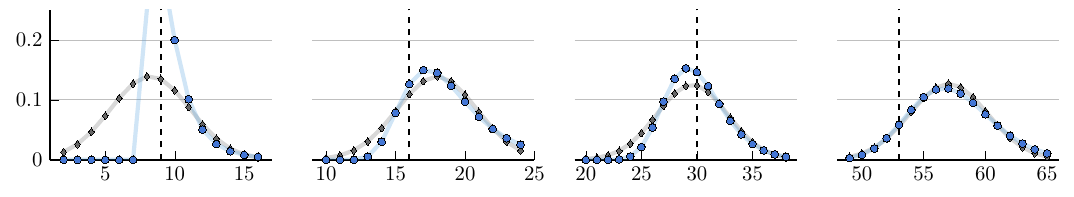}
    \includegraphics[width=\textwidth]{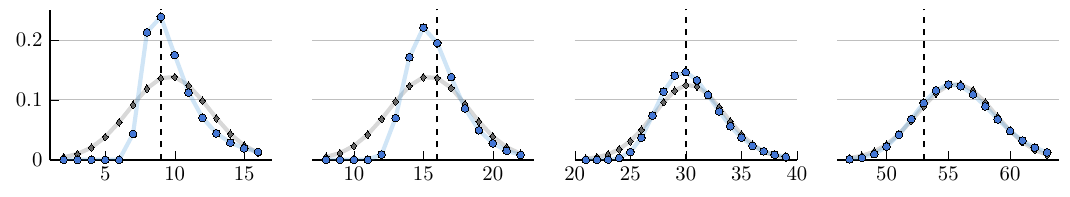}
    \caption{Posterior distribution $\Pi\big(n = \cdot\,|\,\data\big)$ in blue
      and discrete normal distribution $\mathcal{N}_\mathrm{d}\big(\hat{n},
      2n_0^2/kp_0^2\big)$ of Theorem~\ref{thm.BvM} in grey.  The binomial
      parameters are chosen according to $n_0 = \smash{\lfloor 3k^{1/4}\rfloor}$
      and $p_0 = 4/n_0$, where $\lfloor\,\cdot\,\rfloor$ denotes the floor
      function. This results in an (asymptotically) constant value $2n_0^2
      / kp_0^2 \approx 10$. The prior parameters are $\gamma = a = b = 1$.  In
      each graph, independent realizations of $\data$ are used. The dashed lines
      mark the true value $n_0$.
    }
    \label{fig:bvm-examples}
\end{figure}

\begin{figure}
    \centering
    \rotatebox[origin=r]{90}{TV distance\hspace{0.8cm}}
    \hspace{0.1em}
    \begin{minipage}[t]{0.31\textwidth}
      \centering
      \qquad $n_0 \sim k^{1/4}$\\[1ex]
      \includegraphics[width=0.98\textwidth]{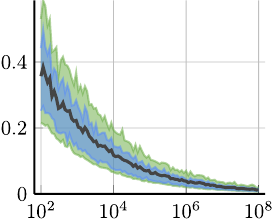}\\
    \end{minipage}
    \begin{minipage}[t]{0.31\textwidth}
      \centering
      \qquad$n_0 \sim k^{1/3}$\\[1ex]
      \includegraphics[width=0.98\textwidth]{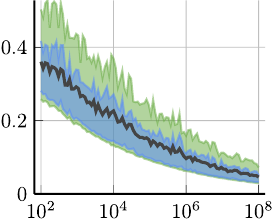}\\
      \qquad sample size $k$
    \end{minipage}
    \begin{minipage}[t]{0.31\textwidth}
      \centering
      \qquad$n_0 \sim k^{1/2}$\\[0ex]
      \includegraphics[width=0.98\textwidth]{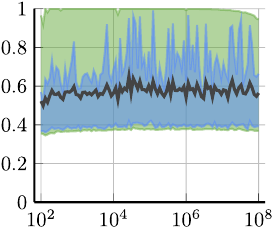}\\
    \end{minipage}
    \caption{Total variation distance between $\Pi\big(n=\cdot\,|\,\data\big)$
      and $\mathcal{N}_\mathrm{d}\big(\hat{n}, 2n_0^2/kp_0^2\big)$ in dependence
      of the sample size $k$ for prior parameters $\gamma = a = b = 1$. The black
      line shows the empirical mean over 100 independent realizations of $\data$,
      while the blue and green areas correspond to the respective 25--75 and 10--90
      percentile ranges. The binomial parameters are chosen as $n_0 = \lfloor
      3 k^{1/\delta} \rfloor$ and $p_0 = 4/n_0$ for $\delta = 4, 3, 2$ (from left to
      right).
    }
    \label{fig:bvm-convergence}
\end{figure}

\paragraph{Estimator performance}
We next study the finite sample performance of a number of Bayesian and
frequentist estimators. In total, the following estimators are considered.

\begin{itemize}
  \item The scale estimator $\hat{n}_\mathrm{rql}$ with respect to the relative
    quadratic loss defined in \eqref{eq:scale_rql}. It depends on the scale
    parameter $\gamma$ and the beta parameters $a$ and $b$, and we refer to it
    by $\mathrm{SE}(\gamma)$.  Note that the posterior distribution for the
    scale prior is well defined as long as $a+\gamma>1$ (see \citet{Kahn} for
    a cautionary note in this context). However, we also report results for
    $\mathrm{SE}(0)$ with $a = 1$, in which case the posterior is no probability
    distribution, but we still obtain finite estimates when evaluating
    \eqref{eq:scale_rql} numerically.  The estimator proposed by \cite{Raftery}
    is equivalent to the scale estimator with the choices $\gamma=1$ and
    $a=b=1$, and is denoted by $\mathrm{RE}$ in the following.
  \item The posterior mode estimator $\hat{n}_\mathrm{pm}$ defined in
    \eqref{eq:scale_pm}.  We refer to it by $\mathrm{PME}(\gamma)$ and assume
    the same prior choices as for $\mathrm{SE}(\gamma)$.  If $\gamma = 0$, it
    coincides with the Carroll-Lombard estimator.  Furthermore, if
    $N_0\in\mathbb{N}$ is chosen sufficiently large, $\mathrm{PME}(0)$ in
    practice also coincides with the estimator proposed by \cite{DG}, which is
    the posterior mode estimator under a beta prior on $p$ and $\Pi_n
    = \mathbf{1}_{\{1,\ldots,N_0\}}$.
  \item The (frequentist) new moment estimator $\mathrm{NME}(\alpha)$ with
    parameter $\alpha$, proposed in \cite{Gupta}. The authors use $\alpha = 1$
    in their numerical work.
  \item The (frequentist) sample maximum $\mathrm{MAX}$.
\end{itemize}

Note that we do not include the maximum likelihood estimator and the moment
estimator \eqref{eq.whn_def} in our comparison, since their finite sample behavior
proved to be very unstable in the range of parameters we consider.

\begin{figure}[tbh!]
  {\scriptsize $a = b = 1$}
  \includegraphics[width=0.97\linewidth]{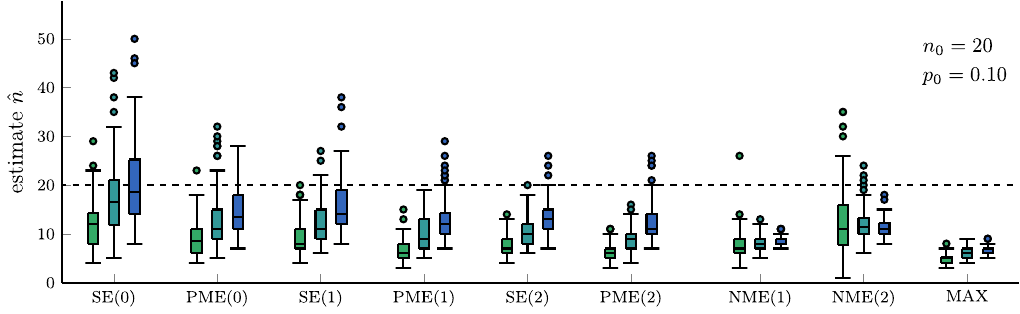}
  \vspace{0.1cm}

  {\scriptsize $a = 2$, $b = a/p_0 - a$}
  \includegraphics[width=0.97\linewidth]{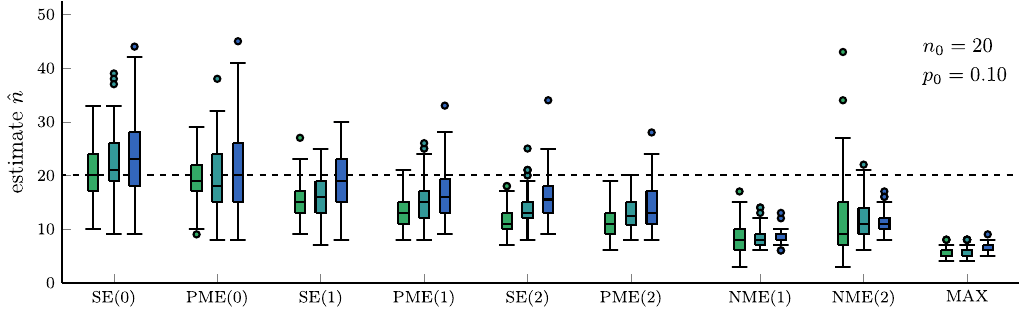}
  \caption{Comparison of estimators for $n$ with underlying parameters
    $(n_0,p_0) = (20,0.1)$ and for three sample sizes $k=30$ (left/green),
    $k=100$ (middle/turquoise), $k=300$ (right/blue). All box plots are based on
    $100$ independent repetitions. Outliers are plotted if they deviate from the
    median by more than 1.5 times the interquartile range.
  }
  \label{fig:comparison}
\end{figure}

Figure~\ref{fig:comparison} summarizes the performance of the proposed
estimators for $n_0 = 20$ and $p_0 = 0.1$ when $k\in\{30, 100, 300\}$.
Further simulation results that cover settings with $n_0\in\{10, 20, 50, 100,
200\}$ and $p_0\in\{0.05, 0.1, 0.3\}$ can be found in
Section~\ref{app:simulations} of the supplement.
We observe several salient tendencies among the Bayesian estimators.
First, the smaller $\gamma$ is, the smaller the bias but the larger the variance
of the estimates becomes. Estimators with $\gamma = 1$ or $2$ typically
underestimate $n_0$, while estimators with $\gamma = 0$ have a larger variability.
Secondly, the bias typically reduces as $k$ is increased from $30$ to $300$. The
variance, on the other hand, only slightly decreases or even increases in some
instances.
Thirdly, the $\mathrm{SE}$ and the $\mathrm{PME}$ perform
similarly for the same $\gamma$, with the former having slightly larger
estimates on average. In particular, we can conclude that the posterior mode does
not suffer from any peculiar instabilities or other drawbacks.
Finally, taking knowledge of $p_0$ into account (by choosing $a = 2$ and $b
= 2/p_0-2$) notably reduces the bias of all Bayesian estimators. As expected, this
effect is most pronounced for small values of $k$.

In comparison, the frequentist estimators typically underestimate $n_0$ more
severely than the Bayesian ones. While the $\mathrm{NME}$ clearly improves
over the sample maximum, it still produces values centered about $n \approx 10$ for
both $\alpha = 1$ and $\alpha = 2$ when $n_0 = 20$. We consistently observed that the variance of the
$\mathrm{NME}$ quickly decreases with increasing $k$ (usually faster than for
the Bayesian estimators), but that its bias barely reduces at the same time.
Indeed, values estimated by the $\mathrm{NME}$ seem to be strongly influenced by
the sample maximum, and it seems to inherit the extremely slow convergence
toward the real value in the setting of moderate to large $k$. Similar issues
were also observed for the bias reduction estimator proposed by \cite{Gupta},
which we did not include in our figures.
We still stress that the $\mathrm{NME}(\alpha)$ with a suitable choice of $\alpha$
is competitive with the Bayesian procedures in some regimes, especially if $k$ is
small and $p$ is moderate (see, e.g., Figure~\ref{fig:app:comparison_noinfo_0.3}
in the supplement).

\paragraph{Prior choices}
In the following, we take a systematic look at the influence of the prior choice
on the performance of the Bayesian estimators in case of small to moderate
sample sizes. Our goal is to establish some practical guidance regarding how to
choose $\gamma$, $a$, and $b$ in different scenarios. 
To this end, we compare the scale estimators $\mathrm{SE}(\gamma)$ with
$\gamma\in\{0, 0.5, 1, 2, 3\}$ and the posterior mode estimator
$\mathrm{PME}(0)$, which corresponds to the Carroll-Lombard estimator,
in several simulations, documenting the parameter constellations that perform best.

In a first study, we consider the settings $k\in\{30,100,300\}$,
$n_0\in\{20,50\}$, and $p_0\in\{0.05, 0.1, 0.3\}$, while assuming a good guess
$\tilde p = p_0$ that correctly informs the $\Beta(a,b)$ prior on $p$
via $a = 2$ and $b = a/\tilde{p} - a$, such that its expectation is $\tilde{p}$.
For all pairs $(n_0,p_0)$ and each estimator $\hat n$, we empirically approximate
\begin{itemize}
  \item the relative mean squared error (RMSE) given by
    $\E_{n_0, p_0}\big[(\hat n/n_0-1)^2\big]$,
  \item the bias $\E_{n_0, p_0}[\hat n] -n_0$ of the estimator,
\end{itemize}
by averaging over $1000$ realizations of $\data$.
In Table~\ref{Table:SummaryAlpha}, we present the estimators that have the
lowest RMSE and the lowest bias for the different choices of $k$. The outcome
generally advises to select smaller values of $\gamma$ the smaller $p_0$ is
expected to be. 
We only found minor differences between the $\mathrm{PME}(0)$ and the $\mathrm{SE}(0)$.
Both of them outperform the other estimators in the regime of very small $p_0$.
The drawback of these estimators is their high variance, which is why larger choices
of $\gamma$ become preferable for low RMSEs as $k$ increases.  The similarity of
Table~\ref{Table:SummaryAlpha20} and \ref{Table:SummaryAlpha50} for $n_0 = 20$ and $n_0 = 50$ suggests
that the influence of $n_0$ is weaker than the one of $p_0$ for the optimal estimator choice.

\begin{table}[!htb]
  \begin{subtable}{.48\linewidth}
    \centering
    \caption{$n_0=20$}\label{Table:SummaryAlpha20}
    \begin{tabular}{rrll}
    $p_0$ & $k$ & RMSE      & bias     \\ \toprule
    0.05  & 30  & PME(0)    & SE($0$)  \\
    0.05  & 100 & PME(0)    & SE($0$)  \\
    0.05  & 300 & SE($0.5$) & PME(0) \\ \midrule
    0.1   & 30  & PME(0)    & SE($0$)  \\
    0.1   & 100 & SE($0.5$) & PME(0) \\
    0.1   & 300 & SE($1$)   & PME(0) \\ \midrule
    0.3   & 30  & SE($2$)   & SE($1$)  \\
    0.3   & 100 & SE($3$)   & PME(0) \\
    0.3   & 300 & SE($3$)   & SE($2$)  \\
    \bottomrule
    \end{tabular}
  \end{subtable}%
  \quad
  \begin{subtable}{.48\linewidth}
    \centering
    \caption{$n_0=50$}\label{Table:SummaryAlpha50}
    \begin{tabular}{rrll}
    $p_0$ & $k$ & RMSE      & bias      \\ \toprule
    0.05  & 30  & PME(0)    & SE($0$)   \\
    0.05  & 100 & PME(0)    & SE($0$)   \\
    0.05  & 300 & SE($0.5$) & PME(0)  \\ \midrule
    0.1   & 30  & PME(0)    & SE($0$)   \\
    0.1   & 100 & SE($0.5$) & PME(0)  \\
    0.1   & 300 & SE($1$)   & SE($0.5$) \\ \midrule
    0.3   & 30  & SE($1$)   & SE($0.5$) \\
    0.3   & 100 & SE($3$)   & PME(0)  \\
    0.3   & 300 & SE($3$)   & PME(0)  \\
    \bottomrule
    \end{tabular}
  \end{subtable}

  \vspace{0.2cm}

  \caption{Overview of the estimators with the smallest RMSE and the smallest
    absolute bias for $a=2$ and $b=2/p_0-2$.}
  \label{Table:SummaryAlpha}
\end{table}

Our next study covers a setting that is motivated by the data example
in Section~\ref{S:DataExample}, and we select $n_0 = 15$, $p_0=0.0339$, and $k=94$. This time, our focus lies
on the influence of the beta prior parameters $a\in\{1, 2\}$ and $b = a / \tilde{p} - a$.
We consider four different scenarios: no information about $p_0$
(setting $\tilde{p} = 0.5$), accurate information ($\tilde{p} = p_0$),
underestimation ($\tilde{p} = 0.5\,p_0$), and overestimation ($\tilde{p}
= 1.5\,p_0$).

The results in Table~\ref{Table:SimsForExample} show that it is advantageous to
choose a small $\gamma$ and a unimodal beta prior (i.e., $a=2$) if a good guess for $p_0$ is
available. If we have no information or are overestimating, it is again advisable to
select $\gamma=0$, while choosing a less confident prior for $p$ with $a = 1$.
In contrast, underestimation of $p_0$ leads to instabilities and
substantial overestimation of $n_0$ if $\gamma$ is small. Here, estimators with
(proper) prior choices $\gamma = 1$ and $2$ perform very well: the tendency of
overestimation caused by the choice $\tilde{p} = 0.5\,p_0$ is in part compensated by the
tendency of underestimation due to the higher value of $\gamma$.

\begin{table}[!htb]
  \begin{subtable}{.48\linewidth}
    \centering
    \begin{tabular}{cllll}
    $\tilde{p}$ & $a$ & est.  & RMSE  & bias    \\ \toprule
    $0.5$     & 1   & SE$(0.5)$  & 0.478 & -10.17  \\
              & 1   & SE$(0)$    & 0.395 & -9      \\ \midrule
    $p_0$     & 2   & PME(0) & 0.034 & -0.266  \\
              & 2   & SE$(0)$    & 0.036 & -0.043  \\
    \bottomrule
    \end{tabular}
  \end{subtable}%
  \quad
  \begin{subtable}{.48\linewidth}
    \centering
    \begin{tabular}{cllll}
    $\tilde{p}$  & $a$ & est. & RMSE  & bias   \\ \toprule
    $1.5\,p_0$ & $1$ & SE$(0)$   & 0.12  & -3.73  \\
               & $2$ & SE$(0)$   & 0.121 & -4.69  \\ \midrule
    $0.5\,p_0$ & $1$ & SE$(1)$   & 0.036 & -0.032 \\
               & $2$ & SE$(2)$   & 0.025 & -0.55  \\
    \bottomrule
    \end{tabular}
  \end{subtable}

  \vspace{0.2cm}

  \caption{The two estimators that perform best under different choices of
    $\tilde p$ for $n_0=15$, $p_0 = 0.0339$, and $k=94$. The respective values of
    $b$ are given by $b = a/\tilde{p} - a$.}
  \label{Table:SimsForExample}
\end{table}

Overall, our findings confirm that the smaller $p_0$, the more difficult it
becomes to estimate $n_0$ and the smaller $\gamma$ should be picked. A smaller
$\gamma$, however, increases the variance of the posterior distribution and
leads to estimators that are potentially more sensitive against
miss-specification in the beta prior. This is further investigated in
Table~\ref{Table:SimsSensitivity}, where we compare the sensitivity of
estimators corresponding to $\gamma=0$ and $\gamma=1$. Miss-specifying $\tilde
p = 0.5\,p_0$ leads to severe overestimates $\E_{n_0, p_0}[\hat{n}] \approx
2\,n_0$ for $\mathrm{PME}(0)$, while $\mathrm{SE}(1)$ is less sensitive in this
regard. Selecting $\gamma=0$ can therefore help to estimate $n_0$ in very
difficult scenarios, but it can also lead to heavily biased results if
$\tilde{p}$ is chosen too small. 

\begin{table}[!htb]
  \centering
  \begin{tabular}{l@{\qquad}cll}
    estimator  & $\wt p$  & RMSE & bias \\ \toprule
               & $p_0$      & 0.122 & -4.85 \\
    SE$(1)$    & $0.5\,p_0$ & 0.129 & 4.43 \\
               & $1.5\,p_0$ & 0.279 & -7.73 \\\midrule
               & $p_0$      & 0.034 & -0.27 \\
    PME$(0)$   & $0.5\,p_0$ & 1.002 & 14.32 \\
               & $1.5\,p_0$ & 0.139 & -5.09 \\
    \bottomrule
  \end{tabular}

  \vspace{0.2cm}
 
  \caption{Sensitivity of $\mathrm{SE}(1)$ and $\mathrm{PME}(0)$ against
    miss-specification of $\tilde p$. The value $a$ is set to 2, all other
    parameters are selected as in Table~\ref{Table:SimsForExample}. The
    behavior of $\mathrm{PME}(0)$ and $\mathrm{SE}(0)$ is comparable in this
    setting.}
  \label{Table:SimsSensitivity}
\end{table}

\paragraph*{Robustness} Motivated by our data example in
Section~\ref{S:DataExample}, we also investigate the situation where $n$
may vary within the sample. This appears to be relevant in many other situations
as well, e.g., in the capture-recapture method, where the (unknown) population
size of a species may change from experiment to experiment. While varying
probabilities $p$ have been investigated in \cite{Basu}, models with a varying
population size $n$ have not received attention in previous research, as far
as we are aware. 

To study this question numerically,
we generated $1000$ data sets $X_i$, $i=1,\dots,k$, with
sample size $k = 100$, where each observation $X_i$ is drawn independently from
a $\Bin(n_i,p_0)$ distribution and each $n_i$ is a realization of
a binomial random variable $N\sim\Bin(\tilde{n},\tilde{p})$. For each sample,
$p_0$ is drawn from a $\Beta(2,38)$ distribution with expectation $0.05$.
To test the influence of the varying parameter $n_i$, we compare the performance
of the estimators in the described scenario to their performance on binomial
samples with constant $n_0$ (chosen as the integer nearest to $\E[N]
= \tilde{n}\tilde{p}$) and the same realization of $p_0$. For both scenarios, we
simulated the RMSE with respect to $n_0$ and record their ratios in
Table~\ref{Table:SimsForRobustness} for parameters $\tilde n$ and $\tilde p$
resembling the data example in Section~\ref{S:DataExample}.
The resulting ratios are all close to one, which suggests a stable performance
of the estimators: estimating $n_0$ from a sample with heterogeneous $n_i$
(randomly drawn from $N$) instead of constant $n_0$ (close to
$\E[N]$) does not affect the RMSE much (on average).


\begin{table}[!htb]
  \centering
  \begin{tabular}{lll}
              & $\tilde{n}=8$ & $\tilde{n}=22 $  \\ \midrule
    estimator & RMSE-R & RMSE-R \\ \toprule
    SE(0.5)   & 1.022  & 1.130  \\ 
    SE(1)     & 1.011  & 1.067  \\ 
    SE(2)     & 1.020  & 1.010  \\ 
    PME(0)    & 1.032  & 1.073  \\ 
    RE        & 0.988  & 0.981  \\ 
   \bottomrule
  \end{tabular} 

  \vspace{0.2cm}
 
  \caption{Ratios of the RMSE for i.i.d.\ and non-i.i.d.\ samples (RMSE-R) for
    the estimators $\mathrm{SE}(\gamma)$, $\mathrm{PME}(0)$, and the Raftery
    estimator $\mathrm{RE}$. The beta prior for $\mathrm{SE}$ and $\mathrm{DGE}$
    uses $a=2$ and $b=38$.}
  \label{Table:SimsForRobustness}
\end{table}

\section{Data example}
\label{S:DataExample}

We now apply the previously described Bayesian estimators
to quantify the number of fluorescent molecules in super-resolution microscopy.
Reliable methods for this task are highly relevant in quantitative cell biology, which
aims to determine the concentration of specific biomolecules, like proteins, in the cell.
For general information, see \citet{Lee}, \cite{Rollins}, \citet{TaH}, \citet{AEM2015},
\citet{Karathanasis}, \citet{Staudt2020}, and references therein. 

\paragraph{Super-resolution microscopy}
The term super-resolution microscopy denotes a family of recently developed
techniques of fluorescence microscopy. It describes the ability to achieve
resolutions below the diffraction limit of visible light (about $250$ --
$500\,\mathrm{nm}$), which limits classical modes of optical microscopy
\citep{Hell}. The central idea is to separate photon emissions of spatially
close fluorescent markers (fluorophores) in time, e.g., by making them switch
between active and inactive states (until they bleach and become permanently
inactive).
In practice, the separation in time is realized by applying an excitation laser with low
intensity, such that only a small fraction of fluorophores in the sample are in the
active state during a given frame of observation.
By combining the resulting ``sparse'' information recorded over a series of
frames, an increased resolution of up to $20$ -- $30\,\mathrm{nm}$ can be
achieved.  See \citet{Betzig}, \citet{Rust}, \citet{Hess}, or \citet{Folling}
for different variants of this principle.


\paragraph*{Experimental setup}
Our data has been recorded at the Laser-Laboratorium G\"{o}ttingen e.V.
In a preparational step, DNA origami molecules \citep{Schmied} were dispersed on
a microscopic cover slip. DNA origami are nucleotide sequences engineered in
such a way that they fold into a desired shape and that fluorophores
can attach to them (see Figure~\ref{fig:origami-a}).
In the experiment, Alexa647 fluorophores with 22 different types of anchors were used,
each matching a different anchor spot on the origami. The attachment process
itself is random and is expected to occur with a probability between 0.6 and
0.75 according to the manufacturer. Hence, about 13 to
17 fluorophores should on average be attached to a single DNA origami.


The experiment was initialized in such a way that most fluorophores occupy
their active state in the first frame. All origami are therefore visible as bright
spots in Figure~\ref{fig:origami-b}. Note that individual fluorophores occupying
the same origami can not be discerned in this image; this becomes possible only
by analyzing later frames where most fluorophores are inactive and
markers show up individually (see the supplementary video). Each
frame had an exposure time of $15\,\mathrm{ms}$, and $14\,060$ consecutive frames were
recorded in total over a time span of about $3.5$ minutes.

\begin{figure}[tbh]
  \centering
  \begin{subfigure}[b]{0.49\linewidth}
    \centering\includegraphics[width=1.0\linewidth]{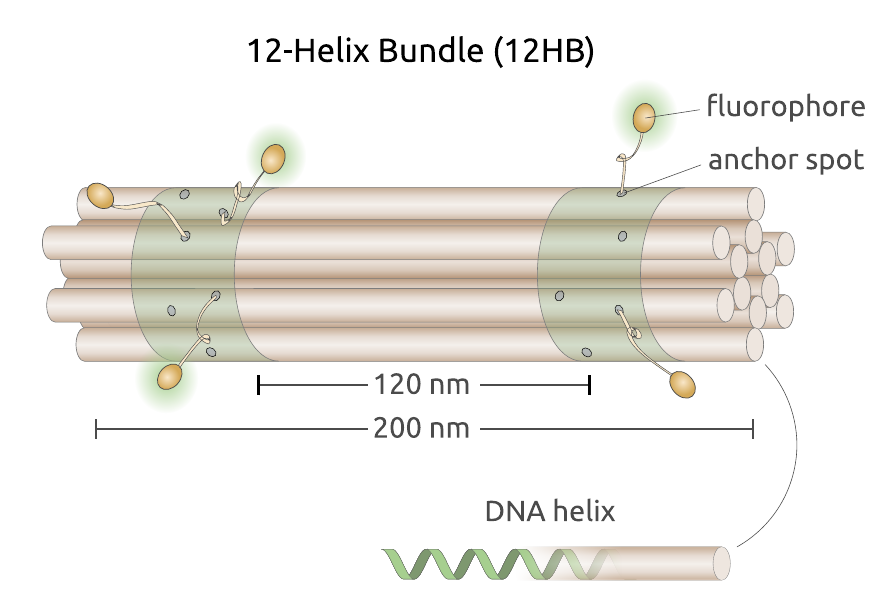}
    \caption{}\label{fig:origami-a}
  \end{subfigure}
  \begin{subfigure}[b]{0.49\linewidth}
    \centering\includegraphics[width=0.95\linewidth]{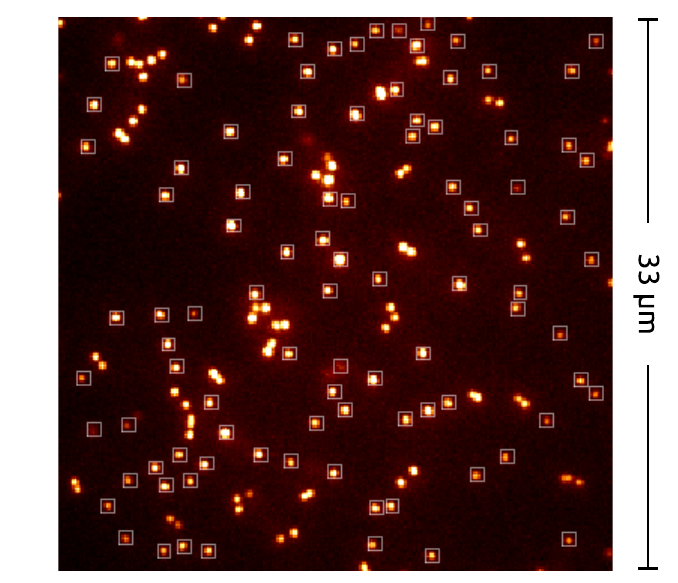}
    \caption{}\label{fig:origami-b}
  \end{subfigure}
  \caption{(a) Schematic drawing of the DNA origami used in the experiment. The
  origami is a tube-like structure that consists of 12 suitably folded DNA
  helices. In each of the two highlighted green regions up to 11 fluorescence
  markers can anchor. (b) First frame of the sequence of microscopic images. The
  94 regions of interest (ROIs) that were chosen for analysis are identified by
  white boxes. No overlap between ROIs
  was allowed, and it was made sure that no excessive background noise and
  disturbances affected the ROI during the course of the experiment.}
  \label{fig:origami}
\end{figure}

\paragraph*{Counting fluorophores} Quantitative biology addresses the issue of
counting the number of fluorophores from measurements like the one described
above. The brightness of each spot is proportional to the number of fluorophores
in the active state within the respective origami. Thus, an origami is invisible if
all of its fluorophores are inactive, but its location on the image is still known from the
first frame. This allows us to register 94 regions of interest (ROIs) marked in
Figure~\ref{fig:origami-b}. For illustration, six microscopic frames recorded at
the times $t\in\{1500, 3000, 4500, 6000, 7500, 9000\}$ are visualized in Figure~\ref{fig:frames}.
The influence of switching and bleaching on the observations is clearly visible.

We aim to estimate the number of fluorophores attached to each origami, which is
expected to be between $13$ and $17$. For simplicity, we assume that
each origami carries the same number $n_0$ of fluorophores and we only model
the mean number $n_t$ of unbleached fluorophores at time $t$. The physical
relation between $n_0$ and $n_t$ is given by
\begin{equation}\label{E:N0_Bleaching}
  n_t = n_0\,(1-\pbleach)^t,
\end{equation}
where $\pbleach$ denotes the bleaching probability.
The brightness observed for a spot in frame $t$ is proportional to the
(random) number $X_t$ of active fluorophores during the frame's exposure.
This number is binomially distributed, $X_t \sim \mathrm{Bin}(n_t, p)$, where $p$
denotes the (time-independent) probability that an unbleached fluorophore is in
its active state. We will estimate $n_0$ and $\pbleach$ by fitting a log-linear
model to equation \eqref{E:N0_Bleaching}, where the respective population sizes
$n_t$ are in turn estimated from the 94 realizations of $X_t$ observed
in frame $t$. 

\begin{figure}[tb!]
  \centering
  \includegraphics[width=0.95\linewidth]{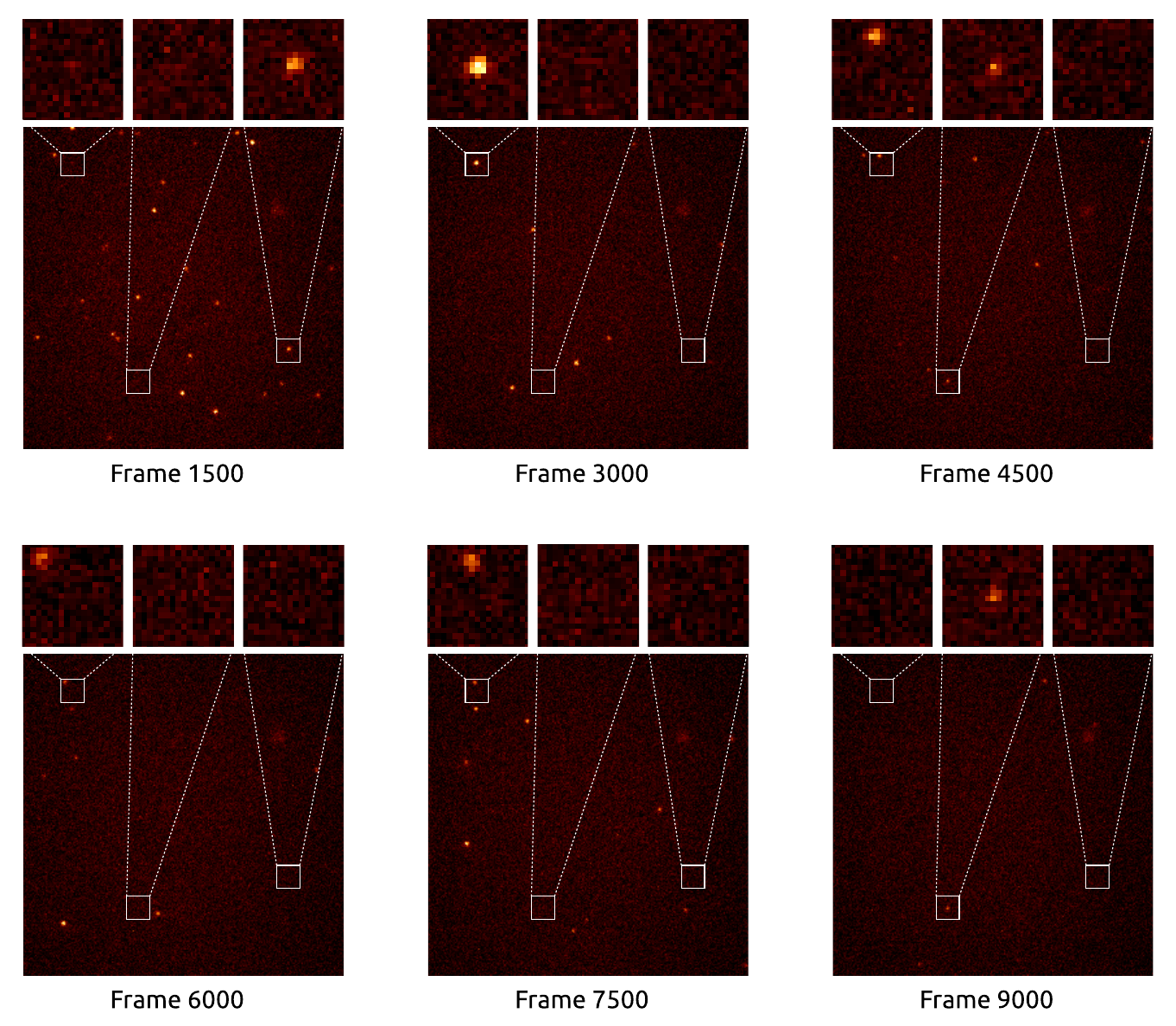}
  \caption{Six selected frames from the data set of recorded origami. The
  (physical) time difference between two consecutive images in this figure is
  roughly 23 seconds. Bleaching causes the number of visible origami to decrease
  with increasing frame number, while switching causes that unbleached origami are
  visible only in some frames.}
  \label{fig:frames}
\end{figure}

To get a sense for the magnitude of $p$,
we use prior information from a similar experiment where each origami has been designed to carry
exactly one fluorophore. We calculate the average ratio between the
number of frames where the fluorophore is active (a bright spot is seen) and the
total number of frames before bleaching, which yields $\tilde{p}\approx 0.0339$
as a prior guess for $p$. Therefore, we are indeed in the difficult small-$p$
regime of the binomial $(n, p)$ problem and will estimate $n_t$ via the Bayesian
scale estimators~\eqref{eq:scale}, using the notation (SE, PME) of
Section~\ref{S:SimulationStudy}. The beta prior for SE and PME uses the parameters
$a = 2$ and $b = 2/\tilde{p} - 2 \approx 56.99$. We choose the unimodal prior with
$a=2$, as suggested by Table~\ref{Table:SimsForExample}, since we assume that our
guess $\tilde{p}$ is reasonably accurate.
Note that a finer degree of modeling would require to view $n_0$, $n_t$
and $p$ as random variables instead of constants.
However, as shown at the end of Section~\ref{S:SimulationStudy}, the Bayesian estimators we
consider are robust against fluctuations in the parameters and are therefore
suited to estimate the respective mean values.

Since most fluorophores are deliberately forced to be active in the first frame,
the relation $X_t \sim \mathrm{Bin}(n_t, p)$ does not hold initially. It
only becomes valid after the initial state has relaxed to an equilibrium, which
is why we only take into account data after frame 1500, about $23$ seconds into the experiment.
To mitigate the influence of correlations between observations (since $X_t$
and $X_{t+1}$ for a spot can not be considered independent), we also add
a waiting time of $1500$ frames between the frames we use for our analysis. In
total, we use the six frames at $t \in \{1500, 3000, 4500,
6000, 7500, 9000\}$ depicted in Figure~\ref{fig:frames}. The 94 realizations of
$X_t$ are extracted from the image data as follows: at each registered
origami position, represented by a $6\times6$ pixel ROI, the total brightness is
measured and then divided by the brightness of a single fluorophore. 
We determined the brightness of a single fluorophore from the late frames of the
experiment, where typically at most one fluorophore of each origami is active.
\begin{table}[!htb]
  \centering
  \begin{tabular}{lcc}
    estimator  & $n_0$   & $\pbleach\cdot 10^3$  \\  \toprule
    SE$(0)$    & \textbf{16} & \textbf{0.152} \\ 
    SE$(0.5)$  & \textbf{13} & \textbf{0.148} \\ 
    SE$(1)$    & 11 & 0.139 \\
    SE$(2)$    & 9  & 0.163 \\
    SE$(3)$    & 6  & 0.123 \\
    SE$(5)$    & 5  & 0.114 \\
    PME$(0)$ & \textbf{16} & \textbf{0.167} \\
    \bottomrule
  \end{tabular} 

  \vspace{0.1cm}
 
  \caption{Estimates of the bleaching probability $\pbleach$ and the number $n_0$
    of fluorophore molecules on single DNA origami.}
  \label{Table:OrigamiEstimates-TimeFrame}
\end{table}

The results for the scale estimator $\mathrm{SE}(0.5)$ are depicted in 
Figure~\ref{Fig:Fit}, which shows the log-linear fit for model
\eqref{E:N0_Bleaching}.
The point estimates of $n_0$ and $\pbleach$ for different estimators are summarized
in Table~\ref{Table:OrigamiEstimates-TimeFrame}.
Given that the true $n_0$ in this experiment is expected to be
between 13 and 17, we can see that the scale estimators with an improper prior
($\gamma\leq1$) produce the most reasonable results. This is in agreement with
our observations in Section~\ref{S:SimulationStudy},
where we noted that priors putting a lot of weight on
large values of $n$ perform better for small $p$ by correcting for the inherent tendency to
underestimate (see Table~\ref{Table:SimsForExample}). To illustrate the
difficulty of this problem, Figure~\ref{Fig:BarChart} shows exemplary counting
results for $t \in \{1500, 7500\}$. Note that estimates
for each $n_t$ are exclusively based on observations $X_t \leq 3$, where
a great majority is even zero.

\begin{figure}[t!]
  \centering
  \subcaptionbox{\label{Fig:Fit}}{
  \includegraphics[width=.58\linewidth]{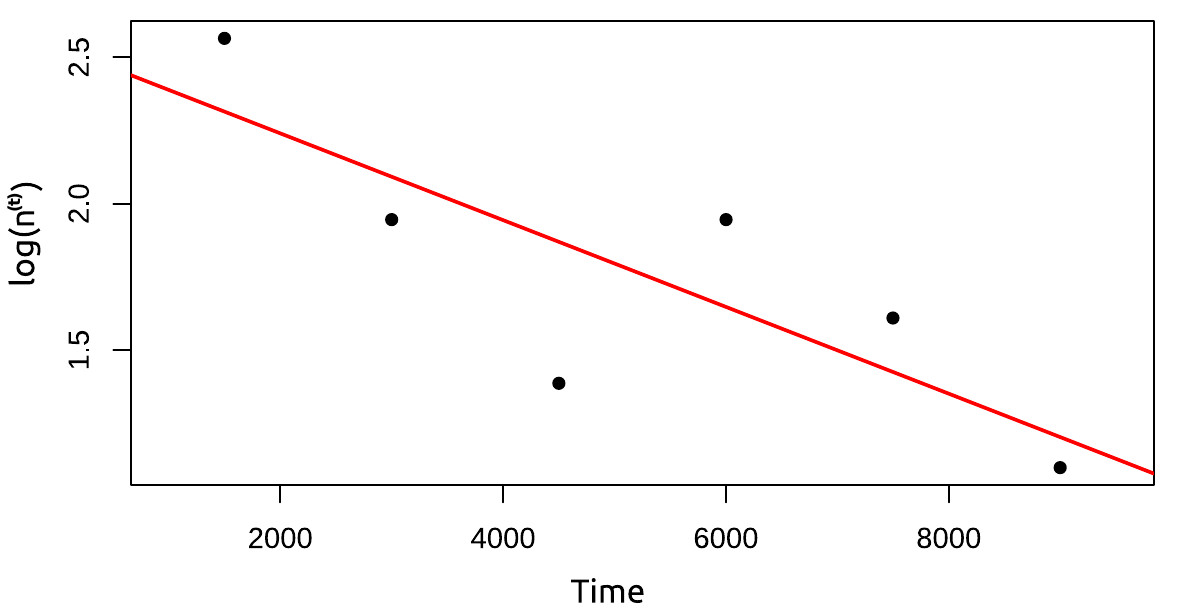}}\hspace{0.5cm}
  \subcaptionbox{\label{Fig:BarChart}}{
  \includegraphics[width=.33\linewidth]{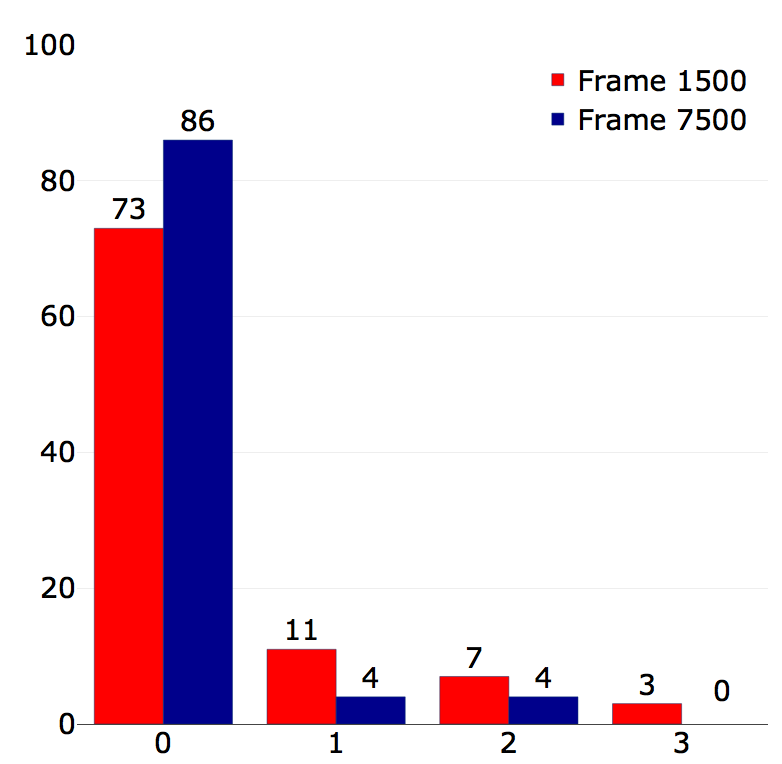}}
  \caption{(a) Log-linear fit described by $n_t = n_0\,(1-\pbleach)^t$ for the SE
    with $\gamma = 0.5$. (b) Bar charts of the observed numbers of fluorophore molecules for time
    frames 1500 and 7500.}
\end{figure}


\section{Proofs and Auxiliary Results}\label{S:Proof}
In the following, we prove the posterior contraction
result of Theorem~\ref{result}, beginning by an outline of the main ideas.

\paragraph{Outline and comments}
Throughout the proof, we fix some $\lambda > 1$ and consider a
generic sequence $(n_k, p_k)_k$ of parameters that satisfies $(n_k, p_k) \in
\mathcal{M}_k(\lambda)$ for all $k\in\mathbb{N}$ with $\mathcal{M}_k(\lambda)$ as defined in \eqref{eq:mathcalM}.
Since the convergence in Theorem~\ref{result} is uniform over $\mathcal{M}_k(\lambda)$,
we emphasize that our arguments are indeed independent of the specific choice
of $(n_k, p_k)_k$ and all bounds are controlled by the
parameter $\lambda$ alone. For brevity, we usually write $\mathbb{P}_k$ and $\E_k$
instead of $\mathbb{P}_{n_k, p_k}$ and $\E_{n_k, p_k}$ from now on.

Let $A_k\subset\mathbb{N}$ be a series of sets that do not contain the true parameter value $n_k$.
The first step of the proof consists of bounding the (marginal) posterior probability $\Pi(n \in A_k\,|\,\data)$
in terms of fractions $L_{a,b}(n)/L_{a,b}(n_k)$ of beta-binomial
likelihoods (defined in \eqref{E:posteriorN}) for integers $n\in A_k$.
Recall that $M_k$ denotes the sample maximum and $S_k =\sum_{i=1}^k X_i$ the sample sum.
Consider the function
$R : [0, \infty)\times(0,\infty) \times [M_k, \infty) \to (0,\infty)$,
\begin{equation}\label{E:helperR}
  R(a, b, m)
  = \prod_{i=1}^k \frac{\Gamma(m+1)}{\Gamma(X_i+1)\,\Gamma(m-X_i+1)}\,
    \frac{\Gamma(km- S_k+b)\,\Gamma( S_k +a)}{\Gamma(km+a+b)},
\end{equation}
which is well defined (even for $a=0$) if $S_k > 0$.
In particular, note that $R(a, b, n) = L_{a, b}(n)$ for
$a,b\in(0, \infty)$ and $n\ge M_k$, so that one can write
\begin{equation}\label{eq.R_ratio}
  \frac{R(a,b,n)}{R(a,b,n_k)}
    = \exp\left(k\,\int_{n_k}^{n}f'(m)\,\mathrm{d}m\right),
\end{equation}
where $f(m) := \frac{1}{k}\log R(a,b,m)$ is differentiable. The derivative
$f'(m)$ is studied in \cite{Hall}.

The remainder of the proof focuses on bounding $f'(m)$. This
includes the definition of an event $\mathcal{X}_k$ that satisfies
$\mathbb{P}(\mathcal{X}_k)\to 1$ for $k\to\infty$. We construct
this event in such a way that $M_k$, $S_k$, and the factorial moments $(X_i)_j$,
where $(c)_j := c\cdot(c-1)\cdots(c-j+1)$ for $c\in\mathbb{R}$ and $j\in\mathbb{N}$,
exhibit benign properties if $\data\in\mathcal{X}_k$. We also need to distinguish
between the cases $m \le n_k$ and $m > n_k$, for which we have to lower-,
respectively upper-bound $f'(m)$ on $\mathcal{X}_k$. This requires several technical
interim steps, which are largely outsourced to Section~\ref{app} in the supplement.
Combining the resulting bounds yields an upper bound for
$\Pi\big(n \in A_k\,\big|\,\data\big)$ that can be used to show
consistency in the asymptotic setting explored in Theorem \ref{result} if
the sets $A_k$ are chosen suitably.

\begin{proof}[Proof of Theorem~\ref{result}] Let $A_k\subset\mathbb{N}$ be sets
such that $n_k\not\in A_k$ for all $k$. It will later become evident how these sets
are best be chosen. First observe that
\begin{equation*}
  \Pi\big(n \in A_k \,|\, \data\big) 
    = \frac{\sum_{n\in A_k, n\ge M_k} \!L_{a,b}(n)
      \,\Pi_n(n)}{\sum_{n=M_k}^\infty L_{a,b}(n) \,\Pi_n(n)}
    \le \!\!\!\sum_{n\in A_k, n \ge M_k}
      \!\frac{L_{a,b}(n)\,\Pi_n(n)}{L_{a,b}(n_k)\,\Pi_n(n_k)}.
\end{equation*}
Under the assumption that $S_k \ge 2$, which we justify below, we can apply
Lemma~\ref{Monotonicity_a} and find
\begin{equation*}
  \frac{L_{a,b}(n)}{L_{a,b}(n_k)} 
  \le \frac{R(\lfloor a\rfloor, b, n)}{R(\lceil a\rceil, b, n_k)}
  \le c_1 \frac{kn_k}{S_k}\,
  \frac{R(\lfloor a\rfloor, b, n)}{R(\lfloor a\rfloor, b, n_k)}
\end{equation*}
for $c_1 = 2\,(1 + \lceil a \rceil + b)$, where $\lceil\cdot \rceil$ and
$\lfloor\cdot\rfloor$ denote the ceiling and floor functions, and where $R$ was
defined in \eqref{E:helperR}. It follows that
\begin{equation}\label{posterior-bound}
  \Pi\big(n\in A_k\,|\,\data\big) 
  \le c_1 \frac{kn_k}{S_k}\,
    \sum_{n\in A_k, n\ge M_k} \exp\left(k\int_{n_k}^{n} f'(m)\,\mathrm{d}m\right)
      \,\frac{\Pi_n(n)}{\Pi_n(n_k)},
\end{equation}
where $f(m) = \frac{1}{k}\log R(\lfloor a\rfloor, b, m)$. In case that $n < n_k$, we find
$\int_{n_k}^{n} f'(m)\,\mathrm{d}m = -\int_{n}^{n_k} f'(m)\,\mathrm{d}m$.
For an upper bound on the posterior we thus need a lower bound of $f'(m)$ if
$M_k \le m\le n_k$ and an upper bound if $m \ge n_k$.
Since $f$ only depends on $a$ via $\lfloor a \rfloor$, we for brevity write
$a \in\mathbb{N}_0$ to denote $\lfloor a \rfloor$ from now on.
Lemma 4.1 in \cite{Hall} states that
\begin{equation*}
  \sum_{j=1}^r \frac{1}{c-j+1} =\sum_{j=1}^r \frac{(r)_j}{(c_j)_j j}
\end{equation*}
for integers $r\in\mathbb{N}$ and positive numbers $c>k-1$. We therefore find
\begin{equation}\label{fprime}
\begin{split}
	f'(m)
    &= \frac{1}{k}\sum_{i=1}^k \sum_{j=1}^{X_i} \frac 1{m-j+1} - \sum_{j=1}^{ S_k +a} \frac 1{km+a+b-j} \\
	&= \sum_{j=1}^{M_k}\frac{ T_j-U_j}{j} -\sum_{j=M_k+1}^{ S_k+a}\frac{U_j}{j}
\end{split}	
\end{equation}
with
\begin{equation}\label{tjuj-def}
  T_j := \frac{1}{k}\sum_{i=1}^k \frac{(X_i)_j}{(m)_j}
  \qquad \text{and}\qquad
  U_j := \frac{( S_k +a)_j}{(km+a+b-1)_j}
\end{equation}
for $j \le M_k$ and $j \le S_k + a$ respectively. If $j > M_k$,
we define $T_j := 0$ for all $j > M_k$. The expectation of $T_j$ is given
by $t_j := \E_k[T_j] = (n_k)_j(p_k)^j/(m)_j$, which follows from
\begin{equation*}
  \E_k[(X_i)_j]
  = \sum_{x=j}^{n_k} \binom{n_k}{x} (x)_j\,p_k^x q_k^{n_k - x}
  = (n_k)_j(p_k)^j\,
  \underbrace{\sum_{y=0}^{n_k-j}\binom{n_k-j}{y}\,p_k^y q_k^{n_k - j - y}}_{=\,1}
\end{equation*}
for all $i= 1, \ldots, k$, where we set $q_k := 1-p_k$ and substituted $y = x - j$.

Next, recall that $\lambda > 1$ is the constant in the definition of the spaces
$\mathcal{M}_k(\lambda).$ For a fixed positive and
diverging sequence $l_k=o\big(\sqrt{\log(k)}\big)$ and $c_2
= 2\,\lambda\,(\lambda + 1)$, we introduce the events
\begin{align}\label{eq.events_proof_MR}
\begin{split}
  \mathcal{R}_k &:= \big\{ \min(n_k,l_k) \leq  M_k \le 2\log(k) \big\}, \\[0.1cm]
  \mathcal{T}_k &:= \bigcap_{j=1}^{M_k} \Big\{ (m)_j\,\big|\,T_j - t_j\big| \le \sqrt{(c_2j)^j\,l_k\log(k)/k}\,\Big\}, \\[0.1cm]
  \mathcal{S}_k &:= \Big\{ \big|S_k - kn_kp_k\big| \le \sqrt{\lambda\,k\log(k)}\Big\},
\end{split}  
\end{align}
and denote their intersection $\,\mathcal{R}_k\, \cap\,
\mathcal{T}_k\, \cap\, \mathcal{S}_k\,$ by $\mathcal{X}_k$. The probability of
the event $\mathcal{T}_k$ is independent of $m$ due to the definition of $T_j$.
On the event $\mathcal{S}_k$, Lemma~\ref{Uj-concentration} grants us the
additional property
\begin{equation*}
  \big|U_j - u_j\big| 
    \le j \,\sqrt{\frac{\lambda \log(k)}{k}} \left(\frac{c_3}{m}\right)^j
  \qquad\text{with}\qquad
  u_j := \frac{(kn_kp_k + a)_j}{(km + a + b - 1)_j}\label{utilde-def}
\end{equation*}
for $j \le S_k + a$ and $c_3 = 2e^2(3\lambda + a + 1)$. 
If $k/\log(k) \ge 4\lambda^3$, then $k/2\lambda \le S_k \le 2\lambda k$
and $S_k \ge 2$ on $\mathcal{S}_k$. Hence, equations \eqref{posterior-bound} and
\eqref{fprime} apply on $\mathcal{X}_k$ if $k$ is sufficiently large. Also, we
can use
\begin{equation}\label{eq:c1bound}
  c_1 kn_k/S_k \le 2\lambda c_1\,n_k
\end{equation}
to bound the factor preceding the sum in \eqref{posterior-bound}.

For the remainder of the proof, we can restrict $\data$ to $\mathcal{X}_k$ since
\begin{equation}\label{restriction}
  \E_k\big[\Pi\big(n\neq n_k\,|\,\data\big)\big]
  - \E_k\big[\mathbf{1}_{\mathcal{X}_k}\Pi\big(n\neq
  n_k\,|\,\data\big)\big] \le \mathbb{P}_k\big(\mathcal{X}_k^\mathrm{c}\big) \longrightarrow 0
\end{equation}
uniformly over $\mathcal{M}_k(\lambda)$ for $k \to \infty$. To show this, we bound
\begin{equation*}
  \mathbb{P}_k\big(\mathcal{X}_k^\mathrm{c}\big) \le
  \mathbb{P}_k\big(\mathcal{S}_k^\mathrm{c}\big)
  + 2\,\mathbb{P}_k\big(\mathcal{R}_k^\mathrm{c}\big)
  + \mathbb{P}_k\big(\mathcal{T}_k^\mathrm{c} \cap \mathcal{R}_k\big).
\end{equation*}
The first contribution vanishes by the application of Chebyshev's inequality
(see, e.g., \cite{degroot2012}), because $\E_k[S_k]=kn_kp_k$ and
$\Var_k[S_k]=kn_kp_k(1-p_k) \leq k \lambda$.
The second term is controlled by Lemma~\ref{Max_cases}. For the last term,
observe that
\begin{equation*}
  \mathrm{Var}[(X_i)_j]
  \le (2j\,np\,(np+1))^j\leq (2j\,\lambda \,(\lambda+1))^j=(c_2 j)^j
\end{equation*}
by Lemma~\ref{moment_bound}. For any $r>0$, Chebyshev's inequality yields
  \begin{equation*}
    \mathbb{P}_k\bigg( \underbrace{\bigg| \frac{1}{k}\sum_{i=1}^k(X_i)_j -
      \E[(X_1)_j] \bigg| > \sqrt{\frac{r \,(c_2 j)^j}{k}}}_{=:\,\mathcal{T}_{jk}^\mathrm{c}(r)}\bigg)
    \le \frac{\mathrm{Var}[(X_1)_j]/k}{r\,(c_2 j)^j/k} \le \frac{1}{r}.
  \end{equation*}
With $r = l_k \log(k)$ and $M_k \le 2\log(k)$ on $\mathcal{R}_k$, 
\begin{align*}
  \mathbb{P}_k\big(\mathcal{T}_k^\mathrm{c}\cap\mathcal{R}_k\big) 
    &= \mathbb{P}_k\left(\bigcup\nolimits_{1\leq j\leq 2\log(k)}
    \mathcal{T}_{jk}^\mathrm{c}\big(l_k \log(k)\big) \right) \le \frac{
    2\log(k)}{l_k\log(k)} \longrightarrow 0
\end{align*}
follows. It is important to note that the upper bounds in these inequalities are
all controlled by $\lambda$, which implies that the convergence in
\eqref{restriction} is indeed uniform over $\mathcal{M}_k(\lambda)$.

\paragraph{Auxiliary lower bound} For $M_k \le m < n_k$, we prove a lower bound
for $f'(m)$. We may assume that $M_k \ge l_k \to \infty$ for $k\to\infty$ in
this case, since $\data\in\mathcal{R}_k$.  For $k$ such that $l_k \ge 4,$
equation \eqref{fprime} yields
\begin{equation}\label{lower_bound_f}
  f'(m) \ge \sum_{j=1}^4 \frac{T_j - U_j}{j} - \sum_{5}^{S_k + a} \frac{U_j}{j},
\end{equation}
as $T_j \ge 0$ for all $j$. Due to the definition of $\mathcal{M}_k(\lambda),$
we can (generously) bound $m < n_k \le \lambda\sqrt{k\log(k)}$ and
\begin{equation*}
  T_1 - U_1 = \frac{S_k}{km} - \frac{S_k + a}{km + a + b - 1} \ge
  - \frac{a+1}{km-1} \ge - 2\,\frac{\lambda(a+1)}{m^2}\sqrt{\frac{\log(k)}{k}}.
\end{equation*}
To handle the terms in \eqref{lower_bound_f} with $j \ge 2$, we exploit that
$\data\in\mathcal{T}_k$ and apply $(m)_j \ge \big(m/e^2\big)^j$ (see
Lemma~\ref{lem.falling_factorial}) in order to derive
\begin{equation*}
  \sum_{j=2}^4 \frac{\big|T_j - t_j\big|}{j} \le
  \sqrt{\frac{l_k\log(k)}{k}}~ \sum_{j=2}^4
  \left(\frac{\sqrt{c_2j}}{m/e^2}\right)^j \le
  2\,\frac{4\,c_2\,e^4}{m^2}~\sqrt{\frac{l_k\log(k)}{k}}
\end{equation*}
for sufficiently large $k$ such that $\sqrt{4\,c_2}\,e^2 / l_k < 1/2$.
Similarly, we find
\begin{equation*}
  \sum_{j=2}^{S_k+a} \frac{\big|U_j - u_j\big|}{j} \le
  \sqrt{\frac{\lambda\log(k)}{k}}~ \sum_{j=2}^{S_k + a}
  \left(\frac{c_3}{m}\right)^j \le
  2\,\frac{\sqrt{\lambda}\,c_3^2}{m^2}~\sqrt{\frac{\log(k)}{k}} 
\end{equation*}
for $m \ge l_k \ge 2 c_3$. By applying Lemma~\ref{Utilde-bound} with $c_4
= 6e^2(\lambda + a)$ and using $S_k \le 2\,k\lambda$ on $\mathcal{S}_k,$ we
furthermore observe
\begin{equation*}
  \sum_{j=5}^{S_k+a} \frac{u_j}{j}  \le \sum_{j=5}^{S_k+a} \frac{1}{j}
  \left(\frac{c_4}{m}\right)^j \le 2\,\left(\frac{c_4}{m}\right)^5 
\end{equation*}
for all $k$ (and thus $m$) that are sufficiently large. The first result of
Lemma~\ref{Utilde-Tj-bound}
combined with $m < n_k \le \lambda \sqrt{k\log(k)}$ reveals
\begin{equation*}
  \sum_{j=2}^4 \frac{t_j - u_j}{j} \ge
  \frac{1}{2\lambda^2} \frac{n_k - m}{n_km^3} - \frac{3\,c_5^4}{mk} \ge
  \frac{1}{2\lambda^2} \frac{n_k - m}{n_km^3} - 2\,\frac{2\,\lambda\,c_5^4}{m^2}
  \sqrt{\frac{\log(k)}{k}},
\end{equation*}
where $c_5 = 3\,\lambda (1+a+b) + 2\,a + 4$. All bounds calculated
above can be inserted into inequality \eqref{lower_bound_f}, yielding
\begin{align}
  f'(m) &\ge (T_1 - U_1) + \sum_{j=2}^4 \frac{T_j - t_j}{j}
    + \sum_{j=2}^{S_k+a} \frac{u_j - U_j}{j} + \sum_{j=2}^4
    \frac{t_j - u_j}{j} - \sum_{j=5}^{S_k+a}
    \frac{u_j}{j} \nonumber \\
  &\ge C_1 \frac{n_k - m}{n_km^3}
    - \frac{C_2}{m^2}\,\sqrt{\frac{l_k\log(k)}{k}}
    + \underbrace{\left[\frac{1}{4\lambda^2}\frac{n_k-m}{n_km^3}
    - 2\,\left(\frac{c_4}{m}\right)^5\right]}_{=:\,h(m)}
    \label{auxiliary_lower_bound}
\end{align}
with $C_1=1/(4\lambda^2)$ and $C_2 = 2\,\big(\lambda(a+1)
+ 4\,c_2\,e^4 + \sqrt{\lambda}\,c_3^2 + 2\,\lambda\,c_5^4\big)$. 

\paragraph{Auxiliary upper bound} We next provide an upper bound for $f'(m)$ for
$m > n_k \ge M_k$. Unlike for the lower bound, we can not assume that $m$ becomes
larger than any given constant with increasing $k$ as $n_k$ could stay bounded.
Since $U_j$ is nonnegative, we can derive
\begin{equation*}
  f'(m) \le \sum_{j=1}^{M_k} \frac{T_j - U_j}{j}
\end{equation*}
from equation \eqref{fprime}. For $j = 1,$
\begin{equation*}
  T_1 - U_1 = \frac{S_k}{km} - \frac{S_k + a}{km + a + b - 1} \le \frac{S_k\,(a
  + b)}{km\,(km - 1)} \le
  \frac{4\lambda\,(a+b)}{m^2} \sqrt{\frac{\log^3(k)}{k}},
\end{equation*}
where we used that $S_k \le 2\lambda\,k$ on the event $\mathcal{S}_k$.
Next we set $m_0 := 4\,c_2\,e^4$ and derive
\begin{align*}
  \sum_{j=2}^{M_k} \frac{\big|T_j - t_j\big|}{j} &\le
  \sqrt{\frac{l_k\log(k)}{k}}~ \sum_{j=2}^{M_k}
  \left(\frac{\sqrt{c_2j}}{m/e^2}\right)^j \\
  &\le \frac{c_2M_k\,e^4}{m^2}~\sqrt{\frac{l_k\log(k)}{k}}~ \sum_{j=0}^{\lfloor
  m\rfloor} \left(\frac{e^2\sqrt{c_2}}{\sqrt{m}}\right)^j \\
  &\le \frac{c_2M_k\,e^4}{m^2}~\sqrt{\frac{l_k\log(k)}{k}} \cdot
  \begin{cases}
    2 &\text{if}~ m > m_0 \\
    m_0 \,(e^2\sqrt{c_2} + 1)^{m_0} &\text{if}~ m \le m_0
  \end{cases} \\
  &\le \frac{c_6}{m^2}~\sqrt{\frac{l_k\log^3(k)}{k}}
\end{align*}
for $c_6 = 2\,c_2\,e^4\big(m_0\,(e^2\sqrt{c_2} + 1)^{m_0}
+ 2\big)$. In the last step, we used that $M_k \le 2\log(k)$ on the event
$\mathcal{R}_k$.
In a similar fashion, we can establish the bound
\begin{equation*}
  \sum_{j=2}^{M_k} \frac{\big|U_j - u_j\big|}{j} \le
  \sqrt{\frac{\lambda\log(k)}{k}}~ \sum_{j=2}^{M_k} \left(\frac{c_3}{m}\right)^j
  \le \frac{c_7}{m^2}~\sqrt{\frac{\log^3(k)}{k}},
\end{equation*}
where $c_7 = 4\,\sqrt{\lambda}\,c_3^2\,\big(c_3^{2c_3+1} + 1\big)$. 
Finally, we apply the second claim of Lemma~\ref{Utilde-Tj-bound} and obtain
\begin{equation*}
  \sum_{j=2}^{M_k} \frac{t_j - u_j}{j} \le - C_1'\,\frac{m - n_k}{n_km^3} + c_8
  \frac{\log(k)}{m^2 k}
\end{equation*}
with $C_1' = \big(\lambda^2\left(a+b+1\right)^2 \big)^{-1}$ and $c_8
= 36(1+\lambda)^3(1+a+b)^2$ for sufficiently large $k$. We conclude
\begin{align}
  f'(m) &\le (T_1 - U_1) + \sum_{j=2}^{M_k} \frac{T_j - t_j}{j}
  + \sum_{j=2}^{M_k} \frac{u_j - U_j}{j}
  + \sum_{j=2}^{M_k}\frac{t_j - u_j}{j} \nonumber \\
  &\le -\frac{C_1'}{m^3}\frac{m - n_k}{n_k}
  + \frac{C_2'}{m^2}~\sqrt{\frac{l_k\,\log^3(k)}{k}}
  \label{auxiliary_upper_bound}
\end{align}
for $C_2' = 4\lambda\,(a+b) + c_6 + c_7 + c_8$.

\paragraph{Posterior bound}
By applying the two inequalities \eqref{auxiliary_lower_bound} and
\eqref{auxiliary_upper_bound} for $m < n_k$ and $m > n_k$, we can now bound the
posterior probability $\Pi\big(n\in A_k \,|\,\data\big)$ on the event
$\mathcal{X}_k$ through equation \eqref{posterior-bound}. Recall that we can
assume $n\neq n_k$ due to the assumption $n_k\not\in A_k$. We observe that
\begin{equation*}
  \int_{n_k}^{n} \frac{m - n_k}{n_km^3}\,\mathrm{d}m = \frac{1}{2}
  \frac{(n-n_k)^2}{(n_kn)^2} \quad~~\text{and}\quad~~\int_{n_k}^{n}
  \frac{1}{m^2}\,\mathrm{d}m = \frac{1}{2}\frac{(n-n_k)^2}{(n_kn)^2}
  \frac{2n_kn}{n - n_k}.
\end{equation*}
Noting $|n - n_k| \ge 1$, it also holds that
\begin{equation}\label{n_over_nk_ineq}
  \left|\frac{n}{n_k} - 1\right| = \frac{|n - n_k|}{n_k} \ge \left.\begin{cases}
    1/n_k &\text{if}~n_k\le 2n \\
    1/2 &\text{if}~n_k > 2n \\
  \end{cases}\right\}
  \ge \frac{1}{2n}\,.
\end{equation}
Therefore, if $l_k \le n < n_k$, the function $h(m)$ introduced in equation
\eqref{auxiliary_lower_bound} satisfies
\begin{align}\label{hm_integral}
  \int_n^{n_k}h(m) \,\mathrm{d}m &= \frac{C_1}{2}\,\frac{(n-n_k)^2}{(n_kn)^2}
  - \frac{c_4^5}{2}\,\frac{n_k^4 - n^4}{(n_kn)^4} \nonumber \\
  &\ge \frac{C_1}{2} \,\frac{(n-n_k)^2}{(n_kn)^2}\left(1
  - \frac{4c_4^5}{C_1}\frac{1}{1-n/n_k}\frac{1}{n^2}\right) \ge 0
\end{align}
for $k$ such that $l_k \ge 8\,c_4^5/C_1$.
Employing bound \eqref{auxiliary_lower_bound} thus yields
\begin{align*}
  -k \int_n^{n_k} f'(m)\,\mathrm{d}m &\le -k \frac{C_1}{2}
  \frac{(n-n_k)^2}{(n_kn)^2} \left(
  1 - C \frac{n_kn}{n_k-n}\sqrt{\frac{l_k\log(k)}{k}}\right),
  \intertext{where the constant $C$ is given by $2\,C_2/C_1$.
  On the other hand, for $n_k < n$, bound \eqref{auxiliary_upper_bound} similarly
  leads to}
  k \int_{n_k}^n f'(m)\,\mathrm{d}m &\le -k \frac{C_1'}{2}
  \frac{(n-n_k)^2}{(n_kn)^2} \left( 1 - C'
  \frac{n_kn}{n-n_k}\sqrt{\frac{l_k\log^3(k)}{k}}\right)
\end{align*}
for $C' = 2\,C_2'/C_1'$. 
Finally, let $\widetilde{C}_1 = \min\!\big\{C_1, C_1'\big\}$ and $\widetilde{C}
= \max\!\big\{C, C'\big\}$.  Combining the two inequalities for $n_k < n$ and
$n_k > n$ results in
\begin{equation*}
  k\int_{n_k}^n f'(m)\,\mathrm{d}m
  \le -k \frac{\widetilde{C}_1}{2\,n_k^2}\left(\frac{n_k}{n} - 1\right)^2
    \left( 1 - \frac{\widetilde{C}\,n_k}{|1 - n_k/n|} \sqrt{\frac{l_k\log^3(k)}{k}} \right)
\end{equation*}
for all $n \neq n_k$ with $n \ge M_k$. In order to bound $\Pi\big(n\in
A_k\,|\,\data\big)$ via \eqref{posterior-bound}, we need that the second factor
in this expression is positive for large $k$. Since $l_k \log^3(k)
= o\big(\log^{7/2}(k)\big)$, this motivates the choice
\begin{equation*}
    A_k := \left\{n\in\mathbb{N} \,\bigg|\, |n_k - n| \ge n n_k
    \frac{\log^{\rho}(k)}{2\sqrt{k}}\right\}\qquad\text{with}\qquad\rho = 7/4.
\end{equation*}
For $n\in A_k$ and $k$ large enough, we thus find
\begin{equation*}
  k\int_{n_k}^n f'(m)\,\mathrm{d}m
  \le -k \frac{\widetilde{C}_1}{4\,n_k^2}\left(\frac{n_k}{n} - 1\right)^2
\end{equation*}
Applying the inequalities \eqref{posterior-bound} and \eqref{eq:c1bound}
combined with the constraint $\Pi_n(n) \ge \beta n^{-\alpha}$ for all
$n\in\mathbb{N}$ on the (proper) prior yields
\begin{align}\label{bounding_prior}
  \mathbf{1}_{\mathcal{X}_k}\Pi\big(n\in A_k\,|\,\data\big) 
  &\le 2\lambda c_1\, n_k \sum_{n\in A_k}
    \exp\!\left(-\frac{\widetilde{C}_1}{4}\frac{k}{n_k^2}\left(\frac{n_k}{n}-1\right)^2\right)
    \frac{\Pi_n(n)}{\Pi_n(n_k)} \\
  &\le \frac{2\lambda
    c_1}{\beta}\exp\!\left(-\frac{\widetilde{C}_1}{2}\log^{2\rho}(k) + (\alpha
    + 1) \log(n_k)\right) \longrightarrow 0 \nonumber
\end{align}
as $k\to\infty$ uniformly over $\mathcal{M}_k(\lambda)$. Due to
\eqref{restriction}, we have therefore established uniform convergence of
$\E_k\big[\Pi\big(n \in A_k\,|\,\data\big)\big]$ to $0$.  To bring this result
in the form of Theorem~\ref{result}, we just have to note that
\begin{equation*}
    A_k^\mathrm{c} \subset \left\{n\in\mathbb{N} \,\bigg|\, |n_k - n| \le n_k^2\,\frac{\log^{\rho}(k)}{\sqrt{k}}\right\}
\end{equation*}
whenever $k$ is large enough such that $n_k\log^\rho(k)/\sqrt{k} < 1/2$.
\end{proof}

\subsection*{Auxiliary results for binomial random variables}

We begin with a bound on the variance of falling factorials of binomial random
variables.

\begin{lem}\label{moment_bound}
Let $X\sim\Bin(n,p)$ and $j\in\{0,\dots,n\}$. Then
\begin{equation*}
  \mathrm{Var}_{n,p}\big[(X)_j^2\big]
  \le \E_{n,p}\big[(X)_j^2\big]
  \leq \big(2j\, np\, \max(np,1)\big)^j.
\end{equation*}
\end{lem}

\begin{proof}
We have $\E_{n,p}[(X)_j^2] = \partial_s^j\partial_t^j\,\E_{n,p}[(st)^X]$
evaluated at $s=t=1$. Since $X\sim\Bin(n, p)$, it holds that
\begin{equation*}
\E_{n,p}\big[(st)^X\big]=(q+pst)^n
\quad\text{and}\quad
\partial_t^j\,\E_{n,p}\big[(st)^X\big]=(n)_j (ps)^j(q+pst)^{n-j}
\end{equation*}
for $q = 1 - p$. By the general Leibniz rule for derivatives of products,
\begin{align*}
    \partial_s^j\partial_t^j\,\E_{n,p}\big[(st)^X\big]
    = (n)_j p^j \sum_{r=0}^j \binom{j}{r} (j)_r s^{j-r}
    (n-j)_{j-r}(pt)^{j-r}(q+pst)^{n-2j+r}\!.
\end{align*}
Setting $s=t=1$ implies $q+pst=1$, such that the last equation becomes
\begin{equation*}
  \E_{n,p}\big[(X)^2_j\big] = (n)_j p^j \sum_{r=0}^j \binom{j}{r} (j)_r (n-j)_{j-r}p^{j-r}.
\end{equation*}
The claim of the lemma now follows by bounding $(n)_j\leq n^j,$ $(j)_r\leq j^j$,
and $(n-j)_{j-r}p^{j-r}\leq (np)^{j-r}\leq \max(np,1)^j$, and using that
$\sum_{r=0}^j \binom{j}{r} = 2^j.$
\end{proof}

The next result characterizes the growth of the sample maximum $M_k$ in the
context of Theorem~\ref{result}.

\begin{lem}\label{Max_cases}
  Let $(l_k)_{k\in\mathbb{N}}$ be such that $l_k\rightarrow\infty$ and $l_k^2
  = o\big(\log(k)\big)$. Then
  \begin{equation*}
    \sup_{(n_k,p_k) \in \mathcal{M}_k(\lambda)}\mathbb{P}_k\big(\min\{l_k,
    n_k\}\leq M_k \leq 2\log(k) \big) \to 1 \quad \text{as} \ k \to \infty.
  \end{equation*}
\end{lem}

\begin{proof}
In the first part of the proof we show that $M_k< \min\{l_k, n_k\}$ has
(uniformly) vanishing probability for $k \to \infty$.  If $l_k \ge n_k$,
applying Bernoulli's inequality and using $l_k^2 = o(\log(k))$ yields
\begin{equation*}
  \log\mathbb{P}_{n_k,p_k}(M_k < n_k) \le -kp_k^{n_k} \le -k\,e^{-
  l_k\log(l_k/n_k p_k)}\le -k\,e^{- l_k\log(\lambda l_k)}  \to  -\infty
\end{equation*}
uniformly over $\mathcal{M}_k(\lambda)$ as $k\to\infty$. If $n_k > l_k$, we find
$p_k \le \lambda/n_k \leq \lambda/l_k \leq 1/4$ for all sufficiently large $k$,
and thus Slud's bound from \cite{Tele} can be applied. We find
\begin{align*}
  \mathbb{P}(M_k < l_k) &= \big(1 - \mathbb{P}(X_1 \ge l_k) \big)^k
  \le \Phi\!\left(\frac{l_k}{\sqrt{n_kp_k(1-p_k)}}\right)^k \le
  \Phi\!\left(\frac{\sqrt{2}\,l_k}{\sqrt{n_kp_k}\,}\right)^k,
\end{align*}
where $\Phi$ denotes the cumulative density function of the standard normal
distribution.  \cite{Gordon1941} derives the lower tail bound $\smash{1
- \Phi(t) \ge \frac{1}{2\pi} \frac{t}{t^2+1}\,e^{-t^2/2}}$ for $t > 0$. We apply
this bound for $t = \sqrt{2} \,l_k / \sqrt{n_kp_k}$ with $t > 1$ if $k$ is
sufficiently large.  Combined with the elementary inequality $t/(t^2+1) \ge
1/2t^2$ for $t \ge 1$ and $l_k^2 = o\big(\log(k)\big)$, we conclude
\begin{align*}
  \Phi\left(\frac{\sqrt{2}\,l_k}{\sqrt{n_kp_k}\,}\right)^k &\le \left(1
  - \frac{n_kp_k}{8\pi\,l_k^2}\,e^{-\frac{l_k^2}{n_kp_k}}\right)^k
  \le  \exp\left(-\frac{k}{8\pi}\,\frac{n_kp_k}{l_k^2\,e^{l_k^2/n_kp_k}}\right)\to 0
\end{align*}
as $k \to \infty$. The convergence is uniform over $(n_k,p_k)_k \in
\mathcal{M}_k(\lambda)$.
  
It remains to show that the probability of $M_k\leq  2\log(k)$ uniformly
converges to one for $k \to \infty$.  To see this, write $X_1$ as a sum of $n_k$
i.i.d.\ Bernoulli random variables with success probability $p_k.$ Since
Hoeffding's inequality is not precise for small $p_k$, we apply Bernstein's
inequality (see, e.g., \cite{vdvaart1996}), and conclude
\begin{align*}
  \mathbb{P}\left( M_k \leq 2\log(k)  \right) 
  &= \big( 1-\mathbb{P}\big( X_1-n_kp_k > 2\log(k) - n_kp_k \big)\big)^k\\
  &\geq \left(1 - \exp\!\left(-\frac{ \big(2\log(k)
    - n_kp_k\big)^2}{2\big(n_kp_k(1-p_k) + \log(k)/3\big)} \right)\right)^k\\
  &\geq \big(1 - e^{-2\log(k)}\big)^k \to 1,
\end{align*}
where the second inequality holds for $\log(k) \ge 3\lambda \geq 3n_kp_k$.
\end{proof}

\paragraph{Consistency of the sample maximum}
The following lemma examines the consistency of the sample maximum in the
binomial $(n,p)$ problem if $np \to \mu > 0$ with increasing sample size $k$.

\begin{lem}\label{Consistency} 
Let $0< c<1<C$ and let $M_k:= \max_{i=1,\dots,k}X_i$ be the sample maximum for
independent random variables $X_1,\dots,X_k \sim\Bin(n_k,p_k)$ such that
$n_kp_k\to\mu > 0$. Then, as $k \to \infty$,
\begin{align*}
    \mathbb{P}(M_k=n_k) \to 
    \begin{cases}
    1  &\text{if} \ \  n_k\log(n_k)<c \log(k), \\
    0  &\text{if} \ \  n_k\log(n_k)> C\log(k).
    \end{cases}
\end{align*}
\end{lem}

\begin{proof}

  As mentioned in the introduction, \eqref{max_exp_fast} implies
  $1-e^{-kp_k^{n_k}}\le \mathbb{P}(M_k=n_k) \leq kp_k^{n_k}.$ To show that
  $\mathbb{P}(M_k=n_k)\to 1,$ it is sufficient to prove that $kp_k^{n_k} \to
  \infty.$
  This holds if $\log(k) - n_k\log(n_k/n_kp_k) \to \infty$, which in turn follows
  from
  \begin{equation*}
    \frac{n_k\log(n_k)}{\log(k)} < c < 1\qquad \text{and}\qquad \frac{n_k\,|\log(n_kp_k)|}{\log(k)} \le \frac{c\,|\log(n_kp_k) |}{~\log\big(c\log(k)\big)}.
  \end{equation*}
  To show convergence to zero, we use $P(M_k = n_k) \le kp_k^{n_k} \le
  \exp\big(\log(k) - n_k\log(n_k/n_kp_k)\big)$.  Similar to the argument above,
  the right hand side in this inequality converges to $0$ since
  \begin{equation*}
    \frac{n_k\log(n_k)}{\log(k)} > C > 1\qquad \text{and}\qquad \frac{n_k\,|\log(n_kp_k)|}{\log(k)} \le \frac{C\,|\log(n_kp_k) |}{~\log\big(C\log(k)\big)}.
  \end{equation*}

\end{proof}

\section*{Acknowledgements}
We would like to thank the reviewers and are particularly grateful to one
referee for a detailed report with additional insights and hints to the
literature. These comments have lead to a substantial improvement of the
article.
Support of DFG CRC 755 (A6), Cluster of Excellence MBExC, and DFG RTN 2088 (B4)
is gratefully acknowledged. JSH was supported by a TOP II grant from the NWO. We
also thank Oskar Laitenberger for providing us with data recorded at the
Laser-Laboratorium Göttingen e.V.

\section*{Supplementary video: fluorescence microscopy}
Video of the first 9000 frames of the data used for estimating the fluorophore
number in Section~\ref{S:DataExample}
(\href{http://www.stochastik.math.uni-goettingen.de/SMS-movie.mp4}{http://www.stochastik.math.uni-goettingen.de/SMS-movie.mp4}).


\bibliographystyle{chicagoa}
\bibliography{binomial}

%
%
\clearpage
\begin{center}\bfseries
  SUPPLEMENT
\end{center}
\renewcommand\thesection{\Alph{section}}
\setcounter{section}{0}

\section{Auxiliary technicalities}\label{app}

The following result is needed in the beginning of the proof of
Theorem~\ref{result}, as we have to bound the beta-binomial likelihood in terms
of integer valued parameters $a$ in order to apply Lemma 4.1 in
\cite{Hall}.

\begin{lem}\label{Monotonicity_a}
  For $k, n, s\in\mathbb{N}$ and $b > 0$ such that $2 \le s \le kn$ define the
  function
  \begin{equation*}
    f(a) = \frac{\Gamma\big(s + a\big)}{\Gamma(kn+a+b)}
  \end{equation*}
  for $a \ge 0$. Then $f$ is monotonically decreasing and $f\big(\lfloor
  a\rfloor\big) / f\big(\lceil a\rceil\big) \le c\,kn/s$ for $c \ge 2\,(1
  + \lceil a \rceil + b)$.
\end{lem}

\begin{proof}
  It is sufficient to look at $h(a) := \Gamma\big(y + a\big) / \Gamma(z + a)$,
  where $2 \le y < z$ are fixed. For $\epsilon > 0$, we find that $\log h(a
  + \epsilon) \le \log h(a)$ is equivalent to
  \begin{equation*}
    \gamma(y + a + \epsilon) - \gamma(y + a) \le\gamma(z + a + \epsilon)
    - \gamma(z + a)
  \end{equation*}
  with $\gamma(t) = \log \Gamma(t)$ for $t > 0$. This inequality is true since
  $\gamma$ is convex, see \cite{Merkle1996}, which therefore establishes
  monotonicity. We also find
  \begin{equation*}
    \frac{h\big(\lfloor a \rfloor\big)}{h\big(\lceil a \rceil\big)}
      \le \frac{z + \lceil a\rceil - 1}{y + \lfloor a \rfloor - 1}
      \le \frac{z + \lceil a \rceil}{y-1} \le 2\,\frac{z + \lceil a \rceil}{y}.
  \end{equation*}
  Substituting $y = s$ and $z = kn + b$, and using that $kn + \lceil a \rceil
  + b \le (1 + \lceil a\rceil + b)\,kn$ yields the second result.
\end{proof}

Recall that $(c)_j = c(c-1)\cdots(c-j+1)$ denotes the $j$-th falling factorial
for real values $c$.  Factorials like these have their origin in Lemma 4.1 in
\cite{Hall} and play a prominent role in our proof.

\begin{lem}
  \label{lem.falling_factorial}
  Let $j\in\mathbb{N}$ and $n, m > 1$ with $j \le \min\{m, n\}$. Then
  \begin{enumerate}
    \item[i)] $(m/e^2)^j \le (m)_j \le m^j$.
    \item[ii)] $\frac{m^j(n)_j}{n^j(m)_j} \ge 1 + \frac{n-m}{nm}$ for $n \ge m$
    and $j > 1$.
  \end{enumerate}
\end{lem}

\begin{proof}
 
To prove the first statement of the lemma, we apply the Stirling type bounds
$\sqrt{2\pi}\, m^{m+1/2} e^{-m} \le \Gamma(m+1) \le e\,\sqrt{2\pi}\,m^{m+1/2}
e^{-m}$ of Theorem~1 in \cite{Jameson2015}, which can be used to derive
  \begin{equation*}
    (m)_j = \frac{\Gamma(m+1)}{\Gamma(m-j+1)} \ge \frac{1}{e}
    \left(\frac{m}{m-j}\right)^{m-j + 1/2} \left(\frac{m}{e}\right)^j \ge
    \left(\frac{m}{e^2}\right)^j.
  \end{equation*}
To see the second statement, assume $n \ge m$, $j > 1$, and bound
  \begin{equation*}
  \frac{m^j(n)_j}{n^j(m)_j} = \prod_{i=0}^{j-1} \frac{m(n-i)}{(m-i)n}
  = \prod_{i=0}^{j-1}\left(1+i\frac{n-m}{n(m-i)}\right) \ge 1 + \frac{n-m}{nm}.
  \end{equation*}
 
\end{proof}

When we construct the event $\mathcal{X}_k$ in the proof of
Theorem~\ref{result}, we rely on proper behavior of the random variables $U_j$
and $T_j$ introduced in~\eqref{tjuj-def}. In particular, we exploit that $U_j$
concentrates around
\begin{equation}\label{eq:smallujdef}
    u_j = \frac{(knp + a)_j}{(km + a + b - 1)_j},
\end{equation}
which is a surrogate for $\mathbb{E}_{n,p}[U_j]$. The next three results
characterize the concentration of $U_j$ around $u_j$, the decay of $u_j$ as $k$
and $m$ become large, and the difference of $u_j$ and $t_j
= \mathbb{E}_{n,p}[T_j] = (n)_kp^k / (m)_k$.

\begin{lem}\label{Uj-concentration}
  Let $k, n, s \in \mathbb{N}$, $m > 0$ and $p\in (0,1),$ such that $k \ge 2$
  and $km \ge s$. Let furthermore $a \ge 0, b > 0$ and define $\Delta = s - knp$
  as well as
  \begin{equation*}
    U_{j} = \frac{(s+a)_j}{(km + a + b - 1)_j}
  \end{equation*}
  for $j\in\mathbb{N}$ with $j \le s + a$. Then, for any $\lambda \ge np$ with
  $\lambda > 1$, $|\Delta| \le \sqrt{\lambda\,k\log{k}}$ implies 
  \begin{equation*}
     |U_j - u_j| \le j\,\sqrt{\frac{\lambda\,\log{k}}{k}}\,
     \left(\frac{2e^2\,(3\lambda + a + 1)}{m}\right)^j.
  \end{equation*}
\end{lem}

\begin{proof}
  Let $t = \sqrt{\lambda\,k\log{k}},$ $D=knp+a$ and assume that $|\Delta| \le
  t$. Noting that $t \le k \lambda$ and thus $j \le s + a \le 2k\lambda + a$, we
  can bound $D + t + 1 + j \le (3 \lambda + a + 1)k =: ck$.
  Applying a telescoping sum, we find
  \begin{align}
  \begin{split}
    |U_j - u_j|
    &= \frac{|(D + \Delta )_j - (D)_j|}{(km + a + b - 1)_j} \\
    & \le \sum_{l=0}^{j-1} \frac{|D + \Delta|\dots|(D + \Delta  - l) - (D
      - l)|\dots |D - j + 1|}{(km + a + b - 1)_j} \\
    & \le \sum_{l=0}^{j-1} \frac{(ck)^{j-1} \, t}{(km/2e^2)^j} \le
      j \,\frac{t}{k}\,\left(\frac{2e^2\,(3\lambda + a + 1)}{m}\right)^j,
    \end{split}
    \label{eq:Uj-concentration-inequality}
\end{align}
where the denominator is bound from below by the first statement of
Lemma~\ref{lem.falling_factorial},
  \begin{equation}
    \label{eq:U-denominator-bound}
    (km + a + b - 1)_j \ge \big((km - 1)/e^2\big)^j \ge (km/(2e^2))^j.
  \end{equation}
\end{proof}

\begin{lem}\label{Utilde-bound}
  Let $k, n \in\mathbb{N}$, $m > 0$ and $p\in(0, 1)$ such that $k \ge 2$, and
  let $a \ge 0, b > 0$. Let $u_j$ as defined in \eqref{eq:smallujdef}.
  For any $\lambda \ge np$ it holds that
  \begin{equation*}
    |u_j| \le \left(\frac{6e^2(\lambda+a)}{m}\right)^j
  \end{equation*}
  if $j\in\mathbb{N}$ with $j \le 2\,k\lambda + a$ and $j < km + a + b$.
  Furthermore, if $2 \le j \le m$ and $m\le k$, then
  \begin{equation*}
    |u_j| \le \left(\frac{np}{m}\right)^j + \frac{j(3\,\lambda\,(a+b+1) + 2a
    + 4)^j}{mk}.
  \end{equation*}
\end{lem}

\begin{proof}
 Generously bounding $|knp + a - j| \le |3k(\lambda + a)|$ and therefore $|(knp
 - a)_j| \le  \big(3k(\lambda + a)\big)^j$,
 the first result follows from inequality \eqref{eq:U-denominator-bound} of the
 previous lemma.
 In case of $j\le m$ it holds that
	\begin{align*}
		\left|\frac{knp+a-i+1}{km+a+b-i}- \frac{np}{m}\right|
    &=
    \left| \frac{m(a-i+1)-np(a+b-i)}{m(km+a+b-i)} \right|\\
    &\leq \frac{im + (a+1) m + inp + (a+b)np}{m(km-i)}\\
    &\leq \frac{m^2 + m (a+1 + np\,(a+b + 1))}{m^2(k-1)}\\
    &\leq 2\,\frac{a + 2 + \lambda\,(a+b+1)}{k}
    =: \frac{c}{k}
	\end{align*}
  for each $i = 1, \dots, j$. This inequality yields the upper bound
  \begin{equation}\label{eq:utilde-bound}
    |u_j| \le \bigg(\frac{np}{m} + \frac{c}{k}\bigg)^j.
  \end{equation}
  Expanding $(x + y)^j$ as binomial sum shows the elementary inequality $(x
  + y)^j \le x^j + j\,y\,\big(x + y\big)^{j-1}$ for $x, y > 0$.
  Hence, for $j \ge 2$ and $m \le k$, we obtain
  \begin{equation*}
    |u_j| \le \left(\frac{np}{m}\right)^j + j\,\frac{c}{k}\left(\frac{np}{m}
    + \frac{c}{k}\right)^{j-1} \le \left(\frac{np}{m}\right)^j + \frac{j(c
    + \lambda)^j}{mk}. 
  \end{equation*}
\end{proof}

\begin{lem}\label{Utilde-Tj-bound}
  Let $k, j \in\mathbb{N}$ with $k \ge 2$, $m > 0$, and $a \ge 0, b > 0$.
  Assume $(n,p)\in\mathcal{M}_k(\lambda)$ for $\lambda > 1$ as defined in
  \eqref{restriction}.  If $m < n \le k$ and $2 \le j \le m$, it holds that
  \begin{equation*}
    t_j - u_j \ge \frac{1}{\lambda^{j}} \frac{n - m}{nm^{j+1}}
      - j\,\frac{(3\,\lambda\,(1+a+b) + 2a + 4)^j}{mk}.
  \end{equation*}
  If $n<m$ and $j \le n$, we have
  \begin{equation*}
    t_2 - u_2 \le - \frac{1}{\lambda^2(1+a+b)^2}\,\frac{m - n}{nm^3}
      + 2\,\frac{18(1+\lambda)^3(1+a+b)^2}{m^2k}
  \end{equation*}
  and, for $j > 2$ and $k > \lambda|n-a|$,
  \begin{equation*}
    t_j - u_j \le j\,\frac{18(1+\lambda)^3(1+a+b)^2}{m^2k}.
  \end{equation*}
\end{lem}

\begin{proof}
  Let $c=j(3\,\lambda\,(a+b+1) + 2a + 4)^j$.
  Applying the respective second statements of Lemma~\ref{Utilde-bound} and
  Lemma~\ref{lem.falling_factorial}, we observe
  \begin{equation*}
    t_j - u_j
    \ge \left[\frac{(n)_j\,p^j}{(m)_j} - \left(\frac{np}{m}\right)^j\right]
        - \frac{c}{mk}\ge \left(\frac{np}{m}\right)^j\frac{n - m}{nm} - \frac{c}{mk}
  \end{equation*}
  for $n > m$, which establishes the first result due to $np \ge 1/\lambda$ for
  $(n,p)$ in $\mathcal{M}_k(\lambda)$.
  For the second result, assume $m > n$. We look at the case $j = 2$ first.
  Recalling $1/\lambda \le np \ge \lambda$ for $(n,p)\in\mathcal{M}_k(\lambda)$,
  straightforward calculations show
  \begin{align*}
    t_2 - u_2 &=\frac{n(n-1)p^2}{m(m-1)} - \frac{(knp +a)(knp + a - 1)}{(km
    + a + b -1)(km + a + b -2)} \\
    &\le \frac{np}{m} \left( \frac{n-1}{m-1}p - \frac{np}{m}
    \frac{1-1/knp}{\big(1+(a+b)/km\big)^2} \right) \\
    &\le - \frac{np}{nm} \frac{np\,(m - n)
    - \tilde{c}\,nm/k}{m\,(m-1)\,\big(1+(a+b)/km\big)^2} \\
    &\le - \frac{1}{\lambda^2 (1+a+b)^2}\,\frac{m-n}{nm^3}
    + \frac{2\lambda\tilde{c}}{m^2k}
  \end{align*}
  with $\tilde{c} = (1 + \lambda)\,(1 + a + b)^2$. This shows the first part of
  the second statement. For $j > 2$, we consider
  \begin{align*}
    t_j - u_j
    &= \frac{(n)_j p^j}{(m)_j} - \frac{(knp + a)_j}{(km+a+b-1)_j} \\
    &= \left[\frac{(n)_j p^j}{(m)_j} - \frac{n^jp^j}{m^j}\right] + \left[
    \left(\frac{np}{m}\right)^j - \frac{(knp+a)_j}{(km+a+b-1)_j}\right].
  \end{align*}
  Applying the second part of Lemma~\ref{lem.falling_factorial} with the roles
  of $m$ and $n$ reversed (since $n < m$), we see that the first term in this
  expression is $\le 0$. Using $j\le n$ , we thus find for $k > \lambda|n-a|$
  that
  \begin{equation*}
    t_j - u_j \le \left(\frac{np}{m}\right)^j - \left(\frac{knp+a-n}{km+a+b}\right)^j.
  \end{equation*}
  Note that both terms in brackets are positive and smaller than $1$.
  Consequently, we can apply the basic inequality
  \begin{equation*}
    |x^j - y^j| \le j\,|x-y|\,(x+y)^2
  \end{equation*}
  for $x,y\in[0, 1]$ and $j \ge 3$, which can be derived by an induction
  argument. This yields
  \begin{equation*}
    t_j - u_j \le j\,\left|\frac{np}{m}
    - \frac{knp+a-n}{km+a+b}\right|\,\left(\frac{np}{m}
    + \frac{knp+a-n}{km+a+b}\right)^2.
  \end{equation*}
  The result of the lemma follows from observing that
  \begin{equation*}
    \frac{np}{m} + \frac{knp+a-n}{km+a+b} \le \frac{3\lambda}{m}
    \quad\text{and}\quad
    \left|\frac{np}{m} - \frac{knp+a-n}{km+a+b}\right| \le
    \frac{2(a+b+1)\lambda}{k}.
  \end{equation*}
  
\end{proof}

\section{Proof of Theorem \ref*{thm:bounded_n}}\label{app:thm3} 
Like for Theorem \ref{result}, the proof crucially relies on suitably bounding
a second order component of the likelihood. The following auxiliary result helps
making sure that we always observe at least some values $X_i \ge 2$ in the
setting of Theorem \ref{thm:bounded_n}.

\begin{lem}\label{helper-bounded_n}
  Let $X_1, ..., X_k\sim \mathrm{Bin}(n, p)$ be independent for $n\ge 2$ and
  $p\in(0, 1)$, and let $M_k$ denote the sample maximum. Then 
  \begin{equation*}
    P\big(M_k < 2\big) \le \exp(-p^2 k).
  \end{equation*}
\end{lem}
\begin{proof}
  For $n \ge 2$, we find
  \begin{align*}
    \mathbb{P}(M_k < 2) &= \mathbb{P}(X_1 \le 1)^k \le
    \mathbb{P}\big(\mathrm{Bin}(2, p) \le 1\big)^k = (1-p^2)^k \le e^{-p^2k}.
  \end{align*}
\end{proof}

\begin{proof}[Proof of Theorem~\ref*{thm:bounded_n}]
The proof follows the same overarching structure as the proof of
Theorem~\ref{result}. Several details and arguments, however, have to be
adapted to the new asymptotic setting described by
$\mathcal{M}^\mathrm{b}_k(B)$, as we can no longer rely on the assumption that
the product $n_kp_k$ is bounded away from zero.

We begin by arguing that it is sufficient to consider
parameters $(n_k, p_k)$ in
\begin{equation*}
  \wt{\mathcal{M}}^\mathrm{b}_k(\MM)
  :=
  \Big\{(n, p)\in\mathcal{M}^\mathrm{b}_k(\MM)\,\Big|\,p\le \frac 1{(2\MM)^6}\Big\},
\end{equation*}
which guarantees that $p_k$ is sufficiently small.  The general case for an
arbitrary sequence $(n_k,p_k)_k$ in $\mathcal{M}_k^\mathrm{b}(B)$ is then
a simple conclusion of Theorem~\ref{result}: let $K\subset \mathbb{N}$ be the
set of indices $k$ where $p_k > 1/(2\MM)^6$. Then, $(n_k, p_k)_{k\in K}$ has
elements in $\mathcal{M}_k(\lambda)$ for $\lambda = (2\MM)^6$ and $(n_k,
p_k)_{k\in\mathbb{N}\setminus K}$ has elements in
$\wt{\mathcal{M}}^\mathrm{b}_k(\MM)$. Thus, posterior consistency holds for both
partial sequences, and hence also for the sequence $(n_k, p_k)_{k\in\mathbb{N}}$
as a whole.

To show posterior consistency for $(n_k,p_k)$ in
$\wt{\mathcal{M}}^\mathrm{b}_k(B)$, we recall the upper bound
\eqref{posterior-bound} for $A_k = \mathbb{N}\setminus\{n_k\}$,
\begin{equation}\label{posterior-bound-nconst}
  \Pi\big(n\neq n_k\,|\,\data\big) 
  \le c_1 \frac{kn_k}{S_k}\,
    \sum_{n\neq n_k, n\ge M_k} \exp\left(k\int_{n_k}^{n} f'(m)\,\mathrm{d}m\right)
      \,\frac{\Pi_n(n)}{\Pi_n(n_k)},
\end{equation}
which holds for $S_k \ge 2$ with $c_1 = 2\,(1 + \lceil a \rceil + b)$. The
derivative $f'(m)$ is given by equation
\eqref{fprime} and reads
\begin{equation}\label{fprime-nconst}
	f'(m)
	=
  \sum_{j=1}^{M_k}\frac{ T_j-U_j}{j} -\sum_{j=M_k+1}^{ S_k+a}\frac{U_j}{j},
\end{equation}
where we again change the notation and write $a \in\mathbb{N}_0$ to refer to
$\lfloor a \rfloor$ from now on.
The random variables $T_j$ and $U_j$ are defined in equation
\eqref{tjuj-def} in the proof of Theorem~\ref{result}, and we will
again make use of the quantities $t_j := \mathbb{E}_{k} T_j$ and $u_j := (k n_k
p_k + a)_j / (km+a+b-1)_j$ in order to center $T_j$ and $U_j$.

We now consider a sequence of events $\mathcal{X}_k$ that is the intersection of
the sets
\begin{align*}
  \mathcal{R}_k &:= \big\{ M_k \ge 2 \big\},\\[0.1cm]
  \mathcal{T}_k &:= \bigcap_{j=1}^{M_k} \Big\{ (m)_j\big|T_j - t_j\big| \le
    \sqrt{(c_2p_kj)^j\,l_k\log(k)/k}\Big\}, \\[0.1cm]
  \mathcal{S}_k &:= \Big\{ \big|S_k - kn_kp_k\big| \le \sqrt{\MM l_k k}\Big\}
\end{align*}
for a fixed diverging sequence $1 \le l_k = o\big(\log(k)\big)$ and
the constant $c_2 = 2\MM(\MM+2)$. From $kn_kp_k \ge kp_k \ge
\sqrt{k}\log(k)/\MM$ and $kn_kp_k \le k/64\MM^6 \le k/64\MM$ for $k$ large
enough, one can derive that $\sqrt{k}/2\MM \le S_k \le k/32\MM$ holds on
$\mathcal{S}_k$ (use $\sqrt{\MM l_k k} \le \sqrt{k}\log(k)/2\MM \le k/64\MM$ for
large $k$).
On the event $\mathcal{R}_k$ we additionally find
\begin{equation}\label{eq:ska-bounded_n}
  kM_k \ge k \ge 16\,(S_k + a)
\end{equation}
for $k$ large enough such that $a \le S_k$.

With analog arguments as before we can show that equation \eqref{restriction}
holds in this setting, too. The only novel step is to prove
$\mathbb{P}_{k}(\mathcal{R}_k^c)\to 0$ as $k\to\infty$ for the sets
$\mathcal{R}_k$. This is taken care of by Lemma~\ref{helper-bounded_n}, where we
use that $kp^2_k\to\infty$ for sequences $(n_k, p_k)\in \wt
{\mathcal{M}}_k^{\mathrm{b}}(B)$.

We next bound $f'(m)$ from below ($m < n_k$) and above ($m > n_k$) for $\data$
restricted to $\mathcal{X}_k$. Our approach closely follows the proof of
Theorem~\ref{result}. 

\paragraph{Auxiliary lower bound}
Let $M_k\le m < n_k \le \MM$. In order to lower bound $f'(m)$, we consider the
inequality
\begin{equation}
  f'(m) \ge \sum_{j=1}^2 \frac{T_j-U_j}{j} - \sum_{j=3}^{S_k+a} \frac{U_j}{j},
\end{equation}
which follows from $T_j \ge 0$ for all $j$. For $j = 1$, a short
calculation shows
\begin{equation*}
  T_1 - U_1
  =
  \frac{S_k}{km} - \frac{S_k + a}{km+a+b-1} \ge -\frac{2(a+1)}{k}
  \ge
  -2\,(a+1)\MM \,\frac{p_k}{\sqrt{k}},
\end{equation*}
where we inserted $1/\sqrt{k} \le \MM p_k$ for sequences in $\wt
{\mathcal{M}}_k^{\mathrm{b}}(B)$. In case of $j > 1$, we first note that
\begin{equation*}
  \frac{\big|T_2 - t_2\big|}{2} \le c_2\,\frac{p_k \sqrt{l_k}}{\sqrt{k}}
\end{equation*}
on the event $\mathcal{T}_k$. Next, we apply a modification of
Lemma~\ref{Uj-concentration},
where we bound the telescoping sum in \eqref{eq:Uj-concentration-inequality}
more carefully. Let $t = \sqrt{\MM l_kk}$. For $2 \le j \le S_k + a$, we derive 
\begin{equation}\label{uj-concentration-alternative}
  \big|U_j - u_j\big|
  \le j\,
  \underbrace{\left(\frac{t}{km - j}\right)}_{\le 8\sqrt{\MM l_k}/7m\sqrt{k}}
  \underbrace{\left(\frac{kn_kp_k + t + a}{km - j}\right)}_{\le 16\MM p_k/7m}
  {\underbrace{\left( \frac{kn_kp_k + t + a + 1 + j}{km - j} \right)}_{\le 4/7}}^{j-2}
\end{equation}
on $\mathcal{X}_k$, where we used that $\max(kn_kp_k, t, j) \le k/8 \le km/8$
(guaranteed by \eqref{eq:ska-bounded_n}) and $(kn_kp_k + t + a)/k \le 2\MM p_k$ for
$k$ large enough. The quantities $u_j$, which are used as surrogate
for the expectations $\mathbb{E}_k[U_j]$, are defined on page \pageref{utilde-def}
in the proof of Theorem~\ref{result}. Bounding $m \ge 1$, it follows that
\begin{equation*}
  \sum_{j=2}^{S_k + a} \frac{\big|U_j - u_j\big|}{j} 
  \le
  \frac{128\,\MM^2}{49}\frac{p_k\sqrt{l_k}}{\sqrt{k}}\sum_{j=0}^{\infty}
  \left(\frac{4}{7}\right)^{j}
  \!=: c_3 \,\frac{p_k \sqrt{l_k}}{\sqrt{k}}
\end{equation*}
Next, assuming that $j \le S_k + a \le km/4$, which is again guaranteed by
\eqref{eq:ska-bounded_n}, we bound
\begin{align*}
  \sum_{j=3}^{S_k + a} \frac{|u_j|}{j}
  &=
  \sum_{j=3}^{S_k + a} \frac{1}{j}\frac{|(kn_kp_k + a)_j|}{(km + a + b -1)_j} \\
  &\le
  \sum_{j=3}^{S_k + a} \frac{(kn_kp_k + a)_3}{(km - 1)_3}
    \frac{|(kn_kp_k + a - 3)_{j-3}|}{(km - 4)_{j-3}} \\
  &\le 
  \left(\frac{kn_kp_k + a}{km - 3}\right)^{\!3}\,\sum_{l=0}^{S_k + a - 3}
  \left(\frac{kn_kp_k + km/4 + a + 1}{3\,km/4}\right)^l \\
  &\le 
  \left(\frac{2n_kp_k}{m}\right)^3\sum_{l=0}^\infty 2^{-l}
  \le 16 \MM^3p_k^3
\end{align*}
for $k$ large enough such that $kn_kp_k + a \ge 3$ and $3\,k/4 \ge 6\,(kn_kp_k
+ a + 1)$ (first two inequalities) as well as $k-3 \ge 2k/3$ and $3a/2k \le
n_kp_k/2$ (second to last inequality). 
We now apply the first part of Lemma~\ref{Utilde-Tj-bound} (adopted to
parameters in $\mathcal{M}^\mathrm{b}_k(B)$) to derive
\begin{align*}
  \frac{t_2 - u_2}{2}
  &\ge 
  \frac{(n_kp_k)^2}{2n_k} \frac{n_k - m}{m^3} - \frac{c_4}{k} \\
  &\ge
  \frac{n_k - m}{2\MM^3}\,p_k^2 - c_4\MM\,\frac{p_k}{\sqrt{k}},
\end{align*}
where $c_4 \ge 2\,\big(3\,\MM(a+b+1)+2a+4\big)^2$. Combining the bounds above,
we find
\begin{align*}
  f'(m) 
  &\ge (T_1 - U_1) + \frac{T_2 - t_2}{2}
    + \sum_{j=2}^{S_k+a} \frac{u_j - U_j}{j} +
    \frac{t_2 - u_2}{2} - \sum_{j=3}^{S_k+a}
    \frac{u_j}{j} \\
  &\ge 
    \left(\frac{n_k - m}{2\MM^3} - 16\MM^3\,p_k\right) p_k^2 
    - \big(2(a+1)\MM + c_2 + c_3 + c_4\MM\big)\,\frac{p_k\sqrt{l_k}}{\sqrt{k}} \\
  &\ge 
  \left(\frac{n_k - m}{2} - \frac{1}{8}\right) \frac{p_k^2}{\MM^3}
  - C\,\frac{p_k\sqrt{l_k}}{\sqrt{k}}
\end{align*}
for $C = 2(a+1)\MM + c_2 + c_3 + c_4\MM$, where we made use of the restriction
$p_k\le1/(2\MM)^6$ that we demanded for the set
$\wt{\mathcal{M}}^\mathrm{b}_k(\MM)$.
Integration over $m$ for some $n\in\mathbb{N}$ with $n < n_k$ yields
\begin{equation}\label{eq:lower-bound-bounded_n}
  -k \int_n^{n_k} f'(m)\,\mathrm{d}m \le -\frac{1}{8\MM^3}\,kp_k^2
  + \MM\,C \,p_k\sqrt{l_kk}.
\end{equation}
Our asymptotic setting implies $kp_k \ge \log(k)\sqrt{k}/\MM$ for sequences in
$\wt {\mathcal{M}}_k^{\mathrm{b}}(B)$. Since $l_k$ was defined such that
$l_k/\log(k) \to 0$ as $k\to\infty$, the lower bound in
\eqref{eq:lower-bound-bounded_n} thus tends to $-\infty$ as $k\to\infty,$ and
the convergence is uniform over all sequences in $\wt
{\mathcal{M}}_k^{\mathrm{b}}(B)$.

\paragraph{Auxiliary upper bound}
We now assume $n_k < m$ and bound $f'(m)$ from above. Since each $U_j$
is non-negative, it is sufficient to upper bound
\begin{equation*}
  f'(m) \le \sum_{j=1}^{M_k} \frac{T_j-U_j}{j}.
\end{equation*}
Due to inequality \eqref{eq:ska-bounded_n}, the first term in this sum obeys
\begin{equation*}
  T_1 - U_1 \le \frac{S_k (a+b)}{km (km-1)} \le \frac{a+b}{m^2k}
  \le
  (a+b)\MM\,\frac{p_k}{m^2\sqrt{k}}
\end{equation*}
for $k$ sufficiently large, once again using $1/\sqrt{k}\le\MM p_k$.
Recalling the definition of the event $\mathcal{T}_k$ and
applying the result $(m)_j \ge m^j/e^2$ derived in Lemma~\ref{lem.falling_factorial},
we furthermore find
\begin{equation*}
  \sum_{j=2}^{M_k} \frac{\big|T_j - t_j\big|}{j}
  \le
  \sum_{j=2}^{M_k} \frac{\sqrt{(c_2p_kj)^j l_k}}{\sqrt{k}}\,\frac{e^2}{m^j}
  \le
  c_2e^2\MM \sqrt{c_2\MM}^\MM\,\frac{p_k\sqrt{l_k}}{m^2\sqrt{k}}
  =:
  c'_2\,\frac{p_k\sqrt{l_k}}{m^2\sqrt{k}}.
\end{equation*}
We next employ inequality \eqref{uj-concentration-alternative} to establish the
upper bound
\begin{align*}
  \sum_{j=2}^{M_k} \frac{\big|U_j - u_j\big|}{j}
  &\le
  \frac{128\,\MM^2}{49}\frac{p_k\sqrt{l_k}}{m^2\sqrt{k}}~\sum_{j=0}^\infty
  \left(\frac{4}{7}\right)^j
  =: c'_3\,\frac{p_k\sqrt{l_k}}{m^2\sqrt{k}},
\end{align*}
and cite (a slightly adapted version of) the second result of
Lemma~\ref{Utilde-Tj-bound} to obtain
\begin{equation*}
  \sum_{j=2}^{M_k} \frac{t_j - u_j}{j}
  \le
  - C_1' \frac{m - n_k}{m^3} p_k^2 + c_4' \frac{p_k\log(k)}{m^2k}
\end{equation*}
with $C_1' = 1/(1+a+b)^2$ and $c_4' = 18 (1+B)^3(1+a+b)^2$.
Combining the inequalities above, we
find
\begin{align*}
  f'(m)
  &\le (T_1 - U_1) +
    \sum_{j=2}^{M_k} \frac{T_j - t_j}{j} +
    \sum_{j=2}^{M_k} \frac{u_j - U_j}{j} +
    \sum_{j=2}^{M_k} \frac{t_j - u_j}{j} \\
  &\le
    - C_1'\,\frac{m - n_k}{m^3} p_k^2
    + C'\,\frac{p_k\sqrt{l_k}}{m^2\,\sqrt{k}}
\end{align*}
for $C' = (a + b)\MM + c_2' + c_3' + c_4'$. Integration over $m$ yields
\begin{equation}\label{eq:upper-bound-bounded_n}
  k\int_{n_k}^n f'(m)\,\mathrm{d}m
  \le
  -\frac{C_1'}{8\MM^3}\,kp_k^2
    + C'\,p_k\sqrt{l_kk}
\end{equation}
for any $n\in\mathbb{N}$ with $n > n_k$.
In this bound, we exploited that
\begin{align*}
  \int_{n_k}^n \frac{m-n_k}{m^3}\,\mathrm{d}m
  &\ge
  \int_{\MM}^{\MM+1} \frac{m-\MM}{m^3}\,\mathrm{d}m
  =
  \frac{1}{2\MM(\MM+1)^2}
  \ge
  \frac{1}{8\MM^3}
\end{align*}
and $\int_{n_k}^n m^{-2}\,\mathrm{d}m \leq \int_1^\infty m^{-2}\,\mathrm{d}m=1.$

Similarly to the lower bound \eqref{eq:lower-bound-bounded_n}, the upper bound
\eqref{eq:upper-bound-bounded_n} diverges to $-\infty$ as
$k\to\infty$, uniformly over $\wt{\mathcal{M}}^\mathrm{b}_k(\MM)$.
Inserting these bounds in inequality \eqref{posterior-bound-nconst} thus yields
$\Pi\big(n\neq n_k\,\big|\,\data\big) \to 0$ as $k\to\infty$ on the event
$\mathcal{X}_k$, since the prefactor $c_1kn_k/S_k$ is bounded by
$c_1\MM\exp\!\big(\log(k)\big)$. Posterior consistency follows from
$\mathbb{P}_{k}(\mathcal{X}_k^c) \to 0$ as $k \to \infty$ and
equation \eqref{restriction}.
\end{proof}

\section{Proof of Theorem 3}\label{app:thm4}

For probability measures $P$ and $Q$ on the same measurable space, the
$\chi^2$-divergence is defined as $\chi^2(P,Q)=\int (dP/dQ)^2 dQ-1$ if $P$ is
dominated by $Q,$ and $\chi^2(P,Q)=\infty$ otherwise. If $P,Q$ are discrete
distributions with probability mass functions $p$ and $q$, then
$\chi^2(P,Q)=\sum_k p(k)^2/q(k) -1$.

The following result quantifies the $\chi^2$-divergence between two distinct
binomial distributions that are chosen to match each other as closely as possible.

\begin{lem}\label{lem:chi2bound}
  Let $n\in\mathbb{N}$, $p\in(0, 1),$ $q = 1-p,$ and $p_n = np/(n+1)$.  If
  $\chi^2(n, p)$ denotes the $\chi^2$-divergence between the distributions
  $\mathrm{Bin}(n, p)$ and $\mathrm{Bin}(n+1, p_n)$,
  then
  \begin{equation*}
    \chi^2(n, p) + 1 = \left(1 - \frac{p}{nq+1}\right) \left(1 + \frac{p}{nq(nq+1)}\right) \left(1 + \frac{p}{n(nq+1)}\right)^{n-1}.
  \end{equation*}
  Furthermore, if $p \le 1-\delta$ for some $\delta > 0$, then
  \begin{equation*}
    \chi^2(n, p) \le c\, \frac{p^2}{n^2}
  \end{equation*}
  for a positive constant $c$ only depending on $\delta$.
\end{lem}

\begin{proof}
Denote by $f_1$ the probability mass function of the $\mathrm{Bin}(n, p)$
distribution. From the moment generating function of the binomial distribution,
we deduce the moment identities $\sum_{r=0}^n t^r f_1(r) = (q + pt)^n$ and
$\sum_{r=0}^n r\,t^r f_1(r) = t\,np\,(q + pt)^{n-1}$ for $t > 0$, where the
second equation follows from the first by differentiation with respect to $t$.

Using $(n+1)(1-p_n)/(nq)=1+1/(nq)=:t^*$ and applying the two moment identities
for $t=t^*,$ we find
  \begin{align*}
    \chi^2(n, p) + 1
    &= \sum_{r=0}^n
    \frac{\binom{n}{r}p^r(1-p)^{n-r}}{\binom{n+1}{r}p_n^r (1-p_n)^{n+1-r}}~f_1(r) \\
    &= \frac{1}{n+1}\frac{1}{1-p_n}\left(\frac{1-p}{1-p_n}\right)^n
       \sum_{r=0}^n
     (n+1-r) \left(\frac{n+1}{n}\frac{1-p_n}{1-p}\right)^{\!r} \!f_1(r) \\
    &= \frac{1}{1-p_n}\left(\frac{1-p}{1-p_n}\right)^n
    \left(1 + \frac{p}{nq}\right)^{n-1}
    \left(\left(1 + \frac{p}{nq}\right) - \frac{np}{n+1}\left(1
    + \frac{1}{nq}\right)\right).
  \end{align*}
  We also have $(1-p)/(1-p_n)=1-p/(nq+1)$ and therefore
  \begin{equation*}
    \frac{1}{1-p_n} \left(\left(1 + \frac{p}{nq}\right) - \frac{np}{n+1}\left(1
    + \frac{1}{nq}\right)\right) = 1 + \frac{p}{nq(nq+1)},
  \end{equation*}
  proving
  \begin{equation*}
    \chi^2(n, p) + 1
    = \left(1 - \frac{p}{nq+1}\right)^n \left(1 + \frac{p}{nq}\right)^{n-1}
    \left(1 + \frac{p}{nq(nq+1)}\right).
  \end{equation*}
 Since 
  \begin{equation*}
    \left(1 - \frac{p}{nq+1}\right) \left(1 + \frac{p}{nq}\right)
    = 1 + \frac{p}{n(nq+1)},
  \end{equation*}
  we finally obtain 
  \begin{equation*}
    \chi^2(n, p) + 1 = \left(1 - \frac{p}{nq+1}\right) \left(1
    + \frac{p}{nq(nq+1)}\right) \left(1 + \frac{p}{n(nq+1)}\right)^{n-1}.
  \end{equation*}
  This shows the first part of the claim. 
  
  To prove that $\chi^2(n,p)\leq cp^2/n,$ recall that $p \le 1-\delta < 1$. For
  $n = 1$,
  \begin{equation*}
    \chi^2(n, p)
    = \left(1 - \frac{p}{q+1}\right)\left(1 + \frac{p}{q(q+1)}\right) - 1
    = \frac{p^2}{(q+1)^2} \le \frac{p^2}{n^2}
  \end{equation*}
  If $n
  > 1$, we can upper bound the last factor in the formula for $\chi^2(n, p) + 1$ by
  \begin{align*}
    \left(1 + \frac{p}{n(nq+1)}\right)^{n-1}
    &= 1 + \frac{(n-1)\,p}{n(nq + 1)} + \sum_{r=2}^{n-1} \binom{n-1}{r}
       \left(\frac{p}{n(nq+1)}\right)^r \\
    &\le 1 + \frac{(n-1)\,p}{n(nq + 1)} + \frac{p^2}{n^2}\,\sum_{r=2}^{n-1}
    \frac{q^{-r}}{r!}\\
    &\le 1 + \frac{(n-1)\,p}{n(nq + 1)} + e^{1/\delta}\,\frac{p^2}{n^2},
  \end{align*}
  where we used that $e^{1/\delta} \ge e^{1/q}$. Finally, we note
  \begin{equation*}
    -\frac{p}{nq+1} + \frac{p}{nq(nq+1)} + \frac{(n-1)\,p}{n(nq+1)} = 
    \frac{p^2}{nq(nq+1)} \le c_2\,\frac{p^2}{n^2}
  \end{equation*}
  for some $c_2 > 0$ depending on $\delta.$ Because only terms of order
  $p^2/n^2$ or higher occur when expanding the product in the upper bound
  \begin{equation*}
    \chi^2(n, p)
    \le \left(1 - \frac{p}{nq+1}\right)
    \left(1 + \frac{p}{nq(nq+1)}\right)
    \left(1 + \frac{(n-1)\,p}{n(nq+1)} + c_1 \frac{p^2}{n^2}\right) - 1,
  \end{equation*}
  we conclude that $\chi^2(n, p) \le c\,p^2/n^2$ for a finite constant $c=c(\delta).$
\end{proof}

\begin{proof}[Proof of Theorem \ref*{thm:lower_bound}]
Notice that $\P_{n,p}( \wh n \neq n)=\E_{n,p}[\ell(\wh n,n)]$ for the distance
$\ell(m,n):=\mathbf{1}(m \neq n)$ defined on $\mathbb{N}\times \mathbb{N}.$
Applying Theorem 2.1 combined with part (iii) of Theorem 2.2 in \cite{MR2724359}
shows that if $\chi^2\big(\mathbb{P}_{n_k, p_k},\smash{\mathbb{P}_{n_k+1,
p_k'}}\big)\leq \alpha < \infty,$ then there exists a positive constant
$c_\alpha$ such that for all estimators $\hat{n}_k$ and all $k$
\begin{equation}\label{eq:lbtoshow}
  \max_{(n,p)\in\mathcal{M}^*_k} \mathbb{P}_{n,p}(\hat{n}_k\neq n) \ge
  c_\alpha.
\end{equation}
To prove Theorem~\ref{thm:lower_bound}, it thus remains to bound the respective
$\chi^2$-divergence by a finite constant $\alpha =\alpha(\delta, \eta).$

Since we observe $k$ i.i.d.\ realizations of a binomial distribution in
\eqref{eq:lbtoshow}, we have $\mathbb{P}_{n, p}
= \bigotimes_{j=1}^k\mathrm{Bin}(n, p)$. Applying the equality for the
$\chi^2$-divergence between product measures stated on page 86 in
\cite{MR2724359}, we find
\begin{equation*}
  \chi^2\big(\mathbb{P}_{n_k, p_k}, \mathbb{P}_{n_k+1, p_k'}\big)
  = \Big(1 + \chi^2\big(\mathrm{Bin}(n_k, p_k), \mathrm{Bin}(n_k+1,
    p_k')\big)\Big)^k - 1.
\end{equation*}
To complete the proof, it is therefore enough to show that
\begin{equation*}
  \chi^2\big(\mathrm{Bin}(n_k, p_k), \mathrm{Bin}(n_k+1, p_k')\big)
    \le \frac{a}{k}
\end{equation*}
for a constant $a=a(\eta, \delta).$ This, however, follows directly
from the assumption $n_k/p_k \geq \eta\,\sqrt{k}$ and Lemma~\ref{lem:chi2bound},
since
\begin{equation*}
    \chi^2\big(\mathrm{Bin}(n_k, p_k), \mathrm{Bin}(n_k+1, p_k')\big)
    \le c \left(\frac{p_k}{n_k}\right)^2\leq \frac{c}{\eta^2 k}
\end{equation*}
for a constant $c>0$ only depending on $\delta$.
\end{proof}

\section{Proof of Theorem 4}\label{app:BvM}

\begin{lem}\label{lem.TV_comp}
Consider two discrete distributions of the form $P(X=i)=q_i\Delta_i/\sum_{j}q_j
\Delta_j$ and $Q(X=i)=q_i/\sum_j q_j$ for non-negative $\Delta_i,q_i.$ Then, 
\begin{align*}
  \TV(P,Q) \leq \max_i |1-\Delta_i|.
\end{align*}
\end{lem}

\begin{proof}
Using the triangle inequality gives 
\begin{align*}
    2\TV(P,Q)
    &\leq \sum_i \left| \frac{q_i\Delta_i}{\sum_j q_j \Delta_j}
    - \frac{q_i\Delta_i}{\sum_j q_j}\right|
    +\left|\frac{q_i\Delta_i}{\sum_j q_j}-\frac{q_i}{\sum_j q_j}\right| \\
    &\leq \sum_i \frac{q_i\Delta_i}{\sum_j q_j \Delta_j} \left| \frac{\sum_j q_j
    (1-\Delta_j)}{\sum_j q_j}\right| +\sum_i \frac{q_i}{\sum_j q_j} |\Delta_i-1|.
\end{align*}
Taking the maximum over $|1-\Delta_i|$ yields the result. 
\end{proof}

\begin{proof}[Proof of Theorem \ref*{thm.BvM}]
Let $\mathcal{X}_k= \mathcal{R}_k\cap \mathcal{T}_k \cap \mathcal{S}_k$ be as
defined in \eqref{eq.events_proof_MR}. Most of the subsequent arguments assume
$\data\in\mathcal{X}_k$ and are supposed to hold for sufficiently large $k$,
often without this being mentioned explicitly. It is also convenient to work
with a big $O$ notation, where for two positive sequences $(a_k)_k$ and
$(b_k)_k$, we write $a_k=O(b_k)$ if $a_k/b_k \leq C$ for a constant $C$ and all
sufficiently large $k$ such that the inequality holds uniformly over all
$(n_k,p_k)\in \mathcal{M}_k(\lambda)$ and $\data \in \mathcal{X}_k.$

By definition of $\mathcal{X}_k,$ we can use the inequalities $|S_k-kn_kp_k|\leq
\lambda \sqrt{k\log k}$ and $\big|k\sum_{i=1}^k(X_i)_2-k^2(n_k)_2p_k^2\big|\leq
\sqrt{8} c_2 \big(k\log(k)\big)^{3/2}.$ Since $n_k \leq \lambda
\sqrt{k}/\log^6(k),$
\begin{align}
    |\wh n-n_k|
    &=\left| \frac{n_k+O(n_k\sqrt{\log(k)/k})}{1+O(n_k\sqrt{\log^3(k)/k})} - n_k\right| \nonumber\\
    &= O\!\left(n_k^2 \sqrt{\frac{\log^3(k)}{k}}\right)
    \leq \frac 12 n_k^2 \frac{\log^{7/4}(k)}{\sqrt{k}}.
    \label{eq.nhatn_ub}
\end{align}

Separating the sequence $(n_k,p_k)$ in two subsequences, it is sufficient to
assume either $n_k^2\log^2(k)/\sqrt{k} <1$ for all $k$ or
$n_k^2\log^2(k)/\sqrt{k} \geq 1$ for all $k.$

Consider first the case that $n_k^2\log^2(k)/\sqrt{k} <1.$ Because $n$ is
a discrete parameter, Theorem \ref{result} implies that the posterior will
asymptotically put all mass on the true $n_k,$ uniformly over all sequences
$(n_k,p_k)\in \mathcal{M}_k(\lambda).$ This also means that the posterior
converges in total variation distance to the point measure at $n_k$ denoted by
$\delta_{n_k}.$ Because the total variation distance satisfies the triangle
inequality, it is enough to show that
\begin{equation*}
    \sup_{(n_k,p_k)\in \mathcal{M}_k(\lambda), \data \in \mathcal{X}_k}
    \TV\big(\delta_{n_k}, \mathcal{N}_\mathrm{d}(\wh n, 2n_k^2/kp_k^2)\big) \to 0,
\end{equation*}
since by \eqref{restriction} we know $\sup_{(n_k,p_k)\in \mathcal{M}_k(\lambda)}
\P_{n_k,p_k}(\mathcal{X}_k^c) \to 0$. On the event $\mathcal{X}_k,$ the bound
\eqref{eq.nhatn_ub} gives $\P(Z=n_k) \to 1$ for $Z \sim
\mathcal{N}_\mathrm{d}(\wh n, 2n_k^2/kp_k^2)$ and this yields the claim for the
case that $n_k^2\log^2(k)/\sqrt{k} <1.$ 

It therefore remains to study the non-trivial case with $k^{1/4}/\log(k) \leq
n_k \leq \lambda \sqrt{k}/\log^6(k).$ Let $B_k = \big\{n:|n-n_k|\leq
n_k^2\log^{7/4}(k)/\sqrt{k}\big\}.$ In particular, this implies that $B_k
\subset \{n\in [n_k/2,2n_k]\}$ for all sufficiently large $k$ and $|n-n_k|/n_k
\to 0$ as $k \to \infty.$

Lemma E.1 in \cite{2018arXiv180904140R} states that for probability measures
$P,Q$ on the same measurable space $(\Omega,\mathcal{A})$ and any $A \in
\mathcal{A}$ with $P(A),Q(A)>0,$ we have that
\begin{equation}\label{eq:tvcond}
  \TV(P,Q)\leq \TV(P(\cdot|A),Q(\cdot|A))+2P(A^c)+2Q(A^c),
\end{equation}
where $P(\cdot |A)=P(\cdot \cap A)/P(A)$ and $Q(\cdot |A)=Q(\cdot \cap A)/Q(A).$
We will now apply this inequality for $A=B_k,$ $P$ the posterior and $Q$ the
limit distribution. To show that the total variation distance converges to zero,
we need that the event $B_k^c$ has vanishing probability under the posterior and
the limit distribution.

Let $Z\sim \mathcal{N}_\mathrm{d}(\wh n, 2n_k^2/kp_k^2)$ and denote by $c_k$ the
normalizing constant such that $\P(Z=n)=c_k\exp(-kp_k^2(n-\wh n)^2/(4n_k^2))$
for all integers $n.$ By \eqref{eq.nhatn_ub}, we find $0\geq \log
\P(Z=n_k)=\log(c_k)+O\big(\log^3(k)\big)$ and thus $\log(c_k)\lesssim
\log^3(k).$ Observe that for $a>0,$
\begin{equation*}
    \sum_{q=0}^\infty e^{-aq}=1/(1-e^{-a})\leq \sqrt{e} \max\left(\frac{1}{a}, \frac{1}{\sqrt{e}-1}\right).
\end{equation*}
Here the equality follows by evaluating the geometric sum and the inequality can
be deduced by studying the cases $a\leq 1/2$ and $a>1/2$ separately, noting that
since $1+a\leq e^a$ for all $a,$ we have $e^{-a}+e^{-a}a\leq 1$ and thus
$e^{-a}+e^{-1/2}a\leq 1$ for $0<a\leq 1/2$ which is equivalent to
$1/(1-e^{-a})\leq \sqrt{e}/a.$ Applying this inequality together with
\eqref{eq.nhatn_ub} shows that
\begin{align*}
    \P\big(Z\in B_k^c\big)
    &=
    c_k\sum_{n \in B_k^c } \exp\bigg(-\frac{kp_k^2}{4n_k^2}(n-\wh n)^2\bigg)
    \\
    &\leq
    c_k \sum_{n \in B_k^c } \exp\bigg(-\frac{kp_k^2}{16n_k^2}(n-n_k)^2\bigg)
    \\
    &\leq 
    2c_k \sum_{q=0}^\infty \exp\bigg(-\frac{kp_k^2}{16n_k^2}\left(q+\frac{n_k^2
    \log^{7/4}(k)}{\sqrt{k}}\right)^2\bigg) \\
    &\leq 
    2c_k \sum_{q=0}^\infty \exp\bigg(-\frac{\log^{7/2}(k)}{\lambda^2}-
    \frac{\sqrt{k}p_k^2}{8}\log^{7/4}(k) q\bigg) \\
    &=O\bigg(\exp\Big(O(\log^3(k))-\frac{\log^{7/2}(k)}{\lambda^2}\Big)\bigg(1
    \vee \frac{n_k^2}{\sqrt{k} \log^{7/4}(k)} \bigg)\bigg) \\
    &=o(1).
\end{align*}

Thus, applying the total variation bound \eqref{eq:tvcond}, $\sup_{(n_k,p_k) \in
\mathcal{M}_k(\lambda)} \P(Z\in B_k^c) \to 0,$ and $\sup_{(n_k,p_k)\in
\mathcal{M}_k(\lambda)}\E_{n_k,p_k}[\Pi(B_k^c|\data)]\to 0$ (via
Theorem~\ref{result}), it remains to prove that 
\begin{align*}
    \sup_{(n_k,p_k)\in \mathcal{M}_k(\lambda)}
    \E_{n_k,p_k}\big[\TV(\Pi_1, Q_1)\big]
    \to 0 \qquad \text{as} \ k \to \infty
\end{align*}
with $\Pi_1(\cdot)=\Pi(\cdot \cap B_k |\data)/\Pi(B_k|\data)$ and
$Q_1(\cdot)=\P(Z\in \cdot \cap B_k)/\P(B_k).$ Because of $\sup_{(n_k,p_k)\in
\mathcal{M}_k(\lambda)} \P_{n_k,p_k}(\mathcal{X}_k^c) \to 0,$ the proof is
completed once we can show that
\begin{align}\label{eq.BvM_TV_to_show}
    \sup_{(n_k,p_k)\in \mathcal{M}_k(\lambda), \data \in \mathcal{X}_k}
    \TV(\Pi_1, Q_1)\to 0 \qquad \text{as} \ k \to \infty.
\end{align}

In the next step, we derive a sharper bound for the likelihood ratio by proving
that there exists an expression $V'$ that is independent of $n\in\mathbb{N}$,
such that
\begin{align}
    \log\left(\frac{L_{a,b}(n)}{L_{a,b}(n_k)}\right)
    = - \frac{kp_k^2}{4n_k^2} \big(n-\wh n\big)^2+V'+o(1)
    \label{eq.BvM_LR_to_show}
\end{align}
if $(n,\data)\in(B_k, \mathcal{X}_k)$.
As before, $L_{a,b}(n)$ denotes the beta-binomial likelihood. Recall that
$L_{a,b}(n)=R(a,b,n)$ for $n\geq M_k,$ with $R(a,b,n)$ as defined in equation
\eqref{E:helperR}. By assumption, the parameter $a$ is a non-negative integer.
Arguing as for \eqref{eq.R_ratio} and \eqref{fprime}, we find that 
\begin{align*}
    \frac{L_{a,b}(n)}{L_{a,b}(n_k)}
    &= \frac{R(a,b,n)}{R(a,b,n_k)}
    = \exp\left(k \int_{n_k}^n f'(m) \, \mathrm{d}m\right)
\end{align*}
with
\begin{equation*}
	f'(m)
	= \sum_{j=1}^{M_k}\frac{ T_j-U_j}{j} -\sum_{j=M_k+1}^{ S_k+a}\frac{U_j}{j},
\end{equation*}
$T_j = \tfrac 1k \sum_{i=1}^k (X_i)_j/(m)_j,$ and $U_j=( S_k
+a)_j/(km+a+b-1)_j.$ Since $n$ and $n_k$ are in $B_k,$ it is enough to determine
the leading order of $f'(m)$ uniformly over $m \in B_k.$

First observe that for $m\in B_k,$ we have $m\geq n_k/2$ and $(m)_j \geq
(n_k/2-j)_+^j.$ For $j\leq 2\log(k),$ we can conclude that $(m)_j\geq (n_k/4)^j$
due to $2\log(k)\leq n_k/4$. Recall the definition $t_j := \E_k[T_j] = (n_k)_j
p_k^j / (m)_j$.
Using $\data \in \mathcal{R}_k\cap \mathcal{T}_k$, we can apply the closed-form
formula for the geometric sum to obtain
\begin{align}\label{eq.BvM1}
\begin{split}
    \Big|\sum_{j=3}^{M_k} \frac{T_j- t_j}{j}\Big|
    &\leq \frac{l_k\log(k)}{\sqrt{k}}
    \sum_{3\leq j\leq 2\log(k)} \frac{\sqrt{(c_2j)^j}}{(m)_j} \\
    &\leq \frac{l_k\log(k)}{\sqrt{k}}
    \sum_{3\leq j\leq 2\log(k)}
    \left(\frac{4\sqrt{2c_2 \log(k)}}{n_k}\right)^j
    \\
    &= O\bigg(\frac{\log^3(k)}{n_k^3\sqrt{k}}
    \bigg),
\end{split}
\end{align}
for all sufficiently large $n_k$, where $l_k = o\left(\sqrt{\log k}\right)$ as
in equation \eqref{eq.events_proof_MR}. 

Let $u_j$ be as defined in \eqref{eq:smallujdef}. Applying Lemma
\ref{Uj-concentration} and noting $\data \in \mathcal{S}_k,$ we conclude
$n_k\geq 8e^2(3\lambda+a+1)$ for sufficiently large $k$ and thus
\begin{align}\label{eq.BvM2}
    \sum_{j=3}^\infty \frac{|U_j- u_j|}{j}
    \leq \sqrt{\frac{\lambda \log k}{k}} 
    \sum_{j=3}^\infty \left(\frac{2e^2(3\lambda+a+1)}{m}\right)^j
    = O\left(\frac 1{n_k^3}\sqrt{\frac{\log k}{k}}\right),
\end{align}
since $m\geq n_k/2$.

Using that for all sufficiently large $k,$ $S_k \leq 2k\lambda$ on the event
$\mathcal{S}_k,$ it follows that $|T_1-U_1|=O(1/kn_k)$ uniformly over
$\mathcal{M}_k(\lambda).$ Similarly, we also find that
\begin{align}\label{eq.BvM25}
    \left|U_2-\frac{S_k^2}{(km)^2}\right| = O\Big(\frac{1}{kn_k}\Big).
\end{align}
Introduce the sequence $b_j=(n_kp_k/m)^j.$ We find that $|b_j- u_j|=O(1/kn_k^2)$
for $2 \leq j\leq 5.$ We also have that $t_{j+1} = t_j(n_k-j)/(m-j)p_k$ and
$b_{j+1}=b_jn_kp_k/m.$ Therefore, 
\begin{align*}
    \big|t_{j+1}-b_{j+1}\big| 
    &\leq t_j p_k \Big| \frac{n_k-j}{m-j}-\frac{n_k}{m}\Big|
    +\big|t_j-b_j \big| \frac{n_kp_k}{m} \\
    &\leq t_j j p_k \left|\frac{m-n_k}{m(m-j)}\right|
    +\big|t_j-b_j \big| \frac{\lambda}{m}.
\end{align*}
This shows $|t_j-b_j|=O\left(|m-n_k|/n_k^{j+2}\right)$ for $j\leq 5$ and therefore
\begin{align}\label{eq.BvM4}
    \big|t_j-u_j \big|
    = O\left(\frac{|m-n_k|}{n_k^5}+ \frac{1}{kn_k^2}\right)\qquad \text{for}~j=3,4.
\end{align}
Using that $t_j=(n_k)_jp_k^j/(m)_j,$ it follows that 
\begin{align}\label{eq.BvM5}
    \sum_{6\leq j \leq 2\log(k)} \frac{t_j}{j}
    \leq 
    \sum_{6\leq j \leq 2\log(k)}
    \frac{\lambda^j}{\big(m-2\log(k)\big)^j}
    =O\left(\frac{1}{n_k^6}\right),
\end{align}
and similarly, we can conclude that
\begin{align}\label{eq.BvM6}
    \sum_{6\leq j \leq 2k\lambda+a} \frac{u_j}{j}=O\left(\frac{1}{n_k^6}\right).
\end{align}
Combining equations \eqref{eq.BvM1} to \eqref{eq.BvM6} shows that uniformly over
$\data \in \mathcal{X}_k$ and $m \in B_k,$
\begin{align*}
    f'(m) = \frac{\sum_{i=1}^k (X_i)_2}{2km(m-1)}- \frac{S_k^2}{2(km)^2}
    +O\bigg(\frac{\log^3(k)}{n_k^3\sqrt{k}}
    +\frac{|n_k-m|}{n_k^5}+ \frac{1}{n_k^6}+\frac{1}{kn_k}
    \bigg). 
\end{align*}
Recall that for any $n\in B_k,$ $|n-n_k|\leq n_k^2\log^{7/4}(k)/\sqrt{k}.$
Moreover, $k^{1/4}/\log(k) \leq n_k\leq \lambda \sqrt{k}/\log^6(k)$. It follows
that uniformly over $\data \in \mathcal{X}_k$ and $n \in B_k,$
\begin{align*}
    &k \int_{n_k}^n f'(m) \, \mathrm{d}m \\
    &\qquad=
    \frac 12 \sum_{i=1}^k (X_i)_2
    \bigg(\log\Big(1-\frac 1n\Big)-\log\Big(1-\frac 1{n_k}\Big)\bigg)
    +\frac{S_k^2}{2k}\Big(\frac{1}{n}-\frac{1}{n_k}\Big)+o(1) \\
    &\qquad= g(n)+V+o(1), \quad \text{where} \ \  
    g(n):=\frac 12 \sum_{i=1}^k (X_i)_2
    \log\Big(1-\frac 1n\Big)
    +\frac{S_k^2}{2kn}
\end{align*}
and $V$ is an expression that does not depend on $n.$ Noting that the derivative
of $\log(1-1/x)$ is $1/(x^2-x)$, it can be checked that $g'(\wh n)=0$, where
$\wh n$ is the estimator defined in \eqref{eq.whn_def}. A second order Taylor
expansion shows that there exists a $\xi$ between $n$ and $\wh n$ such that
\begin{align}\label{eq.g_T_exp}
    g(n)=g(\wh n)+ \frac {g''(\xi)}2 \big(n-\wh n\big)^2.
\end{align}
By \eqref{eq.nhatn_ub}, it follows that $\wh n \in B_k$ and hence $\xi\in B_k.$
Observe that $S_k^2/(k \wh n^3)=\sum_{i=1}^k (X_i)_2/\wh n^2(\wh n-1)$ and let
$d_k \leq 2n_k^2\log^{7/4}(k)/\sqrt{k}$ denote the diameter of $B_k$. Combined
with the definition of the event $\mathcal{T}_k$ in \eqref{eq.events_proof_MR},
$(m)_2t_2=n_k(n_k-1)p_k^2,$ and $1/\lambda \leq n_kp_k \leq \lambda$, we find
that uniformly over $n \in B_k,$
\begin{align*}
    g''(n)&= \frac{\sum_{i=1}^k (X_i)_2 (1-2n)}{2 (n^2-n)^2}+ \frac{S_k^2}{kn^3} \\
    &= \frac{\sum_{i=1}^k (X_i)_2}{2n^2}
    \left[\frac{1-2n}{(n-1)^2}+\frac{2\wh n}{n(\wh n-1)}\right]\\
    &= \frac{\sum_{i=1}^k (X_i)_2\big[(2n-1-\wh n)(n-\wh n)-\wh n^2+\wh
       n\big]}{2n^3(n-1)^2(\wh n-1)}\\
    &=-\frac{\sum_{i=1}^k (X_i)_2\wh n^2}{2n^3(n-1)^2(\wh n-1)}
    +O\left(\frac{k}{n_k^5}\big(d_k+1\big)\right) \\
    &=-\frac{\sum_{i=1}^k (X_i)_2}{2n_k^4}
    +O\left(\frac{k}{n_k^5}\big(d_k+1\big)\right) \\
    &= - \frac{k n_k(n_k-1)p_k^2}{2n_k^4}
    +O\left(\frac{k}{n_k^5}\big(d_k+1\big)+\frac{\sqrt{k}\log k}{n_k^4 }\right) \\
    &= -\frac{kp_k^2}{2n_k^2}
       +O\left(\frac{k}{n_k^5}\big(d_k+1\big)+\frac{\sqrt{k}\log k}{n_k^4 }\right).
\end{align*}
Recall that $k^{1/4}/\log(k) \leq n_k \leq \lambda \sqrt{k}/\log^6(k).$
Inserting the expression in \eqref{eq.g_T_exp} and using that $|n-\wh n|\leq d_k
\leq 2 n_k^2 \log^{7/4}(k)/\sqrt{k}$ on $B_k$ finally yields
\begin{align*}
    g(n)=g(\wh n)-\frac{kp_k^2}{4n_k^2} \big(n-\wh n\big)^2+o(1).
\end{align*}
Therefore, we can also write 
\begin{align*}
    k \int_{n_k}^n f'(m) \, \mathrm{d}m
    = -\frac{kp_k^2}{4n_k^2} \big(n-\wh n\big)^2+V'+o(1),
\end{align*}
where $V'$ denotes an expression that is independent of $n.$ This proves
\eqref{eq.BvM_LR_to_show}.

Recall that $\Pi(n)\propto n^{-\alpha}$ for some $\alpha>1.$ Because of $n_k
\leq \lambda \sqrt{k}/\log^6(k),$ we have $|n/n_k-1| =o(1)$ and
\begin{align*}
    \log\left(\frac{\Pi(n)}{\Pi(n_k)}\right)=-\alpha
    \log\left(\frac{n}{n_k}\right) =o(1).
\end{align*}
for $n \in B_k$.  Applying Lemma \ref{lem.TV_comp} with 
\begin{equation*}
    q_n=\exp\bigg(-\frac{kp_k^2}{4n_k^2} \big(n-\wh
    n\big)^2+V'\bigg)\mathbf{1}(n \in B_k)
\end{equation*}
then shows that $\TV(\Pi_1,Q_1)\to 0$ uniformly over $(n_k,p_k) \in
\mathcal{M}_k(\lambda)$ and $\data \in \mathcal{X}_k.$ Therefore, we have
established \eqref{eq.BvM_TV_to_show} and the proof is complete.
\end{proof}

\begin{lem}[Lemma E.3 in \cite{2018arXiv180904140R}]\label{lem.TV_scaling}
Consider two discrete distributions $P,Q.$ If for some $\alpha>0$ and for some
$\delta \in (0,1),$ $\sum_i|\alpha P(X=i)-Q(X=i)|\leq \delta,$ then
$$\TV(P,Q)\leq \frac{\delta}{1-\delta}.$$
\end{lem}

\begin{lem}\label{lem.TV_two_discretization_bd}
There exists a universal constant $K$ such that for all $\mu$ and all
sufficiently large $\sigma,$
\begin{align}
    \TV\big(\mathcal{N}_D(\mu,\sigma^2),\mathcal{N}_d(\mu,\sigma^2)\big)
    \leq \frac{K}{\sigma}.
\end{align}
\end{lem}

\begin{proof}
We will apply Lemma \ref{lem.TV_scaling}. Recall that $1-x\leq e^{-x}$ and thus
$|1-e^{-x}|=|x|$ for all $x>0.$ By the mean value theorem, there exists for all
$j$ a $\xi_j\in [j-1/2,j+1/2],$ such that 
\begin{align*}
    &\Big|e^{-\frac{(j-\mu)^2}{2\sigma^2}}- \int_{j-1/2}^{j+1/2} 
    e^{-\frac{(x-\mu)^2}{2\sigma^2}} \, \mathrm{d}x \Big| \\
    &= 
    \Big|e^{-\frac{(j-\mu)^2}{2\sigma^2}}- e^{-\frac{(\xi_j-\mu)^2}{2\sigma^2}}
    \Big| \\
    &=\max\Big( e^{-\frac{(j-\mu)^2}{2\sigma^2}},
    e^{-\frac{(\xi_j-\mu)^2}{2\sigma^2}}\Big)
    \Big| \frac{(j-\mu)^2}{2\sigma^2}
    - \frac{(\xi_j-\mu)^2}{2\sigma^2}\Big| \\
    &\leq \frac{|j-\mu|+1/4}{2\sigma^2}
    \max\Big( e^{-\frac{(j-\mu)^2}{2\sigma^2}},
    e^{-\frac{(\xi_j-\mu)^2}{2\sigma^2}}\Big).
\end{align*}
Let $c=c(\mu,\sigma^2)$ be the normalizing factor of the
$\mathcal{N}_d(\mu,\sigma^2),$ distribution, that is,
$P(X=j)=c\exp(-(j-\mu)^2/2\sigma^2)$ for all integer $j.$ With
$\alpha=1/(c\sqrt{2\pi \sigma^2})$ and the bound from the previous display, some
tedious calculations show that there exists a universal constant $C$ such that
\begin{align*}
    &\sum_j \big|\alpha P(X=j) - \frac{1}{\sqrt{2\pi \sigma^2}} \int_{j-1/2}^{j+1/2} 
    e^{-\frac{(x-\mu)^2}{2\sigma^2}} \, \mathrm{d}x \Big| \\
    &\leq \sum_j \frac{|j-\mu|+1/4}{\sqrt{8\pi}\sigma^3}
    \max\Big( e^{-\frac{(j-\mu)^2}{2\sigma^2}},
    e^{-\frac{(\xi_j-\mu)^2}{2\sigma^2}}\Big) \\
    &\leq \frac{C}{\sigma}.
\end{align*}
Applying Lemma \ref{lem.TV_scaling} yields the result.
\end{proof}

\clearpage

\section{Additional Simulations}\label{app:simulations}
The following graphs are supplements to Figure~\ref{fig:posterior-hist} and
\ref{fig:comparison} of the main article. They show simulations with different
parameter configurations, which have been omitted in the main text.

\vspace{1cm}
\begin{figure}[h!]
  \includegraphics[width=\linewidth]{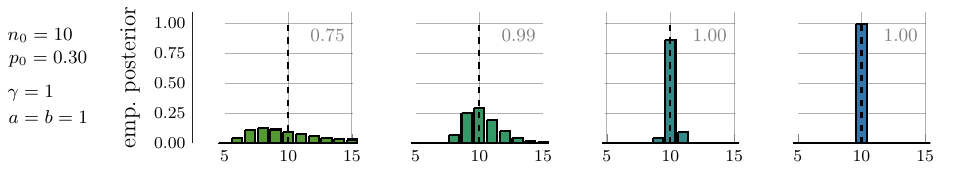}
  \includegraphics[width=\linewidth]{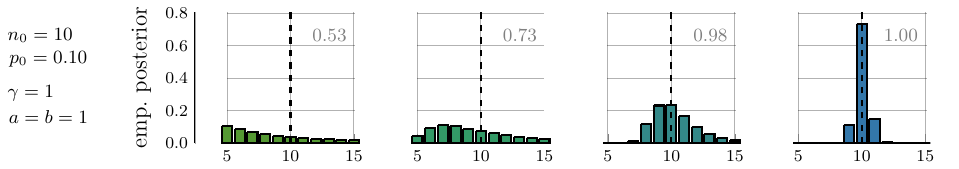}
  \includegraphics[width=\linewidth]{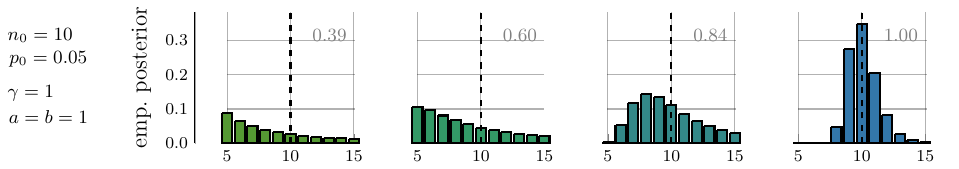}
  \vspace{0.1cm}

  \includegraphics[width=\linewidth]{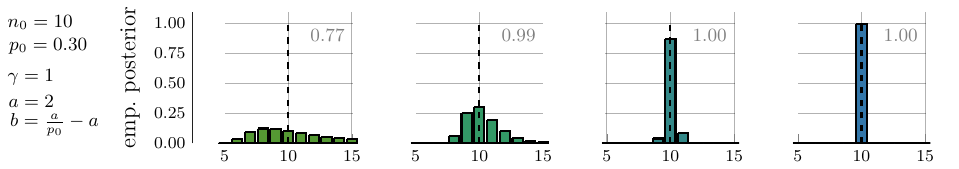}
  \includegraphics[width=\linewidth]{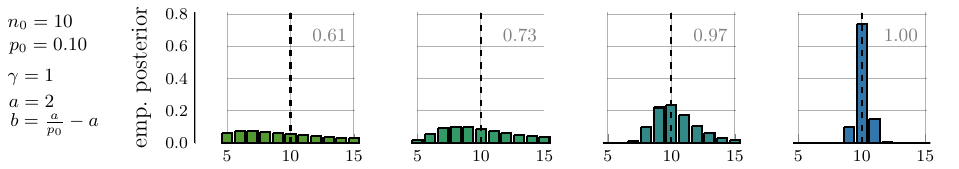}
  \includegraphics[width=\linewidth]{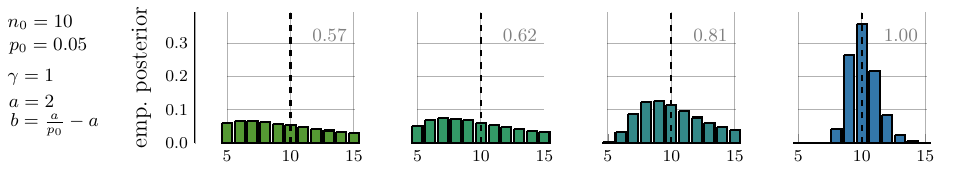}
  \caption{Posterior distributions with true parameter $n_0 = 10$ (supplement to
  Figure~\ref{fig:posterior-hist}).}
  \label{fig:app:posterior_10}
\end{figure}

\begin{figure}
  \includegraphics[width=\linewidth]{posterior/noinfo_20_0.30}
  \includegraphics[width=\linewidth]{posterior/noinfo_20_0.10}
  \includegraphics[width=\linewidth]{posterior/noinfo_20_0.05}
  \vspace{0.1cm}

  \includegraphics[width=\linewidth]{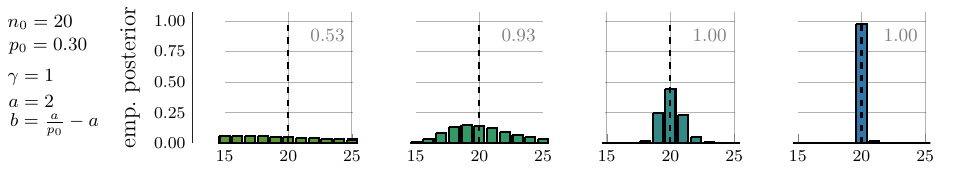}
  \includegraphics[width=\linewidth]{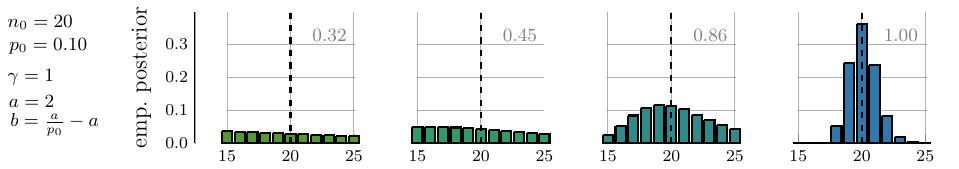}
  \includegraphics[width=\linewidth]{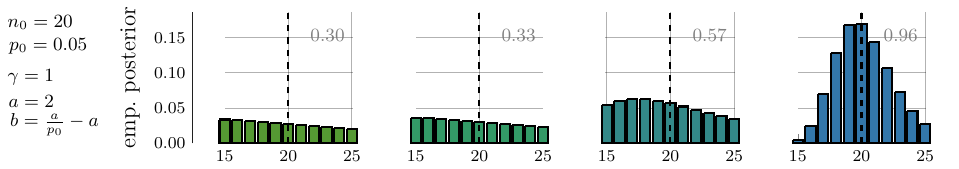}
  \caption{Posterior distributions with true parameter $n_0 = 20$ (supplement to
  Figure~\ref{fig:posterior-hist}).}
  \label{fig:app:posterior_20}
\end{figure}

\begin{figure}
  \includegraphics[width=\linewidth]{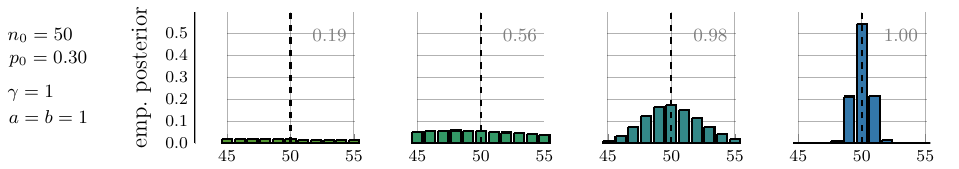}
  \includegraphics[width=\linewidth]{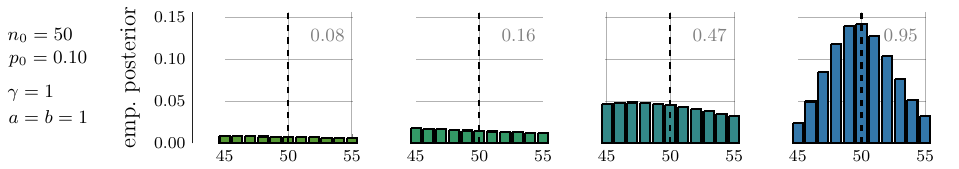}
  \includegraphics[width=\linewidth]{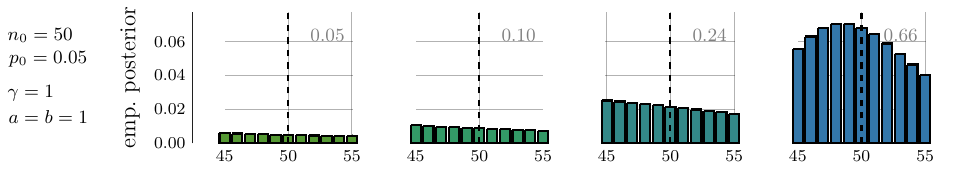}
  \vspace{0.1cm}

  \includegraphics[width=\linewidth]{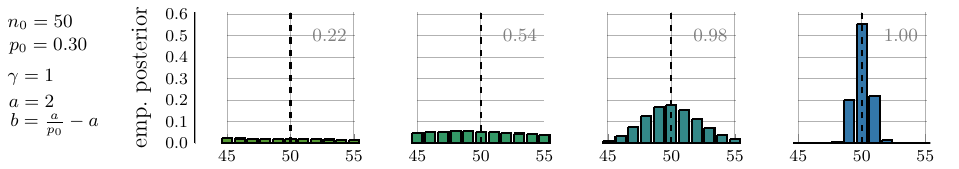}
  \includegraphics[width=\linewidth]{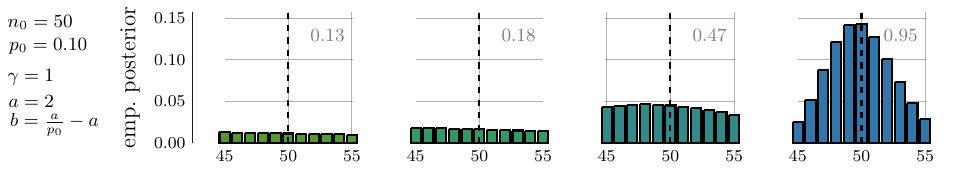}
  \includegraphics[width=\linewidth]{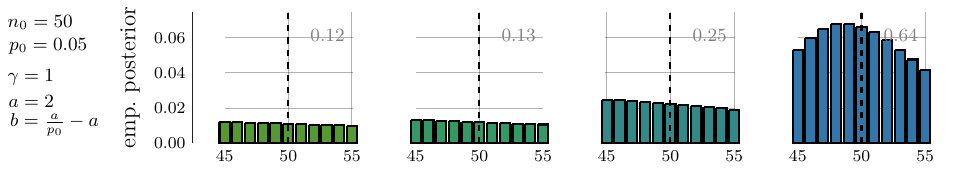}
  \caption{Posterior distributions with true parameter $n_0 = 50$ (supplement to
  Figure~\ref{fig:posterior-hist}).}
  \label{fig:app:posterior_50}
\end{figure}

\begin{figure}
  {\scriptsize $a = b = 1$}
  \includegraphics[width=\linewidth]{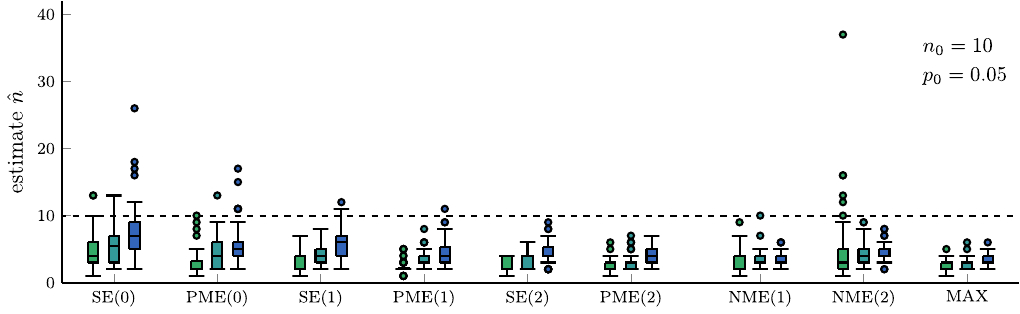}
  \includegraphics[width=\linewidth]{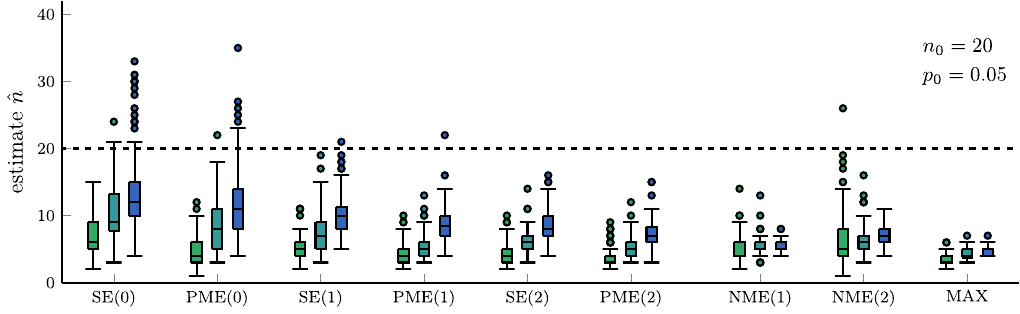}
  \includegraphics[width=\linewidth]{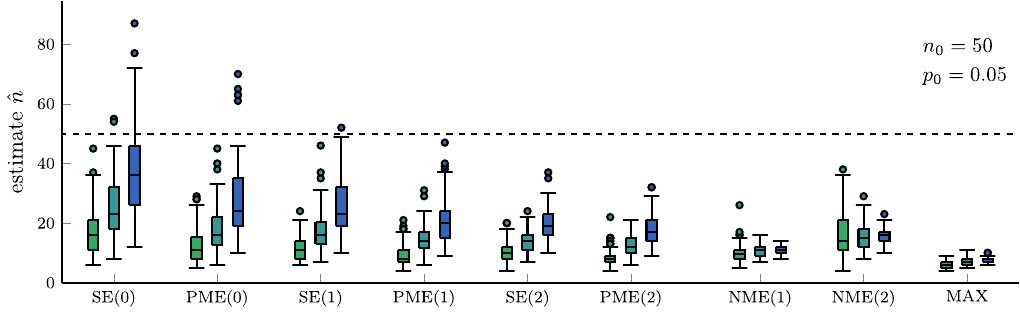}
  \includegraphics[width=\linewidth]{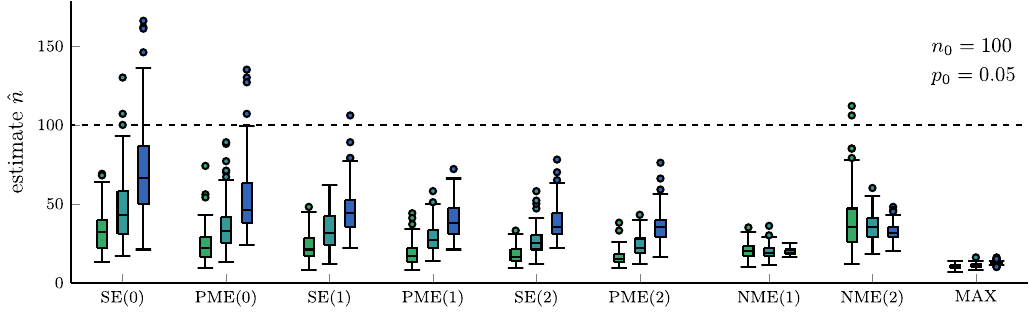}
  \includegraphics[width=\linewidth]{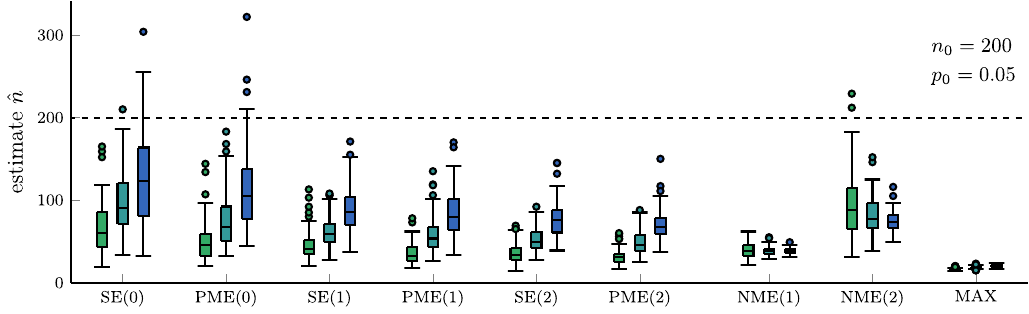}
  \caption{Box plots for $p_0 = 0.05$ and uniform prior on $p$ (supplement to
  Figure~\ref{fig:comparison}).}
  \label{fig:app:comparison_noinfo_0.05}
\end{figure}

\begin{figure}
  {\scriptsize $a = b = 1$}
  \includegraphics[width=\linewidth]{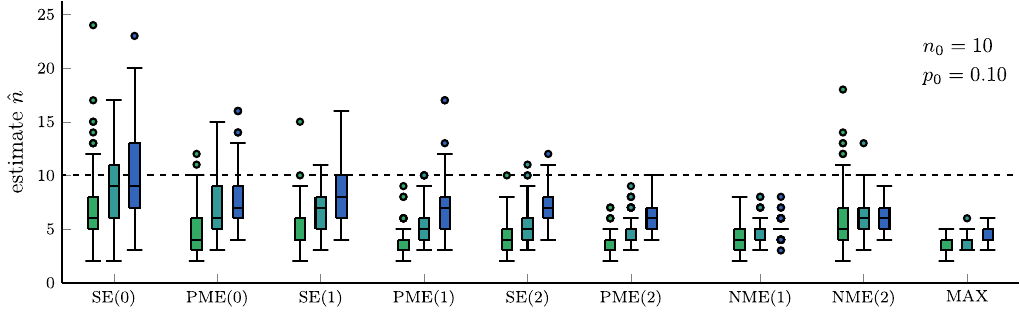}
  \includegraphics[width=\linewidth]{comparison/noinfo_20_0.10}
  \includegraphics[width=\linewidth]{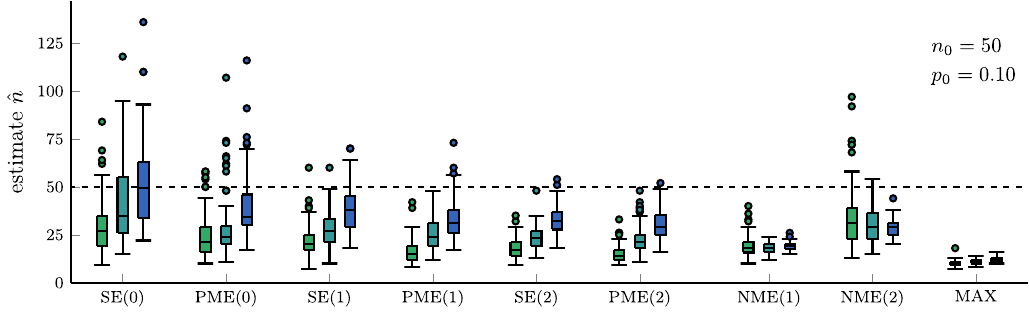}
  \includegraphics[width=\linewidth]{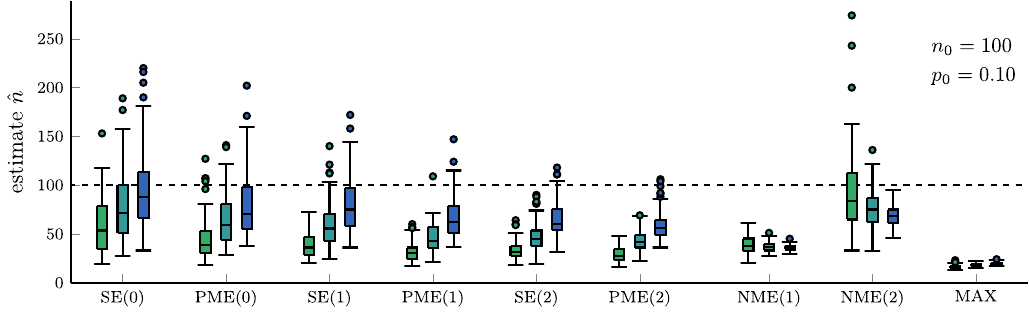}
  \includegraphics[width=\linewidth]{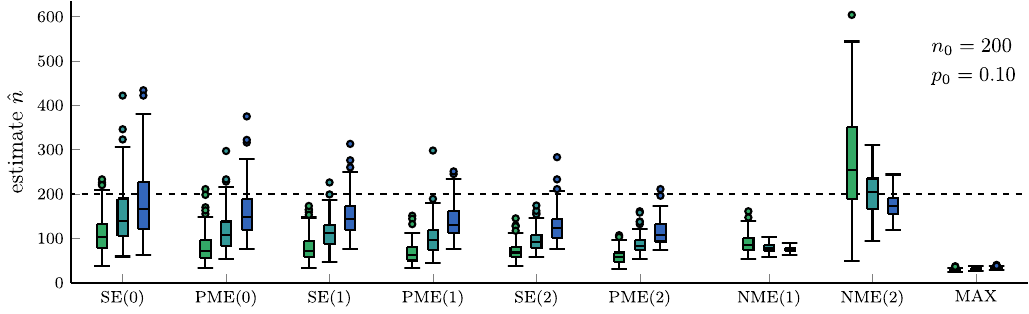}
  \caption{Box plots for $p_0 = 0.1$ and uniform prior on $p$ (supplement to
  Figure~\ref{fig:comparison}).}
  \label{fig:app:comparison_noinfo_0.1}
\end{figure}

\begin{figure}
  {\scriptsize $a = b = 1$}
  \includegraphics[width=\linewidth]{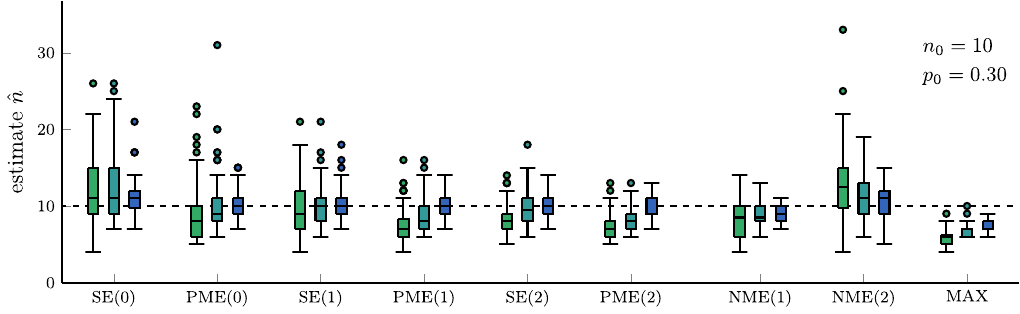}
  \includegraphics[width=\linewidth]{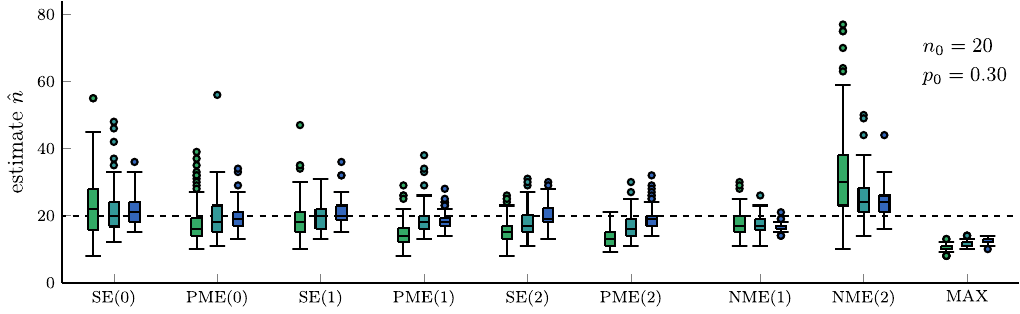}
  \includegraphics[width=\linewidth]{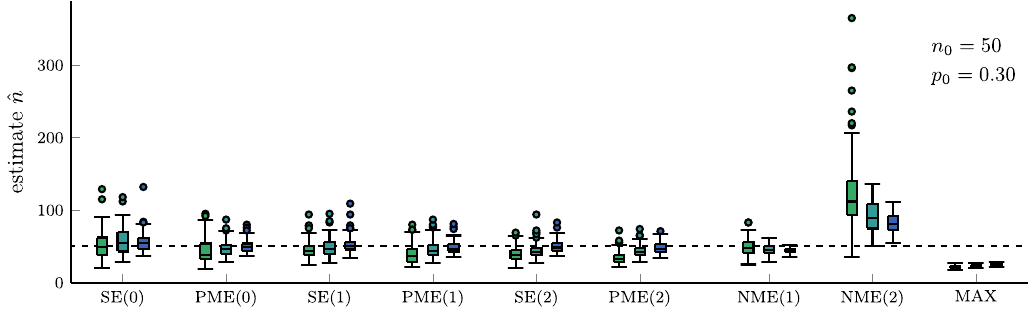}
  \includegraphics[width=\linewidth]{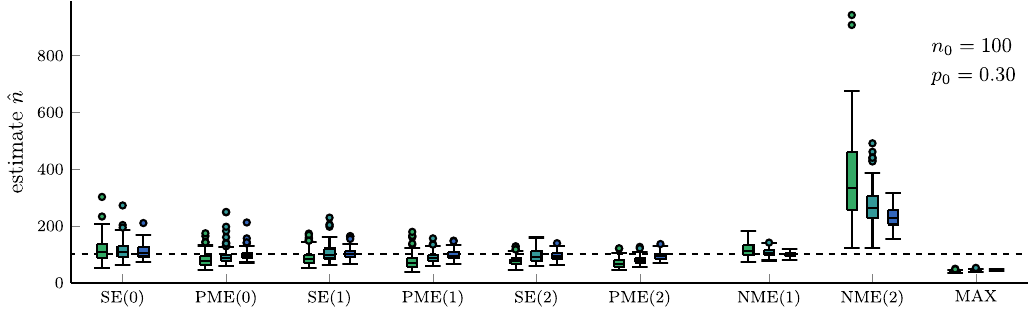}
  \includegraphics[width=\linewidth]{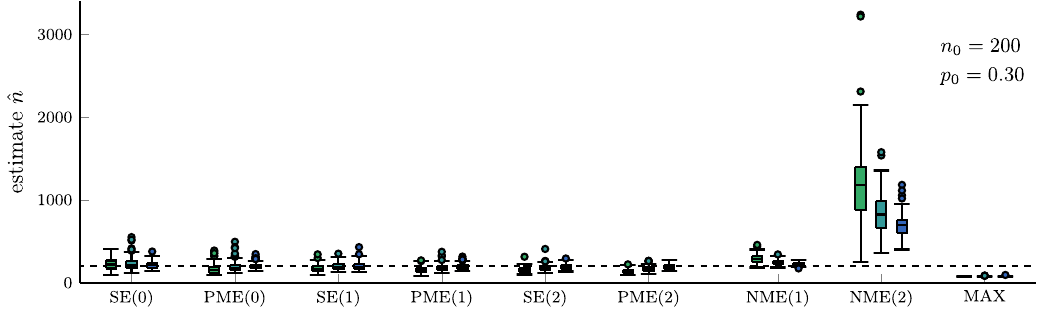}
  \caption{Box plots for $p_0 = 0.3$ and uniform prior on $p$ (supplement to
  Figure~\ref{fig:comparison}).}
  \label{fig:app:comparison_noinfo_0.3}
\end{figure}

\begin{figure}
  {\scriptsize $a = 2$, $b = a/p_0 - a$}
  \includegraphics[width=\linewidth]{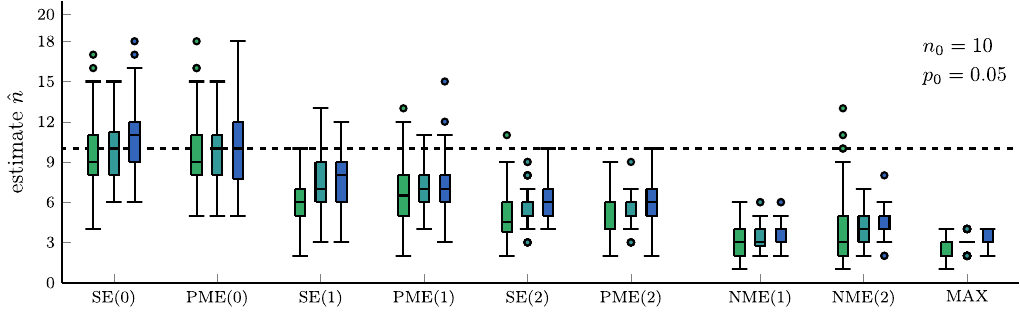}
  \includegraphics[width=\linewidth]{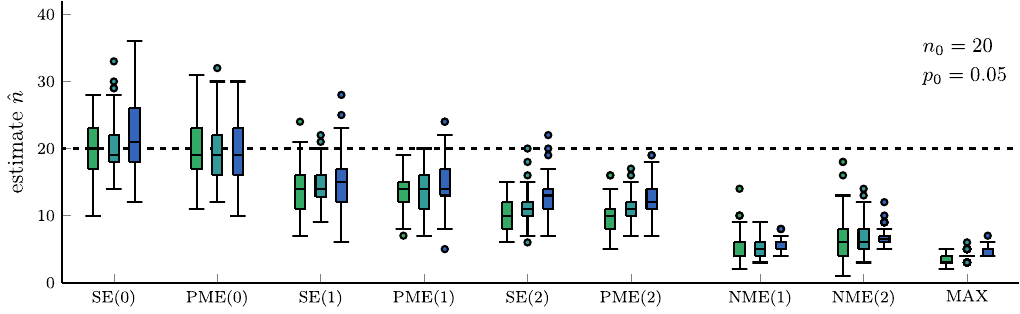}
  \includegraphics[width=\linewidth]{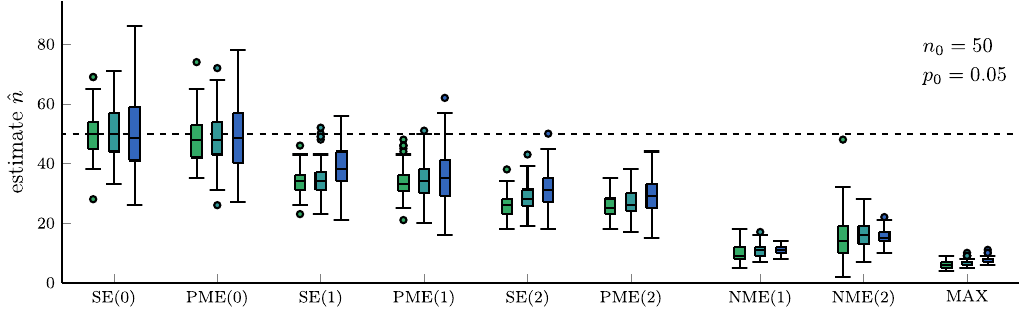}
  \includegraphics[width=\linewidth]{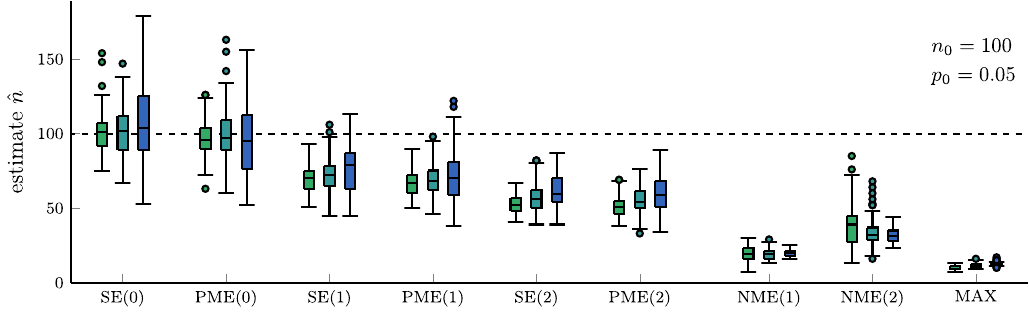}
  \includegraphics[width=\linewidth]{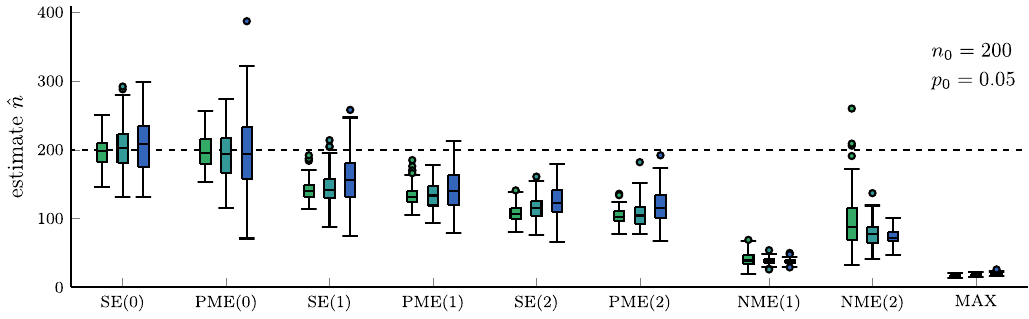}
  \caption{Box plots for $p_0 = 0.05$, $a=2$ and $b=2/p_0-2$ (supplement to
  Figure~\ref{fig:comparison}).}
\end{figure}

\begin{figure}
  {\scriptsize $a = 2$, $b = a/p_0 - a$}
  \includegraphics[width=\linewidth]{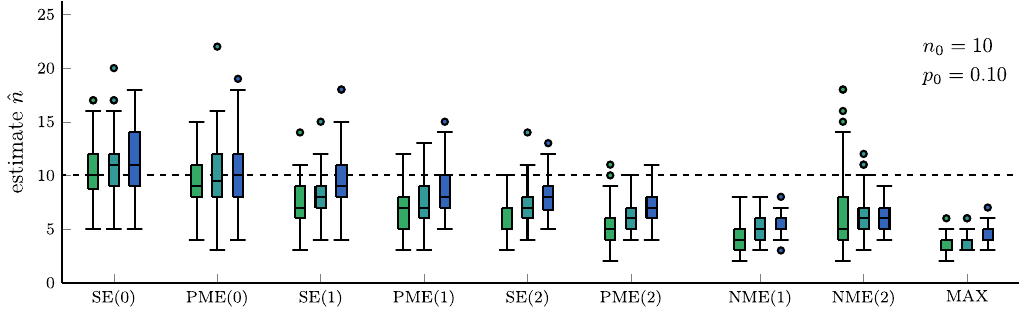}
  \includegraphics[width=\linewidth]{comparison/info_20_0.10}
  \includegraphics[width=\linewidth]{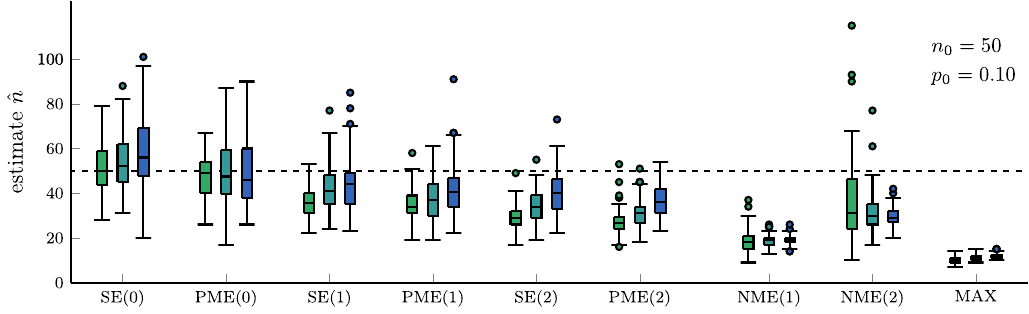}
  \includegraphics[width=\linewidth]{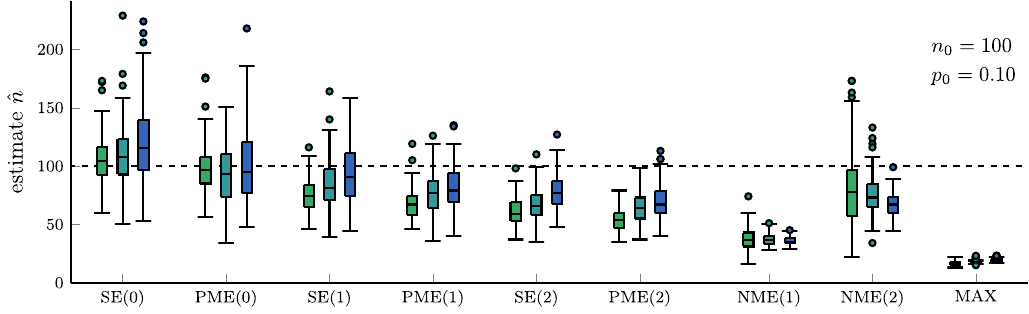}
  \includegraphics[width=\linewidth]{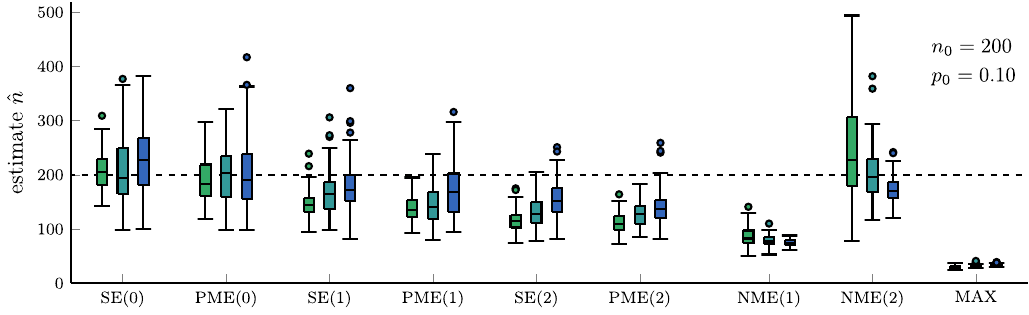}
  \caption{Box plots for $p_0 = 0.1$, $a=2$ and $b=2/p_0-2$ (supplement to
  Figure~\ref{fig:comparison}).}
\end{figure}

\begin{figure}
  {\scriptsize $a = 2$, $b = a/p_0 - a$}
  \includegraphics[width=\linewidth]{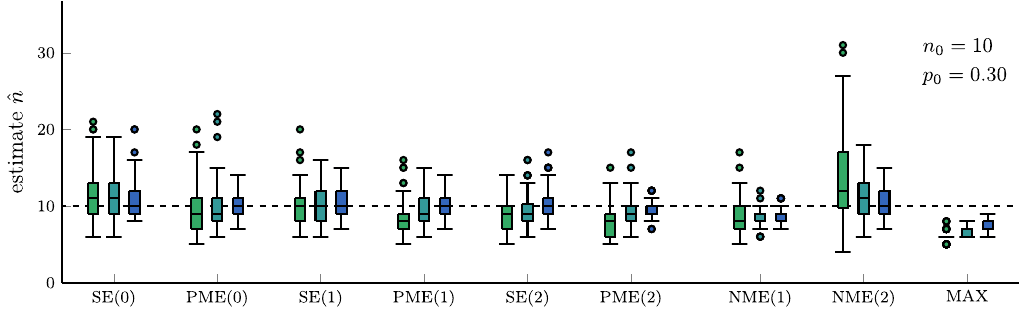}
  \includegraphics[width=\linewidth]{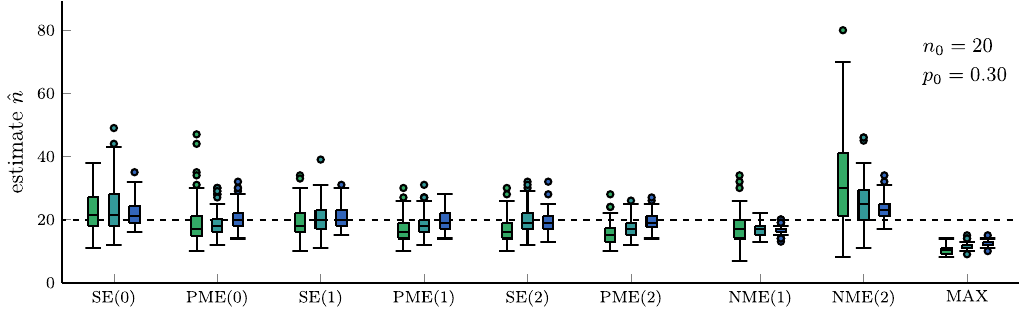}
  \includegraphics[width=\linewidth]{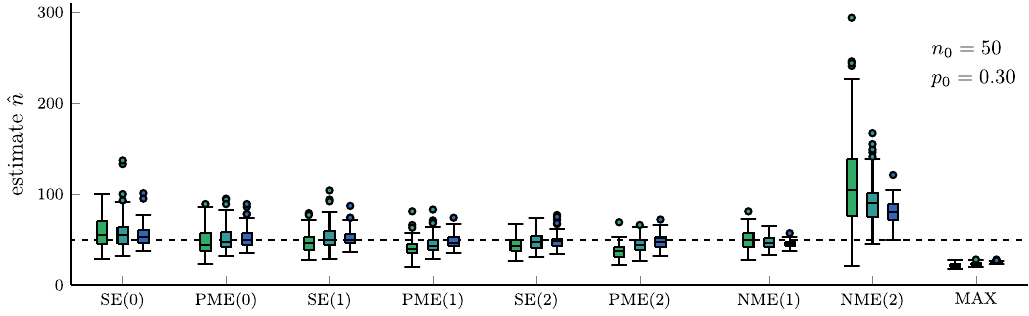}
  \includegraphics[width=\linewidth]{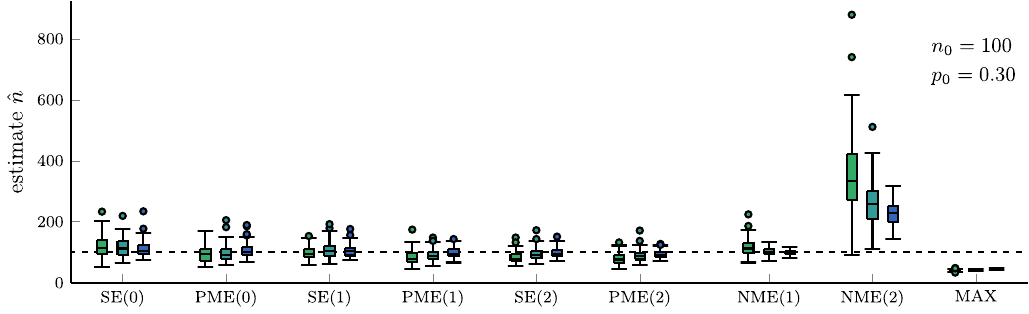}
  \includegraphics[width=\linewidth]{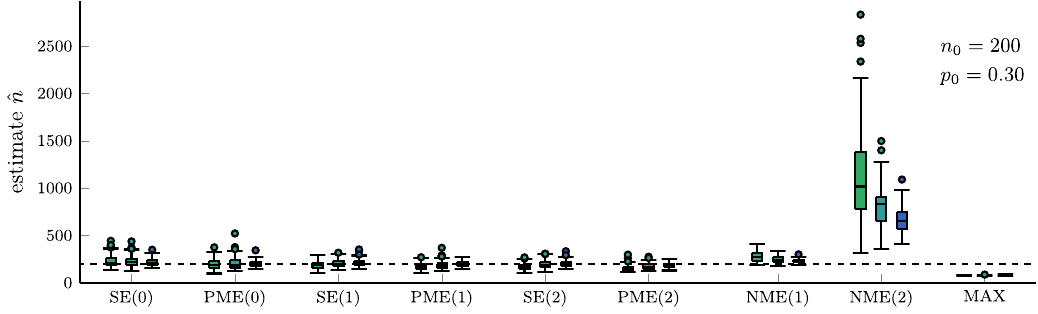}
  \caption{Box plots for $p_0 = 0.3$, $a=2$ and $b=2/p_0-2$ (supplement to
  Figure~\ref{fig:comparison}).}
\end{figure}

\end{document}